\documentclass[a4paper]{amsart}

\usepackage[utf8]{inputenc}
\usepackage[english]{babel}
\usepackage[dvipsnames]{xcolor}
\usepackage{tikz-cd}
\usetikzlibrary{positioning,decorations.text,quotes}
\usetikzlibrary{arrows,babel}
\usepackage{amsmath,amsfonts,amssymb}
\usepackage{mathtools}
\usepackage{float}
\usepackage{amsthm}

\usepackage{stackrel}
\usepackage{pgfplots}
\usepackage{graphicx}

\usepackage{bbm}
\usepackage{amssymb}
\usepackage{geometry,enumerate}
\usepackage{graphicx}
\usepackage{psfrag}
\usepackage{amscd}
\usepackage{comment}
\usepackage[all]{xy}
\usepackage{url}

\usepackage{fullpage}

\usepackage[dvipsnames]{xcolor}

\usepackage[backref=page]{hyperref}

\hypersetup{
 colorlinks,
 citecolor=Green,
 linkcolor=blue,
 urlcolor=Blue}

\usepackage{booktabs}
\usepackage{color}
\usepackage{todonotes}

\usepackage{braket}

\usepackage{stmaryrd} 
\usepackage{mathabx}
\usepackage{mathbbol}
\usepackage{xspace}
\usepackage{enumerate}
\usepackage{caption}
\usepackage{subcaption}

 \usepackage{cleveref}
\crefformat{section}{#2#1#3} 
\crefformat{subsection}{#2#1#3}

\newcommand{\ZZ}{\mathbb{Z}}

\newcommand{\NN}{\mathbb{N}}
\newcommand{\RR}{\mathbb{R}}

\newcommand{\PP}{\mathbb{P}}
\newcommand{\CC}{\mathbb{C}}

\newcommand{\OO}{\mathbb{O}}

\newcommand{\EE}{\mathbb{E}}
\renewcommand{\AA}{\mathbb{A}}
\newcommand{\ee}{\mathbb{e}}

\newcommand{\Pol}{\operatorname{Pol}}

\newcommand{\A}{\mathcal{A}}

\renewcommand{\O}{\mathcal{O}}

\renewcommand{\L}{\mathcal{L}}
\newcommand{\T}{\mathcal{T}}
\newcommand{\I}{\mathcal{I}}
\newcommand{\F}{\mathcal{F}}
\newcommand{\D}{\mathcal{D}}
\newcommand{\X}{\mathcal{X}}

\newcommand{\J}{\mathcal{J}}

\renewcommand{\S}{\mathcal{S}}
\renewcommand{\SS}{\mathcal{SS}}

\renewcommand{\P}{\mathcal{P}}

\newcommand{\un}{\underline}
\newcommand{\ov}{\overline}
\newcommand{\wh}{\widehat}
\newcommand{\wt}{\widetilde}

\newcommand{\m}{\mathfrak{m}}
\newcommand{\s}{\mathfrak{s}}

\newcommand{\n}{\mathfrak{n}}

\newcommand{\val}{\operatorname{val}}

\newcommand{\con}{\operatorname{con}}
\newcommand{\supp}{\operatorname{supp}}
\newcommand{\sm}{\operatorname{sm}}

\newcommand{\Prin}{\operatorname{Prin}}

\newcommand{\VStab}{\operatorname{VStab}}
\newcommand{\Deg}{\operatorname{Deg}}
\newcommand{\BCon}{\operatorname{BCon}}
\newcommand{\Con}{\operatorname{Con}}
\newcommand{\Div}{\operatorname{Div}}
\newcommand{\Pic}{\operatorname{Pic}}
\renewcommand{\Im}{\operatorname{Im}}
\newcommand{\End}{\operatorname{End}}
\newcommand{\Aut}{\operatorname{Aut}}
\newcommand{\Spec}{\operatorname{Spec}}
\newcommand{\Gr}{\operatorname{Gr}}
\newcommand{\Tors}{\operatorname{Tors}}
\renewcommand{\div}{\operatorname{div}}

\newcommand{\Xsing}{\operatorname{X_{\rm{sing}}}}
\newcommand{\NF}{\operatorname{NF}}
\newcommand{\TF}{\operatorname{TF}}
\newcommand{\Simp}{\operatorname{Simp}}
\newcommand{\PIC}{\operatorname{PIC}}
\newcommand{\Gm}{\operatorname{\mathbb{G}_m}}

\newcommand{\BD}{\operatorname{BD}}

\newcommand{\Z}{\mathbb{Z}}
\newcommand{\R}{\mathbb{R}}

\pgfplotsset{compat=1.17}

\definecolor{LivGreen}{cmyk}{68 0 100 0}

\newtheorem{theorem}{Theorem}[section]
\newtheorem{corollary}[theorem]{Corollary}
\newtheorem{lemma}[theorem]{Lemma}
\newtheorem{fact}[theorem]{Fact}
\newtheorem{proposition}[theorem]{Proposition}
\newtheorem{proposition-definition}[theorem]{Proposition-Definition}
\newtheorem{lemma-definition}[theorem]{Lemma-Definition}

\newtheorem{theoremalpha}{Theorem}

\newtheorem*{fact*}{Fact}

\theoremstyle{definition}
\newtheorem{definition}[theorem]{Definition}
\newtheorem{example}[theorem]{Example}
\newtheorem{remark}[theorem]{Remark}

\newtheorem{notation}[theorem]{Notation}

\numberwithin{equation}{section}

\newenvironment{sis}{\left\{\begin{aligned}}{\end{aligned}\right.}

\begin{document}

\title{A complete theory of  smoothable compactified Jacobians of nodal curves}

\author{Marco Fava}
\address{Marco Fava, Department of Mathematical Sciences, University of Liverpool, Liverpool, L69 7ZL, United Kingdom}
\email{marco.fava@liverpool.ac.uk}

\author{Nicola Pagani}
\address{Nicola Pagani, Department of Mathematical Sciences, University of Liverpool, Liverpool, L69 7ZL, United Kingdom; Dipartimento di Matematica,  Universit\`a di Bologna, Piazza di Porta S. Donato, 5, 40126.}
\email{pagani@liv.ac.uk}

\email{nicolatito.pagani@unibo.it}
\urladdr{http://pcwww.liv.ac.uk/~pagani/}

\author{Filippo Viviani}
\address{Filippo Viviani, Dipartimento di Matematica, Universit\`a di Roma ``Tor Vergata'', Via della Ricerca Scientifica 1, 00133 Roma, Italy}
\email{viviani@mat.uniroma2.it}

\keywords{Compactified Jacobians, nodal curves.}

\subjclass[msc2000]{{14H10}, {14H40}, {14D22}.}

		\begin{abstract}
This paper is the second of a series of three. 
In the first paper \cite{FPV1}, we introduced a new class of compactified Jacobians (that we baptized compactified Jacobians of \emph{vine type}, or simply \emph{V-compactified Jacobians}) for families of reduced curves.

Here we focus on the case of a \emph{nodal} curve over an algebraically closed field. We characterize the V-compactified Jacobians as those  that can arise as limits of Jacobians under every one-parameter smoothing of the nodal curve. Furthermore, we also characterize the V-compactified Jacobians as those that have the smallest possible number of irreducible components with a given fixed degeneracy subset. These results extends previous works by Pagani-Tommasi \cite{pagani2023stability} and Viviani \cite{viviani2023new}, where a similar classification result was achieved for \emph{fine} compactified Jacobians.

In the follow-up paper \cite{FPV3}, we will apply the results obtained here and in \cite{FPV1}, to classify the compactified \emph{universal} Jacobians over the moduli of pointed stable curves.

	\end{abstract} 

\maketitle
	\bigskip
	
	\tableofcontents

\section{Introduction}

This paper is the second of a series of three. In the first paper \cite{FPV1}, we introduced an abstract notion of a (modular) relative compactified Jacobian for a family of reduced curves. The aim of this paper, which completes the work started in \cite{pagani2023stability} and \cite{viviani2023new}, is to address the following natural

\vspace{0.1cm}

\textbf{Question:} Suppose that a smooth (projective, connected) curve  degenerates to a nodal (projective, connected) curve  in a $1$-parameter family, how can we degenerate "modularly" the Jacobian of the generic smooth curve?

\vspace{0.2cm}

Let us now clarify what  we mean by a \emph{modular degeneration} in the above question. Given a nodal curve $X$ over an algebraically closed field $k=\ov k$, a \textbf{compactified Jacobian stack} $\ov \J_X^d$ of degree $d$ of $X$ (see Definition~\ref{D:compJac})  is an open  substack  of the stack $\TF_X^d$ of rank-$1$, torsion-free sheaves of degree $d$ on $X$, which admits a proper good moduli space $\ov J_X^d$. The latter is called the associated \textbf{compactified Jacobian space}\footnote{The expression \emph{compactified Jacobian} refers to the fact that the smooth locus (both of the stack and of the space) is a disjoint union of several copies of the generalized Jacobian of $X$.}. We address the reader to the introduction of \cite{FPV1} and the references therein for some history, general motivation, and methods on the problem of studying the compactified Jacobian of singular curves.

In light of this definition, we can restate the above question as the task of characterizing all the degree-$d$ compactified Jacobians of the central nodal curve that can arise as the limit of the degree-$d$ Jacobian of some generic smooth curve.

\vspace{0.1cm}

In the special case where one additionally assumes that the rigidification of the compactified Jacobian stack $\ov \J_X^d$ is already proper, we say that $\ov \J_X^d$  is a \emph{fine compactified Jacobian}. For fine compactified Jacobians, the above question was solved by Pagani-Tommasi in  \cite{pagani2023stability} and by Viviani in \cite{viviani2023new}. The aim of this paper is to answer the same question more generally, i.e. for compactified Jacobians that are not necessarily fine.
\vspace{0.1cm}

Note that non-fine compactified Jacobians appear naturally. For example, if $X$ is stable, then slope-stability with respect to the ample canonical polarization $\omega_X$ (as defined by Simpson \cite{simpson}) yields  `canonical'  compactified Jacobians of any given degree $d$, which are fine if and only if $\gcd(d+1-g(X), 2g(X)-2)=1$, and which can acquire a different `degree of degeneracy' for different $d$'s,  the maximally degenerate ones appearing when $d$ equals $g(X)-1$ modulo $2g(X)-2$. These canonical compactified Jacobians are the fibers of Caporaso's compactification of the universal Jacobian over the moduli space $\ov{\mathcal{M}}_g$ of stable curves of genus $g$ (see \cite{caporaso}).

\vspace{0.1cm}

The main novelty of our work (and of the previous works  \cite{pagani2023stability} and \cite{viviani2023new}) compared to the existing literature is that we aim to classify \emph{all}  possible solutions to the above question.  It turns out that our complete classification produces more than the \emph{classical} compactified Jacobians: this was already observed in \cite{viviani2023new} in the  case of fine compactified Jacobians.

\subsection*{The results}
We first recall from \cite{FPV1} the definition of V-stability condition and of V-compactified Jacobian for a reduced curve. 

Given a connected reduced curve $X$ over $k=\ov k$, we denote by $\BCon(X)$ the set of biconnected subcurves of $X$, i.e. non trivial-subcurves $\emptyset \subsetneq Y\subsetneq X$ such that $Y$ and its complementary subcurve $Y^\mathsf{c}:=\ov{X\setminus Y}$ are connected.

A \emph{stability condition of vine type} (or simply a \textbf{V-stability condition)}  of characteristic $\chi=|\s|\in \Z$ on $X$ is a function  (see Definition \ref{D:VcJ}) 
$$
\begin{aligned}
\s: \BCon(X) & \longrightarrow \Z\\
Y & \mapsto \s_Y
\end{aligned}
$$
satisfying the following properties:
\begin{enumerate}
\item \label{E:condi1} for any $Y\in \BCon(X)$, we have 
\begin{equation*}
\s_Y+\s_{Y^\mathsf{c}}-\chi\in \{0,1\}.
\end{equation*}
A subcurve $Y\in \BCon(X)$ is said to be \emph{$\s$-degenerate} if $\s_Y+\s_{Y^\mathsf{c}}-\chi=0$, and \emph{$\s$-nondegenerate} otherwise.

\item  \label{E:condi2} given subcurves $Y_1,Y_2,Y_3\in \BCon(X)$ without pairwise common irreducible components such that $X=Y_1\cup Y_2\cup Y_3$, we have that:
\begin{enumerate}
 \item if two among the subcurves $\{Y_1,Y_2,Y_3\}$ are $\s$-degenerate, then also the third one is $\s$-degenerate;
            \item the following holds
            \begin{equation*}
            \sum_{i=1}^{3}\s_{Y_i}-\chi \in\begin{cases}
                \{1,2\} \textup{ if $Y_i$ is $\s$-nondegenerate for all $i$};\\
                \{1\} \textup{ if there exists a unique $i$ such that $Y_i$ is $\s$-degenerate}; \\
                \{0\} \textup{ if $Y_i$ is $\s$-degenerate for all $i$}.
            \end{cases}
        \end{equation*}
\end{enumerate}
\end{enumerate}

The \emph{degeneracy subset} of $\s$ is the collection
\begin{equation*}
\D(\s):=\{Y\in \BCon(X): Y \text{ is $\s$-degenerate}\}.
\end{equation*}
    A V-stability condition $\s$ is called \emph{general}
    if every $Y\in \BCon(X)$ is $\s$-nondegenerate, i.e. if $\D(\n)=\emptyset$.

\vspace{0.1cm}

To any V-stability condition on $X$, we  associate a compactified Jacobian stack of $X$ as  follows.

\begin{fact*}\label{ThmA}(see \cite[Thm.~A, Thm~C(2)]{FPV1}) 
For any V-stability condition $\mathfrak s$ on $X$, then 
\begin{equation*}
\ov \J_X(\s):=\{I \in \TF_X\: : \: \chi(I)=|\s| \text{ and } \chi(I_Y)\geq \s_Y  \text{ for any  } Y\in \BCon(X)\}.
\end{equation*}
is a compactified Jacobian stack of $X$, called the \textbf{V-compactified Jacobian stack} associated to $\mathfrak s$. 
Moreover, $\mathfrak s$ is uniquely determined by $\ov \J_X(\mathfrak s)$. 
\end{fact*}

The classical compactified Jacobians (constructed in \cite{Oda1979CompactificationsOT}, \cite{simpson},  \cite{caporaso} and  \cite{esteves}) are special cases of V-compactified Jacobians, as we now explain. 
If $\psi$ is a  \emph{numerical polarization} on $X$ of characteristic $\chi=|\psi|\in \Z$, i.e. an additive function 
$$
\psi:\{\text{Subcurves of } X\} \longrightarrow \RR \: \text{ such } \psi_X=\chi,
$$
then the collection 
 \begin{equation*}
 \mathfrak s(\psi):=\left\{\mathfrak s(\psi)_Y:=\lceil \psi_Y \rceil  \: \text{ for any  } Y\in \BCon(X)\right\}
 \end{equation*}
 is a  V-stability condition of characteristic $\chi$, called the V-stability condition associated to $\psi$. 
 Then we say that the \textbf{classical compactified Jacobian} associated to $\psi$ is the V-compactified Jacobian associated to  $\s(\psi)$, namely:  
\begin{equation}\label{E:fcJphi}
\ov J_X(\s(\psi)):=\ov J_X(\psi)=\left\{I \in \TF_X: \: \chi(I)=|\psi| \text{ and } \chi(I_Y)\geq \psi_Y \text{ for any } Y\in \BCon(X)\right\}.
\end{equation}

Summing up, we have the following inclusions 
\begin{equation}\label{E:inc-fcJ}
\left\{\begin{aligned}
\text{Classical compactified} \\
\text{Jacobians of } X
\end{aligned}\right\} \hookrightarrow
\left\{\begin{aligned}
\text{V-compactified} \\
\text{Jacobians of } X
\end{aligned}\right\} \hookrightarrow
\left\{\begin{aligned}
\text{Compactified} \\
\text{Jacobians of } X
\end{aligned}\right\},
\end{equation}
which were shown, in general, to be strict, in \cite[Thm. C]{viviani2023new} (already for nodal curves and in the fine case). 

\vspace{0.1cm}

In this paper, we provide two characterizations  of V-compactified Jacobians of a nodal curve. The first characterization shows that V-compactified Jacobians provide a solution to the initial Question.

\begin{theoremalpha}\label{ThmB}(see Theorems \ref{T:VcJ-smoo} and \ref{T:cla-smoo})
Let $X$ be a connected nodal curve and let $\ov \J_X^d$ be a compactified Jacobian stack of $X$ of degree $d$. Then the following conditions are equivalent:
\begin{enumerate}
\item  $\ov \J_X^d$ is V-compactified Jacobian stack, i.e. $\ov \J_X^d=\ov \J_X(\mathfrak s)$ for some V-stability condition $\mathfrak s$ on $X$; 
 \item $\ov \J_X^d$ is  \emph{smoothable}, i.e.
    for any one-parameter smoothing $\X/\Delta$ of $X$, the open substack $\ov \J_{\X}^d:=\ov \J_X^d\cup \J_{\X_{\eta}}^d\subset \TF_{\X/\Delta}^d$ admits a good moduli space $\ov J_{\X}^d$ that is proper over $\Delta$.
\end{enumerate}
\end{theoremalpha}
In the above Theorem, we denote by  $\TF^d_{\X/\Delta}$ the stack of relative rank-$1$, torsion-free, degree-$d$ sheaves on the family  $\X\to \Delta$. 

Let us comment on the proof of Theorem~\ref{ThmB}. The implication (1) $\Rightarrow$ (2) is a special case of \cite[Thm.~A]{FPV1}. For the implication (2) $\Rightarrow$ (1) we proceed as follows.  
Recall that the stack $\TF_X^d$ of rank-$1$, torsion-free sheaves of degree~$d$ on $X$ admits an action of the generalized Jacobian $\PIC^{\un 0}(X)$, whose orbits are given by $\TF_X=\coprod_{(G,D)} \TF(G,D)$, where $\TF(G,D)$ is the locally closed substack consisting of sheaves whose non-singular graph, i.e. the spanning subgraph of $\Gamma_X$ obtained by removing the edges corresponding to the nodes where the sheaf is not free,  is $G$ and whose multidegree is the divisor $D$  on $\Gamma_X$ with degree $d-|E(G)^\mathsf{c}|$ (see Section~\ref{Sub:sheaves} for more details). 
We first prove in Proposition~\ref{P:unionorb} that any compactified Jacobian stack $\ov \J_X^d$ is a union of such orbits and hence we can define its poset of orbits as the collection
$$\P(\ov \J_X^d):=\bigcup_G \P(\ov \J_X^d)(G), \  \text{ where } \ \P(\ov \J_X^d)(G):=\{D\in \Div^{d-|E(G)^\mathsf{c}|}(\Gamma_X)\: : \TF(G,D)\subset \ov \J_X^d\},$$
where $G$ runs over all spanning subgraphs of $\Gamma_X$.
We then show in Corollary \ref{C:prop-deg-cJ} that $\P(\ov \J_X^d)$ contains a BD-set (=generalized break divisor set, see Definition \ref{D:BD-set}) with degeneracy subset 
$$
\D(\ov \J_X^d):=\{Y\in \BCon(X): \text{ there exists a sheaf } I=I_{Y}\oplus I_{Y^\mathsf{c}}\in \ov \J_X^d\}.
$$
Next, we show in Proposition~\ref{P:PT-Jbar} that the assumption that  $\ov \J_X^d$ is smoothable forces the poset $\ov \P(\ov \J_X^d)$ to be a Picard Type assignment (shortened in PT-assignment throughout see Definition \ref{D:PTstab}), which roughly speaking means that each set $\P(\ov \J_X^d)(G)$ provides an almost-unique set of representative of divisors on $G$ up to the chip-firing action.   
We then conclude by arguing, in Theorem~\ref{T:V=PT}, that PT-assignments coincide with the semistable sets of V-stability conditions.

Our second main result characterizes V-compactified Jacobians of nodal curves as those, among the compactified Jacobians with a fixed degeneracy subset, that have the smallest poset of orbits, or the smallest number of irreducible components. 
Observe that, from the definition of the degeneracy subset $\D(\ov \J_X^d)$, it follows easily that  if $\P(\ov \J_X^d)(G)\neq \emptyset$ then $G$ must be \text{$\D(\ov \J_X^d)$-admissible}, i.e. it must satisfy that each of its connected component $H$ is such that the corresponding subcurve $X[V(H)^\mathsf{c}]$ of $X$ induced by $V(H)^\mathsf{c}$ is a union of biconnected subcurves in $\D(\ov \J_X^d)$ without common pairwise irreducible components (see Definition~\ref{D:Dadmissible}).

\begin{theoremalpha}\label{ThmC}(see Corollary \ref{C:bound-orb}, Theorems \ref{T:V=PT} and \ref{T:smallcJ})
Let $X$ be a connected nodal curve over $k=\ov k$ and let $\ov \J_X^d$ be a connected compactified Jacobian of $X$ of degree $d$ with degeneracy subset  $\D(\ov \J_X^d)$. 
\begin{enumerate}
    \item \label{ThmC1} For any $\D(\ov \J_X^d)$-admissible spanning subgraph $G$ of $\Gamma_X$, we have that 
$$
|\P(\ov \J_X^d)(G)|\geq c_{\D(\ov \J_X^d)}(G)
$$
where $c_{\D(\ov \J_X^d)}(G)$ is the number (called the $\D(\ov \J_X^d)$-complexity of $G$, see see Definition~\ref{D:Dadmissible}) of minimal $\D(\ov \J_X^d)$-admissible spanning subgraphs contained in $G$. In particular, $\ov \J_X^d$ has at least $c_{\D(\ov \J_X^d)}(\Gamma_X)$ irreducible components.
\item \label{ThmC2} The following conditions are equivalent:
\begin{enumerate}[(a)]
    \item $\ov \J_X^d$ is a V-compactified Jacobian, i.e. $\ov \J_X^d=\ov \J_X(\s)$ for some V-stability condition $\s$ on $X$, in which case   $\D(\ov \J_X^d)=\D(\s)$. 
    \item Equality holds in Part~\eqref{ThmC1} for any $\D(\ov \J_X^d)$-admissible spanning subgraph $G$ of $\Gamma_X$.
    \item The stack $\ov \J_X^d$ has exactly $c_{\D(\ov \J_X^d)}(\Gamma_X)$ irreducible components.
\end{enumerate}
\end{enumerate}
\end{theoremalpha}
Theorem~\ref{ThmC} generalizes the fact, well--known for classical fine compactified Jacobians (see \cite{meloviviani}), and more generally established in \cite[Cor~2.34]{viviani2023new}, that smoothable fine compactified Jacobians of $X$ have a number of irreducible components equal to the complexity of the dual graph of $X$, and that all other fine compactified Jacobians have a larger number of irreducible components.

Let us comment on the proof of Theorem~\ref{ThmC}. Part~\eqref{ThmC1} follows from the fact that the poset of orbits $\P(\ov \J_X^d)$ of $\ov \J_X^d$ contains a BD-set $\BD_I$ with degeneracy subset $\D(\ov \J_X^d)$ (by Corollary~\ref{C:prop-deg-cJ}) together with the inequality $\BD_I(G)\geq c_{\D(\ov \J_X^d)}(G)$ which holds for any BD-set by Proposition~\ref{BDthm}.

As for Part~\eqref{ThmC2}, the implication (a)$\Rightarrow$(b) follows from the fact that if $\ov \J_X^d$ is a V-compactified Jacobian then its poset of orbits $\ov \P(\ov \J_X^d)$ is a PT-assignment with degeneracy subset $\D(\ov \J_X^d)$ by Theorem \ref{T:V=PT} and any PT-assignment $\P$ with degeneracy subset $\D$ satisfies the formula $\P(G)=c_{\D}(G)$ for any $\D$-admissible spanning subgraph $G$ by Proposition~\ref{P:cardPTstab} (which can be seen as a generalization of  Kirchhoff's theorem). The implication (b)$\Rightarrow$(c) is obvious, while the  implication (c)$\Rightarrow$(a) follows from the fact that if a BD set $\BD_I$ with degeneracy subset $\D$ satisfies $\BD_I(\Gamma_X)=c_{\D}(\Gamma_X)$ then it is a PT-assignment by Theorem~\ref{T:V=PT} together with the fact that if two connected compactified Jacobians with the same degeneracy set are contained one into the other then they are equal by Proposition~\ref{P:incl-cJ}.

\subsection*{Open questions}

This work leaves open some natural question which we plan to investigate in the near future.

\begin{enumerate}
    \item Does Theorem \ref{ThmB} extend to other classes of smoothable reduced curves?
    
    For example, it would be interesting to determine whether Theorem \ref{ThmB} extends to reduced curves with locally planar singularities, whose V-compactified Jacobians have nice geometric properties as shown  in \cite[Sec. 8]{FPV1}.
    
    \item Can one characterize V-compactified Jacobians of a nodal curve $X$ as those compactified Jacobians that are \emph{regularly smoothable}, i.e. those compactified Jacobian stacks $\ov \J^d_X$ such that  there exists a regular one-parameter smoothing $\X/\Delta$ of $X$ with the property that the open substack $\ov \J_{\X}^d:=\ov \J_X^d\cup \J_{\X_{\eta}}^d\subset \TF_{\X/\Delta}^d$ admits a good moduli space $\ov J_{\X}^d$ that is proper over $\Delta$?

    A positive answer has been obtained for fine compactified Jacobians in  \cite{viviani2023new} (building upon the results of \cite{pagani2023stability}).
\end{enumerate}

\subsection*{Outline of the paper}

The paper is organized as follows. In Section \cref{Sec:Vstab}, we introduce V-stability conditions on graphs:  we study their associated semistable sets in \cref{Sub:ssV} and some operations on V-stabilities in \cref{Sub:OperV}. In Section \cref{Sec:PT}, we introduce PT-assignments and BD-sets: we prove a Kirchhoff's type formula for the cardinality of a PT-assignment in \cref{Sub:Kirc} and we prove in \cref{Sub:PT=V} that PT-assignments coincide with the semistable sets of V-stabilities conditions and with BD-sets having the smallest number of maximal elements. In Section \cref{Sec:cJ}, we study compactified Jacobians: after reviewing the properties of the stack of rank-$1$ torsion-free sheaves on a nodal curve in \cref{Sub:sheaves}, we introduce compactified Jacobian stacks and spaces in \cref{Sub:cJ}. Then, we study V-compactified Jacobians in \cref{Sub:VcJ}, and we give two characterizations of them in the following subsections: in \cref{Sub:smallcJ} we characterize them as those with the smallest number of irreducible components, and in \cref{SUb:smoothcJ} we characterize them as the smoothable ones. In Section \cref{Sec:Examples} we discuss examples of compactified Jacobians: in \cref{Sub:maxdeg} we classify the maximally degenerate compactified Jacobians; we classify all compactified Jacobians of irreducible curves (resp. vine curves, resp. compact-type curves) in \cref{Sub:irr} (resp. \cref{Sub:vine}, resp. \cref{Sub:compact}); in \cref{Sub:cycle} we classify V-compactified Jacobians for necklace curves.

\vspace{0.5cm}

\paragraph { \bf Acknowledgements}
MF is supported by the DTP/EPSRC award 
EP/W524001/1.

NP is funded by the PRIN 2022 ``Geometry Of Algebraic Structures: Moduli, Invariants, Deformations'' funded by MUR, and he is a member of the GNSAGA section of INdAM. 

FV is funded by the MIUR  ``Excellence Department Project'' MATH@TOV, awarded to the Department of Mathematics, University of Rome Tor Vergata, CUP E83C18000100006, by the  PRIN 2022 ``Moduli Spaces and Birational Geometry''  funded by MIUR,  and he is a member of INdAM and of the Centre for Mathematics of the University of Coimbra (funded by the Portuguese Government through FCT/MCTES, DOI 10.54499/UIDB/00324/2020).

\section{V-stability conditions on graphs}\label{Sec:Vstab}

After fixing some notation, the main definition of this section is Definition~\ref{D:Vstab}, where we introduce the notion of a V-stability condition and its semistable (resp. polystable, stable) set (see Definition~\ref{D:Pn}).

\subsection{Notation on graphs}\label{Sub:not-gr}

Let $\Gamma$ be a finite graph with vertex set $V(\Gamma)$ and edge set $E(\Gamma)$. The \emph{genus} of $\Gamma$ is 
$$g(\Gamma):=|E(\Gamma)|-|V(\Gamma)|+\gamma(\Gamma),$$
where $\gamma(\Gamma)$ is the number of connected components of $\Gamma$.

The spaces of \emph{$1$-cochains} and of \emph{$1$-chains} on a graph $\Gamma$ can be defined in the following way. Denote by $\EE(\Gamma)$ the set of oriented edges of $G$ and denote by $s,t:\EE(\Gamma)\to V(\Gamma)$ the source and target functions that associate to an oriented edge $\ee$ its source and target, respectively.  Given an oriented edge $\ee$, we denote by $\ov{\ee}$ the opposite oriented edge so that $t(\ee)=s(\ov \ee)$ and $s(\ee)=t(\ov \ee)$, and we will denote by $|\ee|=|\ov \ee|\in E(\Gamma)$ the underlying non-oriented edge. 
The space of $1$-chains and $1$-cochains on $\Gamma$ with coefficients in an abelian group $A$ (e.g. $A=\ZZ$)  are defined as 
\begin{equation*}
\CC_1(\Gamma,A):=\frac{\bigoplus_{\ee \in \EE(\Gamma)} A\cdot \ee}{(\ee=-\ov \ee)} \quad \text{ and } \quad \CC^1(\Gamma,A):=\{f:\EE(\Gamma)\to A\: : f(\ee)=-f(\ov \ee)\}. 
\end{equation*}

The space of \emph{$0$-cochains} (or functions) on a  graph $\Gamma$ with coefficients in an abelian group $A$ (e.g. $A=\ZZ$) is given by 
\begin{equation*}\label{E:0-cochains}
C^0(\Gamma,A):=\left\{g:V(\Gamma)\to A\right\}.
\end{equation*}
and it is endowed with the coboundary map 
\begin{equation*}\label{E:cobound}
\begin{aligned}
\delta=\delta_\Gamma: C^0(\Gamma,A) & \longrightarrow \CC^1(\Gamma,A),\\
g &\mapsto \delta(g)(\ee):=g(t(\ee))-g(s(\ee)).
\end{aligned}
\end{equation*}
We will be interested in two kinds of \textbf{subgraphs} of $\Gamma$:
\begin{itemize}
\item Given a subset $T\subset E(\Gamma)$, we denote by   $\Gamma\setminus T$ the
subgraph of $\Gamma$ obtained from $\Gamma$ by deleting the edges belonging to $T$.
Thus we have that $V(\Gamma\setminus T)=V(\Gamma)$ and $E(\Gamma\setminus T)=E(\Gamma)\setminus T$.
The subgraphs of the form $\Gamma\setminus T$ are called \emph{spanning subgraphs}.

The set of all spanning subgraphs of $\Gamma$ is denoted by $\SS(\Gamma)$ and it is a poset via the relation
$$
G=\Gamma\setminus T\leq G'=\Gamma\setminus T' \Longleftrightarrow T\supseteq T'.
$$
We will denote by $\SS_{\rm con}(\Gamma)$ the subset of $\SS(\Gamma)$ consisting of all the spanning subgraphs whose connected componentes coincide with those of $\Gamma$. 
If $\Gamma$ is connected, the minimal elements of $\SS_{\rm con}(\Gamma)$ are exactly the \emph{spanning trees} of $\Gamma$, i.e. the connected spanning subgraphs that are trees.
The set of all spanning trees of $\Gamma$ will be denoted by $\S\T(\Gamma)$.

\item Given a non-empty subset $W\subset V(\Gamma)$, we denote by   $\Gamma[W]$ the subgraph
whose vertex set is $W$ and whose edges are all the edges of $\Gamma$ that join two vertices in $W$.
The subgraphs of the form $\Gamma[W]$ are called \emph{induced subgraphs} and we say that
$\Gamma[W]$ is induced from $W$.

We will use the following notation for the \emph{biconnected} and \emph{connected} subsets of $V(\Gamma)$, respectively: 
\begin{equation}\label{E:Bconn}
\begin{sis}
&     \BCon(\Gamma):=\{\emptyset \subsetneq W \subsetneq V(\Gamma)\: : \Gamma[W] \text{ and } \Gamma[W^\mathsf{c}] \text{ are connected}\},\\
&     \Con(\Gamma):=\{\emptyset \subsetneq W \subseteq V(\Gamma)\: : \Gamma[W] \text{ is connected}\}.\\
\end{sis}
\end{equation}
\end{itemize}

Given a subset $T\subset E(\Gamma)$, we will denote by $\Gamma/T$ the $T$-\emph{contraction} of $\Gamma$, i.e. the graph obtained from $\Gamma$ by contracting all the 
edges belonging to $T$.  A \textbf{morphism of graphs} $f:\Gamma\to \Gamma'$ is a contraction of some edges of $\Gamma$ followed by an automorphism of $\Gamma'$. 
Such a morphism induces an injective pull-back map on edges $f^E:E(\Gamma')\hookrightarrow E(\Gamma)$ and a surjective push-forward map on vertices $f_V:V(\Gamma)\twoheadrightarrow V(\Gamma')$.
Note that the complement $(\Im f^E)^\mathsf{c}$ of the image of $f^E$ consists of the edges that are contracted by $f$.

Fix a subset $S\subseteq E(\Gamma)$. For any pair $(W_1,W_2)$ of disjoint subsets of $V(\Gamma)$, we denote the \emph{$S$-valence} of $(W_1,W_2)$ to be 
$$\val_S(W_1,W_2):=|S\cap E(\Gamma[W_1],\Gamma[W_2])|,$$
where $E(\Gamma[W_1], \Gamma[W_2])$ is the subset of $E(\Gamma)$ consisting of all the edges
of $\Gamma$ that join some vertex of $W_1$ with some vertex of $W_2$. As a special case, the $S$-valence of a subset $W\subseteq V(\Gamma)$ is 
$$\val_S(W):=\val_S(W,W^\mathsf{c}),$$
where $W^\mathsf{c}:=V(\Gamma)\setminus W$ is the complement subset of $W$.  For any subset $W\subseteq V(\Gamma)$, we denote the \emph{S-degree} of $W$ by
$$
e_S(W):=|S\cap E(\Gamma[W])|.
$$
In the special case $S=E(\Gamma)$, we set 
$$
e_{\Gamma}(W):=e_{E(\Gamma)}(W) \text{ and } \val_{\Gamma}(W_1,W_2):=\val_{E(\Gamma)}(W_1,W_2),
$$
and we call them the $\Gamma$-degree and the $\Gamma$-valence, respectively. When the graph $\Gamma$ is clear from the context, we set 
$$
e(W):=e_{\Gamma}(W) \text{ and } \val(W_1,W_2):=\val_{\Gamma}(W_1,W_2).
$$
Note that, using the identification $V(\Gamma)=V(\Gamma\setminus S)$, we have that 
 \begin{equation}\label{E:val-valS}
e_{\Gamma}(W)=e_{S}(W)+e_{\Gamma\setminus S}(W) \text{ and } \val_{\Gamma}(W_1,W_2)=\val_{S}(W_1,W_2)+\val_{\Gamma\setminus S}(W_1,W_2).
\end{equation}
Given pairwise disjoint subsets $W_1, W_2, W_3\subseteq V(\Gamma)$, we have the following formulas
\begin{equation}\label{add-val}
\begin{sis}
& \val_S(W_1\cup W_2,W_3)=\val_S(W_1,W_3)+\val_S(W_2,W_3),\\
& \val_S(W_1\cup W_2)=\val_S(W_1)+\val_S(W_2)-2\val_S(W_1,W_2),\\
& e_S(W_1\cup W_2)=e_S(W_1)+e_S(W_2)+\val_S(W_1,W_2),\\
& |S|=e_S(W_1)+e_S(W_1^\mathsf{c})+\val_S(W_1).
\end{sis}
\end{equation}

A \textbf{divisor} on $\Gamma$ is an integral linear combination of vertices: 
$$
D=\sum_{v\in V(\Gamma)} D_v \cdot v.
$$
With respect to pointwise addition, divisors form a group denoted by $\Div(\Gamma)$.
Given a subset $W\subseteq V(\Gamma)$, we will set 
$D_W:=\sum_{v\in V(\Gamma)} D_v.$
The degree of a divisor $D$ is given by 
$$
\deg(D)=|D|:=D_{V(\Gamma)}=\sum_{v\in V(\Gamma)} D_v.
$$
For any $d\in \ZZ$, we will denote by $\Div^d(\Gamma)$ the set of all divisors of degree $d$.

We can associate to any integral function on $\Gamma$ a divisor on $\Gamma$ in the following way:
\begin{equation}\label{E:div}
  \begin{aligned}
      \div: C^0(\Gamma,\ZZ) & \longrightarrow \Div(\Gamma), \\
      f & \mapsto \div(f)=\sum_{v\in V}\left[\sum_{\substack{\ee\in \EE(\Gamma)\: \\ t(\ee)=v}}\left(f(s(\ee))-f(t(\ee)) \right)\right]v.
  \end{aligned}  
\end{equation}
The map $\div$ is a homomorphism of groups whose image $\Prin(\Gamma):=\Im(\div)$, called the subgroup of principal divisors, is contained in $\Div^0(\Gamma)$. Moreover, if $\Gamma$ is connected, then the kernel of $\div$ consists of the constant functions, which implies that the quotient
\[\Pic^0(\Gamma):=\frac{\Div^0(\Gamma)}{\Prin(\Gamma)}\]
is a finite group, called the \textbf{Jacobian group} of $\Gamma$ (also known as degree class group, or  sandpile group, or  critical
group of the graph $\Gamma$). The Kirchhoff's matrix-tree theorem implies that the cardinality of $\Pic^0(\Gamma)$ is
equal to the \emph{complexity} $c(\Gamma)$ of $\Gamma$, defined as the number of spanning trees of $\Gamma$.

For any $d\in \ZZ$, the set $\Div^d(\Gamma)$ is a torsor for the group $\Div^0(\Gamma)$.
Therefore, the subgroup $\Prin(\Gamma)$ acts on $\Div^d(\Gamma)$ and the quotient 
\begin{equation}\label{lat-class0}
\Pic^d(\Gamma)=\frac{\Div^d(\Gamma)}{\Prin(\Gamma)}
\end{equation}
is a torsor for $\Pic^0(\Gamma)$. 

As in \cite[Definition $1.1$]{viviani2023new}, given a connected graph $\Gamma$ and an integer $d\in\ZZ$, we define the \textbf{poset of divisors on spanning subgraphs} of degree $d$ as the set 
\begin{equation}\label{E:OO}
    \OO^d(\Gamma):=\{(G,D):G\in\S\S(\Gamma),D\in \textup{Div}^{d-|E(G)^\mathsf{c}|}(G)=\textup{Div}^{d-|E(G)^\mathsf{c}|}(\Gamma)\},
\end{equation}
where $\S\S(\Gamma)$ denotes the set of spanning subgraphs of $\Gamma$,
endowed with a partial order that we are now going to define. 

Given a partial orientation $\O$ of $G$, i.e. an orientation of a subset of $E(G)$ denoted by $\supp(\O)$, we define the outgoing divisor of $\O$ to be
\begin{equation}\label{E:outDiv}
D(\O):=\sum_{\ee\in \O}t(\ee)\in \Div^{|\supp(\O)|}(G).
\end{equation}
The poset structure on $\OO^d(\Gamma)$ is defined as follows 
\begin{equation}\label{E:leq-GD}
    (G,D)\geq (G',D')\iff
    \begin{cases}
        \text{there exists a partial orientation $\O$ of $G$ such that} \\
       G'=G\setminus \supp(\O) \text{ and } D'=D-D(\O).
    \end{cases}
\end{equation}
We will write $(G,D)\geq_{\O} (G',D')$ if the above condition holds, and note that any two of the data $\{(G,D), (G',D'),\O\}$ uniquely determines the third.

We also sometimes consider the subposet 
\begin{equation}\label{E:OOcon}
    \OO^d_{\con}(\Gamma):=\{(G,D)\in \OO^d(\Gamma) : G\in\S\S_{\con}(\Gamma)\}.
\end{equation}

There are two operations that we can perform on the poset $\OO^d(\Gamma)$:
\begin{itemize}
    \item for any $G\in \SS(\Gamma)$,
there is a natural inclusion of posets 
\begin{equation}\label{E:incl-O}
\begin{aligned}
    \iota_G:\OO^{d-|E(G)^\mathsf{c}|}(G)& \hookrightarrow \OO^d(\Gamma), \\
    (G',D) & \mapsto (G',D). 
    \end{aligned}
\end{equation}
\item for any morphisms $f:\Gamma \to \Gamma'$ of graphs, there is a push-forward map
\begin{equation}\label{E:f*-O}
\begin{aligned}
f_*: \OO^d(\Gamma) & \longrightarrow \OO^d(\Gamma')\\
(\Gamma\setminus S, D) & \mapsto f_*(\Gamma\setminus S,D):= \left(f_*(\Gamma\setminus S), f_*(D)\right),
\end{aligned}
\end{equation}
where $f_*(\Gamma\setminus S):=\Gamma'\setminus (f^E)^{-1}(S)$ and $f_*(D)$ is defined by 
$$
 f_*(D):=\left\{f_*(D)_v:=D_{f_V^{-1}(v)}+e_{S\cap (\Im f^E)^\mathsf{c}}(f_V^{-1}(v)) \right\}_{v\in V(\Gamma')}
$$
\end{itemize}
It is proved in \cite[Prop. 1.2]{viviani2023new} that $f_*$ is surjective, order-preserving and it has the upper lifting property.

\subsection{Definition of V-stability conditions}\label{Sub:DefVStab}

Given a connected graph $\Gamma$, we denote by $\BCon(\Gamma)$ the set of all nontrivial biconnected subsets of $V(\Gamma)$, as in \eqref{E:Bconn}.

\begin{definition}\label{D:Vstab}
    Let $\Gamma$ be a connected graph. A \textbf{V-stability condition} of degree $d$ on $\Gamma$ is a map
    \begin{align*}
        \n:\BCon(\Gamma)&\to \ZZ\\
        W&\mapsto \n_W
    \end{align*}
    satisfying the following properties:
    \begin{enumerate}
        \item \label{D:Vstab1} For any $W\in \BCon(\Gamma)$, we have
        \begin{equation}\label{E:sum-n}
            \n_W+\n_{W^\mathsf{c}}+\val(W)-d\in\{0,1\}.
        \end{equation}
        The subset $W$ is said to be \emph{$\n$-degenerate} if $\n_W+\n_{W^\mathsf{c}}+\val(W)-d=0$, and \emph{$\n$-nondegenerate} otherwise.

        \item \label{D:Vstab2} For any  pairwise disjoint $W_1,W_2,W_3\in \BCon(\Gamma)$ such that and  $W_1\cup W_2\cup W_3=V(\Gamma)$, we have that:
        \begin{itemize}
            \item for each $i\neq j$, if $W_i$ and $W_j$ are $\n$-degenerate, then $W_k$ is $\n$-degenerate, where $k\neq i,j$;
            \item the integer
            \begin{equation}\label{E:tria-n}
            \sum_{i=1}^{3}\n_{W_i}+\sum_{1\leq i<j\leq 3} \val(W_i,W_j)-d \in\begin{cases}
                \{1,2\}, \textup{ if $W_i$ is $\n$-nondegenerate for all $i$};\\
                \{1\}, \textup{ if there exists a unique $i$ such that $W_i$ is $\n$-degenerate};\\
                \{0\}, \textup{ if $W_i$ is $\n$-degenerate for all $i$}.
            \end{cases}
        \end{equation}
        \end{itemize}
    \end{enumerate}
    The degree $d$ of $\n$ is also denoted by $|\n|$.

\vspace{0.1cm}
The \emph{degeneracy subset} of $\n$ is the collection
\begin{equation}\label{E:Dn}
\D(\n):=\{W\in \BCon(\Gamma)\: : W \text{ is $\n$-degenerate} \}.
\end{equation}
    A V-stability condition $\n$ is called \emph{general} (as in \cite[Definition $1.4$]{viviani2023new}) if every $W\in \BCon(\Gamma)$ is $\n$-nondegenerate, i.e. $\D(\n)=\emptyset$.
\end{definition}

\begin{remark}\label{R:n-deg}
Let $\n$ be a V-stability condition on $\Gamma$.
\begin{enumerate}
    \item \label{R:n-deg1}
   The degeneracy subset $\D(\n)$  satisfies the following two properties:
    \begin{itemize}
        \item for any $W\in \BCon(\Gamma)$, we have that $W\in \D(\n)$ if and only if $W^\mathsf{c}\in \D(\n)$.
        \item for any two disjoint subsets $W_1,W_2\in \BCon(\Gamma)$ such that $W_1\cup W_2\in \BCon(\Gamma)$, we have that if two of the elements of $\{W_1,W_2,W_1\cup W_2\}$ belong to $\D(\n)$ then also the third one belongs to $\D(\n)$. 
    \end{itemize}
   This follows easily from the above Definition \ref{D:Vstab}. 
   \item \label{R:n-deg2} If $W_1,W_2$ are disjoint elements of $\BCon(\Gamma)$ such that $W_1 \cup W_2\in \BCon(\Gamma)$, then we have that 
   \begin{equation}\label{E:n-union}
   \n_{W_1\cup W_2}-\n_{W_1}-\n_{W_2}-\val(W_1,W_2)\in 
   \begin{cases}
        \{0\} & \text{ if $W_1$  or  $W_2$ is $\n$-degenerate};\\
        \{-1\} & \text{ if $W_1$  and  $W_2$ are $\n$-nondegenerate} \\
        & \text{ and $W_1\cup W_2$ is $\n$-degenerate};\\
      \{0,-1\} & \text{ if $W_1, W_2, W_1\cup W_2$ are $\n$-nondegenerate}.\\
   \end{cases}
   \end{equation}
   This follows by combining \eqref{D:Vstab1} for $W_1\cup W_2$ and \eqref{D:Vstab2} for $(W_1,W_2,(W_1\cup W_2)^\mathsf{c})$.
   \end{enumerate}
\end{remark}

\begin{example}[\textbf{Classical V-stability conditions}]
Let $\phi$ be a \emph{numerical polarization} on $\Gamma$ of degree $d$, i.e. an additive function
$$
\begin{aligned}
    \phi:\left\{\text{Subsets of } V(\Gamma) \right\}&  \longrightarrow \RR,\\
   W & \mapsto \phi_W,
\end{aligned}    
$$
such that $|\phi|:=\phi_{V(\Gamma)}=d$. Then the 
function 
\begin{equation}\label{E:n-phi}
\begin{aligned}
        \n(\phi):\BCon(\Gamma)&\to \ZZ\\
        W&\mapsto \n(\phi)_W:=\left\lceil\phi_W-\frac{\textup{val}(W)}{2}\right\rceil.
    \end{aligned}
    \end{equation}
is a V-stability condition of degree $d$, called the V-stability condition associated to $\phi$.
This follows, by taking the upper integral parts, from the following two equalities
$$
\begin{sis}
& \left(\phi_W-\frac{\val(W)}{2}\right)+\left(\phi_{W^\mathsf{c}}-\frac{\val(W^\mathsf{c})}{2}\right)-d+\val(W)=0,\\    
& \sum_{i=1}^3 \left(\phi_{W_i}-\frac{\val(W_i)}{2} \right)+\sum_{1\leq i<j\leq 3} |E(W_i,W_j)|-d=0.
\end{sis}
$$
The V-stabilities of the form $\n(\phi)$ are called \emph{classical}.

Consider the arrangement of hyperplanes in the affine space $\Div^d(\Gamma)_{\RR}$ of numerical polarizations of degree $d$ (see \cite[Sec. 7]{Oda1979CompactificationsOT} and \cite[Sec. 3]{MRV}):
\begin{equation}\label{E:arr-hyper}
\A_{\Gamma}^d:=\left\{\phi_W-\frac{\val(W)}{2}=n\right\}_{W\in \BCon(\Gamma), n\in \ZZ}.
\end{equation}
 We get an induced wall and chamber decomposition of $\Div^d(\Gamma)_{\RR}$ such that two numerical polarizations $\phi, \phi'$ belong to the same region, i.e. they have the same relative positions with respect to all the hyperplanes,  if and only if $\n(\phi)=\n(\phi')$. In other words, the set of regions induced by $\A_{\Gamma}^d$ is the set of classical V-stability conditions of degree $d$.

Note also that $\n(\phi)$ is a general V-stability condition if and only if $\phi$ belongs to a chamber (i.e. a maximal dimensional region), or equivalently, if it does not lie on any hyperplane. 
\end{example}

We now introduce the extended degeneracy set and the extended V-function of a V-stability.

\begin{definition}\label{D:ext-n}
    Let $\n$ be a V-stability of degree $d$ on $\Gamma$.
    \begin{enumerate}
    \item\label{D:ext-n1} The \emph{extended degeneracy subset} of $\n$ is 
       $$
        \wh \D(\n):=\left\{
    \begin{aligned}
        & W\in \Con(\Gamma)\: :  \Gamma[W^\mathsf{c}]=\Gamma[Z_1]\coprod\ldots \coprod \Gamma[Z_k] \\
        &    \text{ with } Z_i \in \D(\n) \text{ for all } i=1, \ldots, k.
     \end{aligned}   
        \right\} \bigcup \{V(\Gamma)\}. $$
        \item \label{D:ext-n2}
    The \emph{extended V-function}  associated to $\n$ is the function 
    (denoted with the same symbol by abuse of notation) 
    $$
    \begin{aligned}
     \n:\left\{\text{Subsets of } V(\Gamma)\right\}  & \longrightarrow \ZZ\\
     W& \mapsto \n_W
     \end{aligned}
    $$
    defined as follows.

    If $W\in \Con(\Gamma)$ and $\Gamma[W^\mathsf{c}]=\Gamma[Z_1]\coprod\ldots \coprod \Gamma[Z_k]$ is the decomposition into connected components (which then implies that each  $Z_i$ belongs to $\BCon(\Gamma)$),we set 
     \begin{equation}\label{E:n-ext}
    \n_W= d-\sum_{i=1}^k \n_{Z_i}-\val(W)+|\{Z_i: \: Z_i\not\in \D(\n)\}|.
    \end{equation}

If $W\not \in \Con(\Gamma)$ and $\Gamma[W]=\Gamma[W_1]\coprod \ldots \coprod \Gamma[W_h]$ is the decomposition into connected components, we set 
 \begin{equation}\label{E:n-ext2}
    \n_W= \sum_{j=1}^h \n_{w_j}.
    \end{equation}
\end{enumerate}
\end{definition}
It follows from \eqref{E:sum-n} that 
\begin{itemize}
\item $\D(\n)=\wh \D(\n)\cap \BCon(\Gamma)$,
\item the restriction of the extended V-function $\n$ to $\BCon(\Gamma)$  coincides with the original V-stability condition $\n$. 
\end{itemize}

The  extended V-function, when restricted to the extended degeneracy subset, satisfies the following additivity property.

\begin{lemma}\label{L:add-whn} 
Let $W_1,W_2$ be disjoint subsets of $\wh \D(\Gamma)$ such that $W_1\cup W_2\in \wh \D(\Gamma)$.
 Then 
     $$
    \n_{W_1\cup W_2}= \n_{W_1}+\n_{W_2}+\val(W_1,W_2).
    $$
\end{lemma}
Observe that if $W_1,W_2\in \D(\n)$ with $W_1\cup W_2=V(\Gamma)$ then the above Lemma follows \eqref{E:sum-n}, while if $W_1,W_2,W_1\cup W_2\in \D(\n)$ then the above Lemma follows from \eqref{E:tria-n} (see Remark \ref{R:n-deg}\eqref{R:n-deg2}). 

\begin{proof} 
The proof is analogous to the proof of \cite[Cor. 4.7]{FPV1} (mutatis mutandis), and hence it will be omitted.
\end{proof}

We now introduce a poset structure on the set 
$$
\VStab^d(\Gamma):=\left\{\text{V-stability conditions of degree $d$ on } \Gamma\right\}. 
$$

\begin{definition}
Let $\n_1,\n_2\in \VStab^d(\Gamma)$. 
We say that $\n_1\geq \n_2$ if
$$
(\n_1)_Z\geq (\n_2)_Z \text{ for every } Z\in \BCon(\Gamma).
$$
\end{definition}

It follows from \eqref{E:sum-n} that if $\n_1\geq \n_2$ then, for any $Z\in \BCon(\Gamma)$, we have that 
\begin{equation}\label{E:n1>n2}
(\n_1)_Z=
\begin{cases}
(\n_2)_Z & \text{ if } Z \text{ is $\n_2$-nondegenerate},\\
(\n_2)_Z \text{ or } (\n_2)_Z+1 & \text{ if } Z \text{ is $\n_2$-degenerate},\\
\end{cases}
\end{equation}

\subsection{The semistable set of a V-stability condition}\label{Sub:ssV}

We now introduce some subsets of the poset $\OO^d(\Gamma)$ of \eqref{E:OO}, associated to a V-stability condition of degree $d$ on $\Gamma$.

\begin{definition}\label{D:Pn}
    Let $\n$ be a V-stability condition of degree $d$ on the graph $\Gamma$. 
    \begin{enumerate}
        \item \label{D:Pn1} The \emph{$\n$-semistable} set is the subset of $\OO^d(\Gamma)$ defined by 
    \begin{equation*}
        \P_\n:=\{(\Gamma\setminus S,D)\in\OO^d(\Gamma):D_Z+e_S(Z)\geq \n_Z \textup{ for any } Z\in  \BCon(\Gamma)\}.
    \end{equation*}
     \item \label{D:Pn2} The \emph{$\n$-polystable} set is the subset of $\OO^d(\Gamma)$ defined by 
    \begin{equation*}
        \P_\n^{ps}:=\{(\Gamma\setminus S,D)\in\P_\n: \textup{ if }  
D_Z+e_S(Z)=\n_Z \textup{ for } Z\in \D(\n) \subseteq \BCon(\Gamma), \textup{ then } E(Z,Z^\mathsf{c})\subseteq S\}.
    \end{equation*}
    \item \label{D:Pn3} The \emph{$\n$-stable} set is the subset of $\OO^d(\Gamma)$ defined by 
    \begin{equation*}
        \P_\n^{st}:=\{(\Gamma\setminus S,D)\in\P_\n:D_Z+e_S(Z)> \n_Z \textup{ for any } Z\in  \D(\n) \subseteq \BCon(\Gamma)\}.
    \end{equation*}
\end{enumerate}
    For any $G\in\S\S(\Gamma)$, we define 
    \begin{equation*}
        \P_\n(G):=\{D\in\textup{Div}^{d-|E(G)^\mathsf{c}|}(\Gamma):(G,D)\in \P_\n\},
    \end{equation*}
    and then define $\P_\n^{ps}(G)$ and $\P_\n^{st}(G)$ in analogy.
\end{definition}

Some remarks on the above definition are in order. 

\begin{remark}\label{R:prop-Pn}
Let $\n$ be a V-stability condition of degree $d$ on $\Gamma$ and let $(\Gamma\setminus S, D)\in \P_\n$.
\begin{enumerate}
    \item \label{R:prop-Pn1} If $Z\in \BCon(\Gamma)$ then 
    \begin{equation}\label{E:uppbound}
     D_Z+e_S(Z)\leq 
     \begin{cases}
         \n_Z+\val_{\Gamma\setminus S}(Z) & \text{ if $Z$ is $\n$-degenerate,}\\
         \n_Z+\val_{\Gamma\setminus S}(Z)-1 & \text{ if $Z$ is $\n$-nondegenerate.}\\
     \end{cases}
    \end{equation}

Indeed, by applying the lower bound for $D_{Z^\mathsf{c}}$ coming from definition of $\P_\n$, 
using \eqref{E:val-valS}, \eqref{add-val}, and the fact that $\deg D=d-|S|$, we compute
$$
d-|S|-D_Z=D_{Z^\mathsf{c}}\geq \n_{Z^\mathsf{c}}-e_S(Z^\mathsf{c})=\n_{Z^\mathsf{c}}+e_S(Z)+\val(Z)+\val_{\Gamma\setminus S}(Z)-|S|.
$$
We now conclude using \eqref{E:sum-n}.
    \item \label{R:prop-Pn2} Let $W$ a non-trivial connected subset of $V(\Gamma)$ and denote by
$$\Gamma[W^\mathsf{c}]=\Gamma[Z_1]\coprod \ldots \coprod \Gamma[Z_k]$$
  the decomposition of $\Gamma[W^\mathsf{c}]$ into connected components 
(so that $Z_i$ is  biconnected for every $1\leq i \leq k$). 
Then we have that 
\begin{equation}\label{E:connW}
\begin{sis}
& \n_{W^\mathsf{c}}\leq D_{W^\mathsf{c}}+e_S(W^\mathsf{c}) \leq  \n_{W^\mathsf{c}}+\val_{\Gamma\setminus S}(W^\mathsf{c})-|\{Z_i: \: Z_i\not\in \D(\n)\}|,\\
& \n_{W}\leq D_{W}+e_S(W) \leq  \n_{W}+\val_{\Gamma\setminus S}(W)-|\{Z_i: \: Z_i\not\in \D(\n)\}|.
\end{sis}
\end{equation}

Indeed, the first inequality in \eqref{E:connW} follows by summing the lower bounds for $D_{Z_i}$ from the definition of $\P_\n$ and the upper bounds \eqref{E:uppbound} for $D_{Z_i}$, and using that $\n_{W^\mathsf{c}}=\sum_{i=1}^k \n_{Z_i}$ by Definition \ref{D:ext-n}.

The second inequality is deduced from the first one using that (by \eqref{add-val})
$$
D_W+e_S(W)=d-|S|-D_{W^\mathsf{c}}+e_S(W)=d-(D_{W^\mathsf{c}}+e_S(W^\mathsf{c}))-\val_S(W),
$$
and the formula $\n_W=d-\sum_{i=1}^k \n_{Z_i}-\val_{\Gamma}(W)+|\{Z_i: \: Z_i\not\in \D(\n)\}|$ from Definition \ref{D:ext-n}. 

\item \label{R:prop-Pn3} If $W\in \wh \D(\n)$ (see Definition \ref{D:ext-n}), we have that 
    $$
    D_W+e_S(W)\geq \n_W,
    $$
    with equality if $E_{\Gamma\setminus S}(W,W^\mathsf{c})=\emptyset$.

    This follows from \eqref{E:connW} using definition \eqref{E:n-ext}.
\end{enumerate}

\end{remark}

In order to study the properties of the semistable set $\P_\n$, we now define subgraphs and divisors on them that are  "admissible" with respect to $\n$.

\begin{definition}\label{D:adm-n}
   Let $\n$ be a V-stability of degree $d$ on $\Gamma$.
 \begin{enumerate}
      \item A connected subgraph $G$ of $\Gamma$ is said to be \emph{$\n$-admissible} 
      if $V(G)\in \wh \D(\n)$.
      
     \item A spanning subgraph $G$ of $\Gamma$ is said to be \emph{$\n$-admissible} if it is a disjoint union of connected $\n$-admissible subgraphs, or in other words if its decomposition $G=\bigsqcup_{i=1}^h G_i$ into connected components satisfies $V(G_i)\in \wh \D(\n)$ for each $i=1,\ldots,h$. 
      
      We let $\S\S_\n(\Gamma)$ be the set of $\n$-admissible spanning subgraphs of $\Gamma$.
      \item 
   The poset of \emph{$\n$-admissible divisors on $\n$-admissible spanning subgraphs} of degree $d$ is the subposet of $\OO^d(\Gamma)$ given by 
    \begin{equation*}
        \OO^d_\n(\Gamma):=\left\{(G=\Gamma\setminus S,D)\in\OO^d(\Gamma): 
        \begin{aligned}
        & G=\bigsqcup_{i=1}^hG_i  \textup{ is $\n$-admissible and } \\
        & D_{V(G_i)}+e_S(V(G_i))=\n_{V(G_i)}
        \end{aligned} \right\}.
    \end{equation*}
    \end{enumerate}
\end{definition}

\begin{remark}\label{R:Ocon}
Let $\n$ be a V-stability of degree $d$ on $\Gamma$.
\begin{enumerate}
    \item Since $\wh \D(\n)$ contains $V(\Gamma)$ by Definition \ref{D:ext-n}, we deduce that 
    $$\S\S_{\con}(\Gamma)\subseteq \S\S_\n(\Gamma)\quad \text{ and } \quad \OO^d_{\con}(\Gamma)\subseteq \OO_\n^d(\Gamma).
    $$
    \item The subset $\OO^d_\n(\Gamma)$ is an upper subset of $\OO^d(\Gamma)$ as it follows easily from the fact that if $(G,D)\geq (G',D')$, then the vertex set of any connected component of $G$ is the disjoint union of the vertex sets of some of the connected components of $G'$.
\end{enumerate}     
\end{remark}

\begin{proposition}\label{P:n-ss}
Let $\n$ be a V-stability of degree $d$ on $\Gamma$.
\begin{enumerate}
 \item \label{P:n-ss1} The subset $\P_\n$  is a finite upper subset of $\OO^d(\Gamma)$.
 \item \label{P:n-ss2} The subset $\P_\n$ is contained in $\OO^d_{\n}(\Gamma)$.
\end{enumerate}
\end{proposition}
\begin{proof}
Part \eqref{P:n-ss1}: the fact that $\P_\n$ is finite follows from the fact that, for
 $(\Gamma\setminus S, D)\in \P_\n$, we have that 
$$
D_Z\in [\n_Z-e_S(Z), \n_Z-e_S(Z)+\val_{\Gamma\setminus S}(Z)] \quad \text{ for any } Z\in \BCon(\Gamma),
$$
as it follows from the definition of $\P_\n$ and Remark \ref{R:prop-Pn}\eqref{R:prop-Pn1}.

Let us now show that  $\P_\n$ is an upper subset. Pick two ordered elements $(G=\Gamma\setminus S,D)\geq_{\O} (G'=\Gamma\setminus S',D')$ of $\OO^d(\Gamma_X)$. Then $\O$ is an orientation of the edges of $S'\setminus S$ and $D=D'+D(\O)$. Fix $Z\in \BCon(\Gamma)$ and observe that for any oriented edge $\ee\in \O$ whose underlying edge $e$ belongs to $E(\Gamma[Z])$, we have that $t(\ee)\in Z$. Hence we have that 
$$
D_Z+e_S(Z)=D'_Z+D(\O)_Z+e_S(Z)\geq D_Z'+e_{S'}(Z).
$$
From the above inequality, we infer that if $(G',D')\in \P_\n$ then $(G,D)\in \P_\n$, which shows that $\P_\n$ is an upper subset.

Part \eqref{P:n-ss2}: take $(G=\Gamma\setminus S, D)\in \P_\n$ and consider the decomposition $G=\bigsqcup_{i=1}^h G_i$ into connected components. 
If $G$ is connected then the result is obvious by Remark \ref{R:Ocon}.
Otherwise, take the vertex set $V(G_i)$ of one of the connected components of $G$
and denote by 
$$\Gamma[V(G_i)^\mathsf{c}]=\Gamma[Z_1]\coprod \ldots \coprod \Gamma[Z_k]$$ 
the decomposition of $\Gamma[V(G_i)^\mathsf{c}]$ into connected components 
(so that $Z_i$ is  biconnected for every $1\leq i \leq k$). 
By applying the second equation of \eqref{E:connW} to $W=V(G_i)$ and using that 
$\val(V(G_i))=\val_S(V(G_i))$ since $\val_{\Gamma\setminus S}(V(G_i))=0$, we deduce that 
$$
\begin{sis}
    & Z_i \text{ is $\n$-degenerate for any }1\leq i\leq k \Rightarrow V(G_i)\in \wh D(\n), \\
    & D_{V(G_i)}+e_S(V(G_i))=d-\sum_{i=1}^k \n_{Z_i}-\val(V(G_i))=\n_{V(G_i)}.
\end{sis}
$$
This shows that  $(G=\Gamma\setminus S, D)\in \OO_\n^d(\Gamma)$.
\end{proof}

\begin{remark}
It follows from the proof of the above Proposition that  $\P_\n^{st}$ is also an upper subset. However, the same is not true of $\P^{ps}_\n$. 
\end{remark}

\subsection{Operations on V-stability conditions}\label{Sub:OperV}

In this subsection, we define two operations on V-stability conditions: the push-forward along a morphism and the restriction to admissible connected subgraphs (generalizing \cite[Lemma-Definitions 1.6 and 1.7]{viviani2023new} for general V-stability conditions). 

Let us begin with the push-forward of a V-stability along a morphism of graphs. 

\begin{lemma-definition}\label{LD:V-func}
Let $\n$ be a V-stability of degree $d$ on a connected graph $\Gamma$. Let $f:\Gamma\to \Gamma'$ be a morphism of graphs and 
denote by $f_V:V(\Gamma)\twoheadrightarrow V(\Gamma')$ the induced surjection on vertices. 
The \emph{push-forward} of $\n$ along $f$ is the V-stability on $\Gamma'$ given by 
\begin{equation}\label{E:V-func}
f_*(\n):=\{f_*(\n)_Z:=\n_{f_V^{-1}(Z) }  \: : Z\in \BCon(\Gamma')\}.
 \end{equation}
We have that 
\begin{itemize}
    \item $| f_*(\n)|=|\n|$;
    \item $\D(f_*(\n))=\{ Z\in \BCon(\Gamma')\: : \: f_V^{-1}(Z) \text{ is $\n$-degenerate}\}$;
    \item $\wh{\D}(f_*(\n))=\{ Z\in \BCon(\Gamma')\: : \: f_V^{-1}(Z) \text{ is $\n$-degenerate}\}$;
    \item  $f_*(\n)_Z:=\n_{f_V^{-1}(Z)}$ for any $Z\subset V(\Gamma')$. 
\end{itemize} 
\end{lemma-definition}
Note that the above definition of $f_*(\n)$ makes sense since if 
$Z\in \BCon(\Gamma')$ then $f_V^{-1}(Z)\in \BCon(\Gamma)$.
\begin{proof}
The fact that $f_*(\n)$ is a V-stability condition of degree $|\n|$ and the description of $\D(f_*(\n))$ follow from the fact that if $Z_1, Z_2\in \BCon(\Gamma')$ are pairwise disjoint then $f_V^{-1}(Z_1), f_V^{-1}(V_2)$ are pairwise disjoint and $|E_{\Gamma}(f_V^{-1}(Z_1), f_V^{-1}(Z_2))|=|E_{\Gamma'}(Z_1, Z_2)|$. 

The last two assertions follow from the fact if $Z\in \BCon(\Gamma')$ (resp. $Z\in \Con(\Gamma')$) then $f_V^{-1}(Z)\in \BCon(\Gamma)$ (resp. $f_V^{-1}(Z)\in \Con(\Gamma)$).
\end{proof}

\begin{example}
Let $\phi$ a numerical polarization of degree $d$ on $\Gamma$ and let $f:\Gamma\to \Gamma'$ be a morphism. Consider the numerical polarization $f_*\phi$ of degree $d$ on $\Gamma'$ given by 
$$
(f_*\phi)_Z:=\phi_{f_V^{-1}(Z)} \quad \text{for any } Z\subseteq V(\Gamma').
$$
Then we have that 
$$
f_*\n(\phi)=\n(f_*\phi).
$$
\end{example}

    The semistable sets behave well with respect to the push-forward on the poset of divisors on spanning subgraphs in \eqref{E:f*-O} and the push-forward of V-stabilities defined in Lemma-Definition \ref{LD:V-func}.

\begin{proposition}\label{P:Pn-func}
Let $\mathfrak n$ be a V-stability on a connected graph $\Gamma$ and let $f:\Gamma\to \Gamma'$ be a morphism of graphs. 
Then the push-forward $f_*$ of \eqref{E:f*-O} is such that  
$$
f_*(\P_{\n})\subseteq  \P_{f_*(\n)}.
$$
\end{proposition}
\begin{proof}
The proof given in \cite[Prop. 1.11]{viviani2023new} for general V-stabilities extends verbatim. 
\end{proof}

Next, we show that a V-stability on a  graph  induces a natural V-stability on each connected $\n$-admissible  subgraph, as in Definition \ref{D:adm-n}.

\begin{lemma-definition}\label{LD:V-subgr}
Let $\n$ be a V-stability of degree $d$ on a connected graph $\Gamma$. Let $G$ be a connected $\n$-admissible subgraph of $\Gamma$ and write $G=\Gamma[V(G)]\setminus S$.  

Then the \emph{restriction of $\n$ to $G$} is a V-stability condition $\n(G)$  on $G$  defined by 
\begin{equation}\label{E:nG}
\n(G)_W:=\n_{W}-e_S(W) \text{ for any } W\in \BCon(G).
\end{equation}
We have that 
\begin{itemize}
    \item $|\n(G)|=\n_{V(G)}-|S|$;
    \item $\wh \D(\n(G))=\wh \D(\Gamma)\cap \Con(G)$. 
\end{itemize} 
\end{lemma-definition}
\begin{proof}
If $G=\Gamma[V(G)]$ (i.e. if $S=\emptyset$), then the proof is analogous to the proof of \cite[Lemma-Definition 4.8]{FPV1} (mutatis mutandis). 

The general case follows from the fact that $\n(G)=\n(H)(H\setminus S)$, where $H:=\Gamma[V(G)]$, and the proof that this restriction is well-defined is analogous to the proof given in \cite[Lemma-Definiion 1.6]{viviani2023new} for general V-stabilities. 
\end{proof}

\begin{example}
 Let $\phi$ be a numerical polarization of degree $d$ on $\Gamma$.
A connected subgraph $G=\Gamma[V(G)]\setminus S$ of $\Gamma$ is $\n(\phi)$-admissible if and only if the decomposition into connected components $\Gamma[V(G)^\mathsf{c}]=\Gamma[Z_1]\coprod \ldots \coprod \Gamma[Z_k]$ satisfies the condition 
$$
\phi_{Z_i}-\frac{\val(Z_i)}{2}\in \ZZ \text{ for every } 1\leq i \leq k,
$$
which then implies that 
$$
\phi_{V(G)}-\frac{\val(V(G))}{2}\in \ZZ.$$
Consider the function given by 
$$
\phi(G)_Z:=\phi_Z-e_S(Z)-\frac{\val(Z,V(G)^\mathsf{c})}{2},
$$
 for any $Z\subset V(G)$. 
 Then $\phi(G)$ is a numerical polarization on $G$ of degree 
 $$
 |\phi(G)|=\phi_{V(G)}-\frac{\val(V(G))}{2}-|S|
 $$
 such that 
 $$
 \n(G)=\n(\phi(G)).
 $$
\end{example}

For any spanning admissible subgraph, the semistable sets behave well with respect to the inclusion  \eqref{E:incl-O} and the restriction of V-stabilities defined in Lemma-Definition \ref{LD:V-subgr}. 

\begin{proposition}\label{P:Pn-rest}
  Let $\n$ be a V-stability of degree $d$ on a connected graph $\Gamma$. Let  $G=\Gamma\setminus S$ be a $\n$-admissible spanning subgraph of $\Gamma$ and let $G=\coprod_{i=1}^ h G_i$ its deomposition into connected components. 
  Then the inclusion $\iota_G$ of \eqref{E:incl-O} is such that 
  $$
  (\iota_G)^{-1}(\P_\n)=\P_{\n(G_1)}\times \ldots \times \P_{\n(G_h)},
  $$
with respect to the inclusion 
$$
\begin{aligned}
    \P_{\n(G_1)}\times \ldots \times \P_{\n(G_h)} & \hookrightarrow \OO^{d-|S|}(G),\\
    ((G_1\setminus S_1,D_1),\ldots, (G_h\setminus S_h,D_h))& \mapsto 
    (G\setminus (S_1\cup \ldots \cup S_h),D_1+\ldots+D_h).
\end{aligned}
$$
\end{proposition}
\begin{proof}
  First of all, we make two reductions.

\un{Reduction 1:} it is enough to prove that 
\begin{equation}\label{E:redu1}
(\iota_G)^{-1}(\P_\n(G))=\P_{\n(G_1)}(G_1)\times \ldots \times \P_{\n(G_h)}(G_h). 
\end{equation}

Indeed, the original statement follows from equality \eqref{E:redu1} applied to any spanning subgraph of $G$, since any such spanning subgraph is of the form   $\coprod_{i=1}^h G_i\setminus S_i$.

 \un{Reduction 2:} it is enough to prove that, given divisors $D_i\in \Div^{d_i}(G_i)$ with $d_i:=|\n(G_i)|$ with $1\leq i \leq h$, we have that 
 \begin{equation}\label{E:redu2}
  (G,D=\sum_{i=1}^h D_i) \text{ is $\n$-semistable } \Longleftrightarrow 
  (G_i,D_i) \text{ is $\n(G_i)$-semistable for any } 1\leq i \leq h. 
 \end{equation}

This follows from the fact that 
$$
(\iota_G)^{-1}(\P_\n)\subseteq \OO^{d_i}(G_1)\times \ldots \times \OO^{d_h}(G_h),
$$
which is implied by Proposition \ref{P:n-ss}\eqref{P:n-ss2}.

\vspace{0.1cm}

We now prove \eqref{E:redu2}, by showing the two implications separately. 

\un{Proof of $\Longrightarrow$ in \eqref{E:redu2}:}
Let $W\in \BCon(G_i)\subseteq \Con(G)$. By Remark~\ref{R:prop-Pn}(\ref{R:prop-Pn2}) and Lemma-Definition~\ref{LD:V-subgr}, we have 
$$
(D_i)_W=D_W\geq \n_W-e_S(W)=\n(G_i)_W.
$$

\un{Proof of $\Longleftarrow$ in \eqref{E:redu2}:}
Let $W\in \BCon(\Gamma)$ and set $W_i:=W\cap V(G_i)$. For each $W_i$, consider the decomposition $G_i[W_i]=\coprod_j G_i[W_i^j]$ into connected components. Again, by Remark~\ref{R:prop-Pn}(\ref{R:prop-Pn2}) and Lemma-Definition~\ref{LD:V-subgr}, we obtain 

\begin{equation*}
D_W=\sum_{i,j} (D_i)_{W_i^j}\geq \sum_{i,j} [\n_{W_i^j}-e_S(W_i^j)]
\geq \n_W-\sum_{(i,j)\neq (i',j')} \val_{\Gamma}(W_i^j,W_{i'}^{j'})-\sum_{i,j} e_S(W_i^j)=
\n_W-e_S(W),
\end{equation*}
where we have used Remark \ref{R:n-deg}\eqref{R:n-deg1} in the second inequality 
and that $\val_{\Gamma}(W_i^j,W_{i'}^{j'})=\val_{S}(W_i^j,W_{i'}^{j'})$ in the last equality.
\end{proof}

\section{PT-assignments}\label{Sec:PT}

In this section we introduce the Picard Type assignments, shortened in PT-assignments (Definition~\ref{D:PTstab}). We then prove (Theorem~\ref{T:V=PT}) that they are exactly the semistable sets of the V-stability conditions (from Definition~\ref{D:Pn}). 

We begin by  introducing the notion of degeneracy subsets and their associated admissible subgraphs. 

\begin{definition}\label{D:Deg}
Let $\Gamma$ be a connected graph. 
\begin{enumerate}
\item A \emph{degeneracy subset} for $\Gamma$ is a subset $\D\subseteq \BCon(\Gamma)$ satisfying the following two properties:
\begin{itemize}
    \item $V\in \BCon(\Gamma)\Rightarrow V^\mathsf{c}:=V(\Gamma)\setminus V\in \D$;  
\item 
for all $V_1, V_2 \in \BCon(\Gamma)$ such that $V_1 \cap V_2=\emptyset$ and $V_1\cup V_2\in \BCon(\Gamma)$, we have
\[
 V_1, V_2 \in \D \Rightarrow V_1\cup V_2\in \D.\]
\end{itemize}
We denote by $\Deg(\Gamma)$ the poset of degeneracy subsets of $\Gamma$, ordered by inclusion. 
\item  The \emph{extended degeneracy subset} associated to $\mathcal{D}\in \Deg(\Gamma)$ is the subset $\widehat{\mathcal{D}}\subseteq \Con(\Gamma)$ defined as
\[
\wh \D:=\Bigg\{V \in \Con(\Gamma): \  \Gamma[V^\mathsf{c}] = \bigsqcup_{i \in I} \Gamma[V_i] \textrm{ with all } V_i \in \mathcal{D}\Bigg\} \cup \{V(\Gamma)\}.
\]
\end{enumerate}
\end{definition}
From the two properties of a degeneracy subset, it follows that, given $\D\in \Deg(\Gamma)$, for all $V_1, V_2 \in \BCon(\Gamma)$ such that $V_1 \cap V_2=\emptyset$ and $V_1\cup V_2\in \BCon(\Gamma)$,  if two among the subsets  $\{V_1, V_2,V_1\cup V_2\}$ are in  $\D$ then also the third one is in $\D$.

\begin{remark}
 \noindent 
 \begin{enumerate}[(i)]
     \item $\Deg(\Gamma)$ has a minimal element, namely $\emptyset$, and a maximal element, namely $\BCon(\Gamma)$.  
     \item The extended degeneracy subset $\wh \D$ of $\D\in \Deg(\Gamma)$ completely determines $\D$ since we have that $\D=\wh \D\cap \BCon(\Gamma)$.
     \item Given a V-stability $\n$ on $\Gamma$, we have that $\D(\n)$ is a degeneracy subset by Remark \ref{R:n-deg}\eqref{R:n-deg1} and $\wh \D(\n)$ is the extended degeneracy subset associated to $\D(\n)$ via Definition \ref{D:ext-n}.
 \end{enumerate}
\end{remark}

\begin{definition}\label{D:Dadmissible}
Let ${\D}$ be a degeneracy subset for $\Gamma$. 
\begin{enumerate}
      \item A connected subgraph $G$ of $\Gamma$ is said to be  \emph{$\D$-admissible} if $V(G)\in \wh \D$.
      
      \item A spanning subgraph $G$ of $\Gamma$ is said to be \emph{$\D$-admissible} if its decomposition $G=\bigsqcup_{i\in I} G_i$ into connected components satisfies $V(G_i)\in \wh \D$ for each $i \in I$. 

      We denote by $\S\S_\D(\Gamma)$ the poset (with respect to inclusions) of \emph{$\D$-admissible spanning subgraphs of $\Gamma$}, by $\S\F_\D(\Gamma)$ the set of \emph{$\D$-admissible spanning forests of $\Gamma$} and by $m\S\S_\D(\Gamma)$ the minimal elements of $\S\S_\D(\Gamma)$.
      Observe that $m\S\S_\D(\Gamma)\subseteq \S\F_\D(\Gamma)$.

\item For any $G\in \S\S_{\D}(\Gamma)$,  the cardinality of the $\D$-admissible spanning forests contained in $G$ will be denoted by $c_\D(G)$ and called the \emph{$\D$-complexity}  of $G$.

\item   The poset of \emph{divisors on $\D$-admissible spanning subgraphs} of degree $d$ is the subposet of $\OO^d(\Gamma)$ given by 
    \begin{equation*}
        \OO^d_\D(\Gamma):=\left\{(G=\Gamma\setminus S,D)\in\OO^d(\Gamma): 
         G \textup{ is $\D$-admissible}  \right\}.
    \end{equation*}
   \end{enumerate}
\end{definition}

\begin{remark}\label{R:Dadm} 
Let $\D$ be a degeneracy subset for $\Gamma$. 
\begin{enumerate}
    \item Since $\wh \D$ contains $V(\Gamma)$, we deduce that 
$$\S\S_{\con}(\Gamma)\subseteq \S\S_\D(\Gamma) \subseteq \S\S(\Gamma)\quad \text{(which implies that }  \OO^d_{\con}(\Gamma)\subseteq \OO_\D^d(\Gamma)\subseteq \OO^d(\Gamma)).$$
The above two inclusions are equalities in the two extreme cases: 
\begin{itemize}
\item $\S\S_{\con}(\Gamma)= \S\S_\D(\Gamma)\Leftrightarrow \D=\emptyset$, in which case $c_{\emptyset}(\Gamma)$ is the complexity $c(\Gamma)$ of $\Gamma$, i.e. the number of spanning trees of $\Gamma$.
\item $\S\S_{\D}(\Gamma)= \S\S(\Gamma)\Leftrightarrow \wh\D\supset V(\Gamma)$, in which case $c_{\D}(\Gamma)$ is the number of spanning forests of $\Gamma$.
\end{itemize}

    \item The subset $\OO^d_\D(\Gamma)$ is an upper subset of $\OO^d(\Gamma)$ as it follows easily from the fact that if $(G,D)\geq (G',D')$, then the vertex set of any connected component of $G$ is the disjoint union of the vertex sets of some of the connected components of $G'$.
    \item If $\D=\D(\n)$ for some V-stability $\n$ on $\Gamma$, then we have that 
     $\S\S_{\D(\n)}(\Gamma)=\S\S_\n(\Gamma)$ and $\OO_\n^d(\Gamma)\subseteq \OO_{\D(\n)}^d(\Gamma)$ (although this last inclusion is in general strict).
\end{enumerate}     
\end{remark}

For future use, we now observe that there are two natural operations on degeneracy subsets: the push-forward along a graph morphism and the restriction to an admissible connected subgraph. 

\begin{lemma-definition}\label{LD:operD}
Let $\D$ be a degeneracy subset for $\Gamma$.
\begin{enumerate}
    \item \label{LD:operD1}
    Let  $f:\Gamma\to \Gamma'$ be a morphism of graphs.  The pushforward of $\D$ along $f$ is  
    $$
    f_*\D:=\{Z\in \BCon(\Gamma'):f^{-1}_V(Z)\in\D\}\in \Deg(\Gamma').
    $$
    and its associated extended degeneracy subset is 
    $$
    f_*\wh\D:=\{Z\in \Con(\Gamma'):f^{-1}_V(Z)\in\wh\D\}.
    $$
\item \label{LD:operD2}
    Let $G\leq\Gamma$ be a $\D$-admissible connected subgraph. The restriction of $\D$ to $G$ is  
    $$
    \D(G):=\{Z\in\BCon(G):Z\in\wh \D\}\in \Deg(G)
    $$
    and its associated extended degeneracy subset is
    $$
    \wh \D(G):=\{Z\in\Con(G):Z\in\wh \D\}.
    $$
\end{enumerate}
\end{lemma-definition}
\begin{proof}
Part \eqref{LD:operD1} follows from the fact that $f_V$ is surjective and that 
if $Z\in \BCon(\Gamma')$ (resp. $Z\in \Con(\Gamma')$) then $f_V^{-1}(Z)\in \BCon(\Gamma)$ (resp. $f_V^{-1}(Z)\in \Con(\Gamma)$).

Part \eqref{LD:operD2}: consider the decomposition into connected components $\Gamma[V(G)^\mathsf{c}]=\coprod_{i=1}^k \Gamma[Z_i]$.  For any $W\in \BCon(G)$ there exists at least one subset  $I\subseteq \{1,\ldots, k\}$ such that $Z_I\coprod W:=\coprod_{i\in I} Z_i\coprod W\in \BCon(\Gamma)$. 
From the properties of $\D$ and $\wh \D$, it follows that  
\begin{equation}\label{E:2def-DG}
    \D(G)=\{ W\in \BCon(G)\: : \: Z_I\coprod W \text{ is biconnected and $\D$-admissible}\},
    \end{equation}
    where the above condition is independent of the choosen $I$ such that $Z_I\coprod W\in \BCon(\Gamma)$. 
Then we conclude that $\D(G)$ is a degenerate subset for $G$ using that $\D$ is a degenerate subset for $\Gamma$ and that 
\begin{itemize}
    \item $(Z_I\coprod W)^\mathsf{c}=Z_{I^\mathsf{c}}\coprod W^\mathsf{c}\in \BCon(\Gamma)$;
\item For $W_1,W_2,W_3\in \BCon(G)$ pairwise disjoint such that $W_1\coprod W_2\coprod W_3=V(G)$, we can choose three pairwise disjoints subsets $I_1,I_2,I_3\subset [k]$ with $I_1\cup I_2\cup I_3=[k]$ such that $Z_{I_i}\coprod W_i\in \BCon(\Gamma)$ for any $1\leq i \leq 3$, which then implies that $\{Z_{I_i}\coprod W_i\}_{1\leq i \leq 3}$ are pairwise disjoint and their union is the entire vertex set $V(\Gamma)$.
\end{itemize}
\end{proof}

\begin{remark}
    If $\D=\D(\n)$ for some V-stability $\n$ on $\Gamma$, then 
    \begin{itemize}
        \item for any morphism $f:\Gamma\to \Gamma'$, we have that $f_*\D(\n)=\D(f_*\n)$ by Lemma-Definition \ref{LD:V-func};
        \item for any $\D(\n)$-admissible connected subgraph $G\leq \Gamma$, we have that $\D(\n)(G)=\D(\n(G))$ by Lemma-Definition \ref{LD:V-subgr} and \eqref{E:2def-DG}.
    \end{itemize}
\end{remark}

The behaviour of the poset of divisors on admissible spanning subgraphs with respect to the above two operations is explained in the following

\begin{lemma}\label{L:operOD}
Let $\D$ be a degeneracy subset for $\Gamma$.
\begin{enumerate}
    \item \label{L:operOD1}
    Let  $f:\Gamma\to \Gamma'$ be a morphism of graphs.  Then the push-forward map \eqref{E:f*-O} is such that 
    $$
    f_*(\OO_{\D}^d(\Gamma))=\OO^d_{f_*(\D)}(\Gamma'). 
    $$
\item  \label{L:operOD2}  Let $G\leq\Gamma$ be a $\D$-admissible spanning subgraph and let $G=\coprod_i G_i$ be its decomposition into connected components. Then 
\begin{enumerate}[(i)]
\item \label{L:operOD2i} we have the equality 
$$
\S\S_{\D}(\Gamma)\cap \S\S(G)=\prod_i \S\S_{\D(G_i)}(G_i)=:\S\S_{\D}(G).$$
In particular, $c_{\D}(G)=\prod_i c_{\D(G_i)}(G_i)$.
\item \label{L:operOD2ii} the natural inclusion \eqref{E:incl-O} is such that 
$$
\iota_G^{-1}(\OO^d_{\D}(\Gamma))=\coprod_{\un d=(d_i)} \prod_i \OO^{d_i}_{\D(G_i)}(G_i)=:\OO^{d-|E(G)^\mathsf{c}|}_{\D}(G)
$$
where $\un d=(d_i)$ vary among all tuples such that $\sum_i d_i=d-|E(G)^\mathsf{c}|$. 
\end{enumerate}
\end{enumerate}
\begin{proof}
    Part \eqref{L:operOD1} follows from the fact that, for any $\Gamma\setminus S\in \S\S(\Gamma)$, the pull-back of the vertex set of any connected component of $f_*(\Gamma\setminus \S)$ is a union of vertex sets of connected components of $\Gamma\setminus S$.

    Part \eqref{L:operOD2} follows from the fact that any spanning subgraph $H\leq G$ can be written as $H=\coprod_i H_i$ where $H_i$ is spanning subgraphs of $G_i$, and that $H$ is $\D$-admissible if and only if each $H_i$ is $\D(G_i)$-admissible.
\end{proof}

\end{lemma}

We now introduce some notation in view of the main definition of this section.

\begin{notation} \label{N:NotationPT} 
\begin{enumerate} 
\noindent
\item  \label{N:Vbullet} Let $V_\bullet=(V_1,\ldots, V_k)$ be an ordered partition of $V(\Gamma)$, i.e. $V_i \cap V_j = \emptyset$ whenever $i \neq j$ and $\bigcup_{i=1}^k V_i = V(\Gamma)$. Setting  $E_G(V_\bullet)=\bigcup_{i=1}^k E_G(V_i, V(G) \setminus V_i)$, each such $V_\bullet$ induces a partial orientation $\mathcal{O}_G({V_\bullet})$ on $E_G( V_\bullet)$ defined by $s(\ee) \in V_i$ and $t(\ee) \in V_j$ if and only if $i<j$. 
Given an ordered partition $V_\bullet=(V_1,\ldots,V_k)$, we denote by $\overline{V}_\bullet:=(V_k,\ldots,V_1)$   the opposite ordered partition.

\item \label{N:subposet} Let $\P$ be an upper subset of $\OO^d_\D(\Gamma)$, for some degeneracy subset $\D$, and let $G\in\S\S_\D(\Gamma)$. We define the full subposet $S_G(\P)$ of $\S\S_{\D}(G)$, whose objects are subgraphs $G'\leq G$ such that:
        \begin{itemize}
            \item there exists an ordered partition $W_\bullet$ of $V(G)$ with $G'=G\setminus E(W_\bullet)$;
            \item $\P(G')\neq\emptyset$.
        \end{itemize} 

\item \label{N:PicDisc} If $G= \coprod_{i \in I(G)} G_i$ is the decomposition in connected components, we note that $\Pic^d(G)= \prod_{\sum_i d_i=d} \Pic^{d_i}(G_i)$. If $\delta : I(G) \to \mathbb{Z}$ satisfies $\sum_{i \in I(G)} \delta(i)=d$, we write $\Pic^\delta(G)$ for the subset $\prod_{i \in I(G)} \Pic^{\delta(i)} (G_i)$ of $\Pic^d(G)$.
\end{enumerate}
\end{notation}

\begin{definition}\label{D:PTstab}
    Let $\D\in\Deg(\Gamma)$, and let $\P$ be an upper subset of $\OO^d_\D(\Gamma)$. 
    We say that $\P$ is:
    
    \begin{enumerate}
        \item\label{D:PTstab1} A \emph{Picard Type assignment}, shortened to \emph{PT-assignment}, of degree $d$ on $\Gamma$ if for any $G\in\S\S_{\D}(\Gamma)$, the homomorphism 
    \begin{equation*}
        \pi_G:\P(G):=\{D\in\textup{Div}^{d-|E(G)^\mathsf{c}|}(G):(G,D)\in\P\} \to \textup{Pic}^{d-|E(G)^\mathsf{c}|}(G)
    \end{equation*}
        satisfies the following properties:
        
        \begin{enumerate}
            \item\label{D:PTstabProp1} it surjects onto $\Pic^{\delta_G}(G)$ for some $\delta_G : I(G) \to \mathbb{Z}$ such that $\sum_{i \in I(G)} \delta_G(i)= d- |E(G)^\mathsf{c}|$;
        \item\label{D:PTstabProp2} for each $D_1,D_2\in\P(G)$ such that $\pi_G(D_1)=\pi_G(D_2)$, there exists an ordered partition  $V_\bullet=(V_1,\ldots,V_k)$ of $V(\Gamma)$, where each $V_i$ is a disjoint union of elements of $\D$, such that the equality 
        \[D_1-D(\O_G(V_\bullet))=D_2-D(\O_G(\overline{V_\bullet}))\]
        holds, and the latter divisor belongs in $\P(G\setminus E_G(V_\bullet))$.
        \end{enumerate}
        \item A \emph{numerical PT-assignment} of degree $d$ on $\Gamma$ if for any $G\in\S\S_\D(\Gamma)$ we have that 
        \begin{equation*}
            |\P(G)|=c_\D(G).
        \end{equation*}
        \item A \emph{weak PT-assignment} of degree $d$ on $\Gamma$ if for any $G\in S_\Gamma(\P)$, the homomorphism 
    \begin{equation*}
        \pi_G:\P(G):=\{D\in\textup{Div}^{d-|E(G)^\mathsf{c}|}(G):(G,D)\in\P\} \to \textup{Pic}^{d-|E(G)^\mathsf{c}|}(G)
    \end{equation*}
        satisfies Properties~\ref{D:PTstabProp1} and \ref{D:PTstabProp2} of Part~\ref{D:PTstab1}.
        \item A \emph{weak numerical PT-assignment} of degree $d$ on $\Gamma$ if we have the equality 
        \begin{equation*}
            |\P(\Gamma)|=c_\D(\Gamma).
        \end{equation*}
    \end{enumerate}
    In each of the above, we say that $\D$ is the degeneracy set of the assignment $\P$ and we write $\D=\D(\P)$.  The extended degeneracy subset associated to $\D(\P)$ is denoted by $\wh \D(\P)$.
\end{definition}

\begin{remark}\label{R:compare-PT}
\noindent
\begin{enumerate}
    \item  Proposition~\ref{P:cardPTstab} below will imply the following diagram of implications:

\begin{tikzcd}
		\{\textup{PT-assignments}\} \arrow[r,Rightarrow] \arrow[d,Rightarrow] 
		& \{\textup{Weak PT-assignments}\} \arrow[d,Rightarrow] \\
		\{\textup{Numerical PT-assignments}\} \arrow[r,Rightarrow]
		& \{\textup{Weak numerical PT-assignments.}\}
\end{tikzcd}    
\item The above assignments with empty degeneracy subset were considered in \cite[Def. 4.4]{pagani2023stability} and \cite[Def. 1.3]{viviani2023new}:
\begin{itemize}
    \item PT-assignments with empty degeneracy subsets are the same as stability assignments in \cite{pagani2023stability} and upper subsets of strong Neron type in \cite{viviani2023new};
    \item weak PT-assignments with empty degeneracy subsets are the same as upper subsets of  Neron type in \cite{viviani2023new};
    \item numerical PT-assignments with empty degeneracy subsets are the same as upper subsets of numerical strongly Neron type in \cite{viviani2023new};
    \item weak numerical PT-assignments with empty degeneracy subsets are the same as upper subsets of numerical Neron type in \cite{viviani2023new}.
\end{itemize}
\end{enumerate}
\end{remark}

\begin{remark}\label{R:DegreemSS}
If $\P$ is an upper subset of $\OO^d_\D(\Gamma)$ that satisfies Condition~\ref{D:PTstabProp2} of Definition~\ref{D:PTstab}, then for every  $G \in \S \S_\D(\Gamma)$ there exists at most one function $\delta_G^\P : I(G) \to \mathbb{Z}$ such that $\P(G) \to \Pic^{\delta_G^\P}(G)$ is surjective, and the function $\delta_G^\P$ may be inductively  determined on the number $t(G)=|I(G)|$ of connected components of $G$, as follows. First of all, by upper closure, we may reduce to the case where $G$ is a minimal $\D(\P)$-admissible spanning \emph{forest} of $\Gamma$.

If $t(G)=1$ there is nothing to prove: there exists a unique function that sums to $d$ on the (unique) connected component(s).

Assume $t(G)>1$. Because $G$ is spanning, there exists $e \in E(\Gamma) \setminus E(G)$ such that $t(G \cup \{e\})=t(G)-1$. Call $G_1$ and $G_2$ the two connected components of $G$ such that $G_1 \cup G_2 \cup \{e\}$ is connected (hence a tree of the forest $G\cup\{e\}$), and call $v_1 \in G_1$ and $v_2 \in G_2$ the two endpoints of $e$.

By the induction hypothesis, Condition~\ref{D:PTstabProp2} determines a unique sum $d$ function $\delta^\P_{G \cup  \{e\}}$ on the set of connected components of $G\cup \{e\}$. This determines the value $\delta^\P_G(G_i)= \delta^\P_{G \cup \{e\}} (G_i)$ on each connected component $G_i$ of $G$ other then $G_1$ and $G_2$. This also determines the sum $\delta^\P_G(G_1)+ \delta^\P_G(G_2)+1$ to equal $\delta^\P_{G \cup \{e\}}(G_1 \cup G_2 \cup \{e\})$. 

Now assume for a contradiction that there exist $D,E \in \P(G)$ with $D_{G_1} \neq E_{G_1}$, $D_{G_2} \neq E_{G_2}$, and $D_{G_i}=\delta^\P_{G \cup \{e\}}(G_i)= E_{G_i}$ for all $i \neq 1,2$. Without loss of generality, we assume $D_{G_1}<E_{G_1}$, hence $D_{G_2}>E_{G_2}$ (because $D_{G_1}+D_{G_2}=E_{G_1}+E_{G_2}$). 

Let $D' = D+v_2$ and $E'=E+v_1$. Then $D' \neq E'$ and, because $\P$ is upper closed, we have $D',E' \in \P(G \cup \{e\})$, from which we deduce \[D'_{G_1 \cup G_1 \cup \{e\}} = \delta^\P_{G \cup \{e\}} (G_1 \cup G_2 \cup \{e\})= E'_{G_1 \cup G_2 \cup \{e\}}.\] Since they have the same degree, these two divisors are chip-firing equivalent on $G \cup \{e\}$, because the latter is a forest. By Condition~\ref{D:PTstabProp2}, we deduce the existence of a partition $V_\bullet$ of $V(\Gamma)$ such that \[D'-D(\O_{G_1 \cup G_2 \cup \{e\}}({V}_{\bullet}))= E' -D( \O_{G_1 \cup G_2 \cup \{e\}}(\overline{V}_{\bullet})).\] We observe that, since $G\cup\{e\}$ is a spanning forest, there exists only one minimal spanning forest $G_{\text{min}}\leq G\cup\{e\}$. Thus, it must be $G_{\text{min}}=G$, and we deduce that $V_\bullet=(V(G_1),V(G_2),\bigcup_{i \neq 1,2}(G_i))$. Therefore, $D(\O_{G_1 \cup G_2 \cup \{e\}}(V_{\bullet}))=v_2$ and $D(\O_{G_1 \cup G_2 \cup \{e\}}(\overline{V}_{\bullet}))=v_1$, hence $D_{G_i}=E_{G_i}$ for $i=1,2$; a contradiction.
\end{remark}

The ordered partitions that appear in Definition \ref{D:PTstab} are quite special, as we show in the following.

\begin{lemma}\label{L:DO-div}
   Let $V_\bullet=(V_1,\ldots,V_k)$ be an ordered partition of $V(\Gamma)$ and let $G$ be a spanning subgraph of $\Gamma$. Then 
   $$D(\O_G(V_\bullet))-D(\O_G(\ov V_\bullet))\in \Prin(G)$$
   if and only if there exists another ordered partition $W_\bullet=(W_0,\ldots,W_q)$ of $V(\Gamma)=V(G)$ such that 
   \begin{enumerate}[(i)]
   \item \label{L:DO-div1} $\O_G(V_\bullet)=\O_G(W_\bullet)$;
   \item  \label{L:DO-div2} $E_G(W_i,W_j)=\emptyset$ unless $j-i=\pm 1$.
    \end{enumerate}
   In such a  case we have that  
   \begin{equation}\label{E:diff-DO-div}
    D(\O_G(V_\bullet))-D(\O_G(\ov V_\bullet))=\sum_{k=0}^{q-1} \div(\chi_{\bigcup_{i\leq k} W_i})=\sum_{i=0}^q (q-i)\div(\chi_{W_i})
    \end{equation}
\end{lemma}
\begin{proof}
Let us first prove the if condition. Assume that there exists an ordered partition $W_\bullet$ satisfying \eqref{L:DO-div1} and \eqref{L:DO-div2}, and let us check Formula \eqref{E:diff-DO-div}.
We compute the right hand side
$$
\div\left(\sum_{i=0}^q(q-i)\chi_{W_i}\right)=\sum_{\substack{\ee\in \EE(G): \\ s(\ee)\in W_i, t(\ee)\in W_j}}[(q-i)-(q-j)]t(\ee)=
$$
$$=
\sum_{\substack{\ee\in \EE(G): \\ s(\ee)\in W_{j-1}, t(\ee)\in W_j}}t(\ee)-\sum_{\substack{\ee\in \EE(G): \\ s(\ee)\in W_{j+1}, t(\ee)\in W_j}}t(\ee)=D(\O_G(W_\bullet))-D(\O_G(\ov W_\bullet)),
$$
where we have used \eqref{L:DO-div2} in the last two equalities. 
We now conclude using that 
$$
D(\O_G(W_\bullet))-D(\O_G(\ov W_\bullet))=D(\O_G(V_\bullet))-D(\O_G(\ov V_\bullet))
$$
because of Condition \eqref{L:DO-div1}.

We now show the only if condition. Assume that 
\begin{equation}\label{E:DO=div}
 D(\O_G(V_\bullet))-D(\O_G(\ov V_\bullet))=\div(f) \: \text{ for some } f:V(G)\to \ZZ.
\end{equation}
Up to translating $f$ by a constant (which does not change $\div(f)$), we can assume that 
$$
0,q\in \Im(f)\subseteq \{0,1,\ldots, q\} \: \text{ for some } q\geq 0.
$$
Setting $W_k:=f^{-1}(q-k)$ for any $0\leq k\leq q$, we get an ordered partition $W_\bullet=(W_0,\ldots, W_q)$ of $V(G)$. We now compute the two sides of \eqref{E:DO=div}
\begin{equation}\label{E:sinistra}
\div(f)=
\sum_{\substack{\ee\in \EE(G): \\ s(\ee)\in W_i, t(\ee)\in W_j\\ i<j}}(j-i)t(\ee)-\sum_{\substack{\ee\in \EE(G): \\ s(\ee)\in W_i, t(\ee)\in W_j\\ i>j}}(i-j)t(\ee)=\sum_{\substack{\ee\in \EE(G): \\ s(\ee)\in W_i, t(\ee)\in W_j\\ i<j}}(j-i)\partial(\ee).
\end{equation}
\begin{equation}\label{E:destra}
 D(\O_G(V_\bullet))-D(\O_G(\ov V_\bullet))=\sum_{\substack{\ee\in \EE(G): \\ s(\ee)\in V_i, t(\ee)\in V_j\\ i<j}}t(\ee)-\sum_{\substack{\ee\in \EE(G): \\ s(\ee)\in V_i, t(\ee)\in V_j\\ i>j}}t(\ee)=\sum_{\substack{\ee\in \EE(G): \\ s(\ee)\in V_i, t(\ee)\in V_j\\ i<j}}\partial(\ee).
\end{equation}
We now check that $W_\bullet$ satisfies Properties~\eqref{L:DO-div1} and \eqref{L:DO-div2}.

\un{Claim 1:} $\O_G(V_\bullet)=\O_G(W_\bullet)$.

For each $1\leq i\leq k$, we have that $f$ is constant on $V_i$, as $\div(f)$ does not depend on the orientation of the edges in $E(V_i)$.
Hence, for each $i$, there exists a $0\leq j\leq q$ such that $V_i\subseteq W_j$. In particular, the ordered partition $V_\bullet$ is a refinement of $W_\bullet$. Indeed, we observe that, for each pair $i,i'$, if $i<i'$, then $f_{|V_i}\geq f_{|V_{i'}}$. 

Now, let $1\leq i<i'\leq k$ such that there exists $W_j\supseteq V_i\cup V_{i'}$. Equation \eqref{E:sinistra} restricted to $V_i$ shows that $\div(f)_{V_i}$ is independent of $\val(V_i,V_{i'})$. Since this is true for each such pair $i,i'$, we conclude that $\val(V_i,V_{i'})=0$, by \eqref{E:destra}. The claim follows.

\un{Claim 2:} $E_G(W_i,W_j)=\emptyset$ unless $j-i=\pm 1$.

Indeed, using Claim 1 and Assumption~\eqref{E:DO=div}, we have that 
\begin{equation}\label{E:DO=div2}
    D(\O_G(W_\bullet))-D(\O_G(\ov W_\bullet))=\div(f).
\end{equation}
Using \eqref{E:sinistra} and \eqref{E:destra} with $V_\bullet$ replaced by $W_\bullet$, we deduce that \eqref{E:DO=div2} is equivalent to 
\begin{equation}\label{E:van-1cycle}
    \sum_{\substack{\ee\in \EE(G): \\ s(\ee)\in W_i, t(\ee)\in W_j\\ i+1<j}}(j-i-1)\partial(\ee)=0.
\end{equation}
This equality implies  Claim 2. Indeed, by contradiction suppose this is not the case and pick the biggest index $j$ such that there exists $i+1<j$ and $\ee\in \EE(G)$ with $s(\ee)\in W_i$ and $t(\ee)\in W_j$. Then the coefficient of $t(\ee)$ in the right hand side of \eqref{E:van-1cycle} is positive, which is absurd.  

\vspace{0.1cm}

We end the proof observing that \eqref{E:diff-DO-div} holds because of \eqref{E:DO=div} and the fact that $f=\sum_{i=0}^q(q-i)\chi_{W_i}$ by construction. 
\end{proof}

\begin{definition}
    Let $\Gamma$ be a connected graph and let $\P$ be a PT-assignment on $\Gamma$. We say that a pair $(G,D)\in\OO_{\D(\P)}(\Gamma)$ is:
    \begin{itemize}
        \item \emph{$\P$-semistable} if $D\in\P(G)$;
        \item \emph{$\P$-polystable} if it is $\P$-semistable and $\{D\}=\pi_G^{-1}(\pi_G(D))$;
        \item \emph{$\P$-stable} if it is $\P$-polystable and $G$ is connected.
    \end{itemize}
\end{definition}

\subsection{A generalisation of Kirchhoff's theorem}\label{Sub:Kirc}

The main result of this subsection is Proposition~\ref{P:cardPTstab}, which gives a formula for the cardinalities of the sets $\P(G)$ for a PT-assignment $\P$ on $\Gamma$ and any $G\in \S\S_{\D(\P)}(\Gamma)$, in terms of the generalization of the $\D$-complexity of a graph introduced in Definition~\ref{D:Dadmissible}). In order to prove this result, we introduce the following.

\begin{definition}\label{D:totalorderGamma}
    Let $\Gamma$ be a connected graph. We inductively define a sequence of ordered partitions $(V^1_\bullet,\ldots,V^K_\bullet)$ of $V(\Gamma)$ such that $V^i_\bullet$ is a refinement of $V^j_\bullet$ if $i>j$, as follows:

\begin{itemize}
    \item [Step 1.]
     \begin{enumerate}
        \item Fix an arbitrary vertex $v_{(0,1)}\in V(\Gamma)$, and define $V_0:=\{v_{(0,1)}\}$. If $V_0=V(\Gamma)$, the procedure ends.
        \item Let the distance $d$ between two vertices be the number of edges in the shortest path connecting them. Fix an ordered partition $(V_0,V_1,\ldots,V_n)$ of $V(\Gamma)$ such that $v\in V_{i_1}$ if and only if $d(v,v_{(0,1)})=i_1$.
        \item Further partition each $V_{i_1}$ into $(V_{(i_1,1)},\ldots,V_{(i_1,l(i_1))})$, such that $\Gamma(V_i)=\bigsqcup_{j_1=1,\ldots,l(i_1)}\Gamma(V_{(i_1,j_1)})$ is the decomposition in connected components, arbitrarily ordered. Clearly, $V_0=V_{(0,1)}$.
        \item Define $V^1_\bullet:=(V_{(0,1)},V_{(1,1)},\ldots,V_{(1,l(1))},\ldots,V_{(n,l(n))})$. In other words, $V_{(i_1,j_1)}<V_{(i_1',j_1')}$ iff $(i_1,j_1)>_{\textup{lex}}(i_1',j_1')$, for every $i_1,j_1,i_1',j_1'$.
    \end{enumerate}
    \item [Step k.] Assume an ordered partition $V^{k-1}_\bullet=(V_{(i_1,j_1),(i_2,j_2),\ldots,(i_{k-1},j_{k-1})})_{\{(i_1,j_1),(i_2,j_2),\ldots,(i_h,j_h),{1\leq h\leq k-1}\}}$ of $V(\Gamma)$ has been defined, we now define an ordered partition $V^k_\bullet$ as a refinement of $V^{k-1}_\bullet$, as follows:
    
    \begin{enumerate}
        \item For every $(
        i_1,j_1),(i_2,j_2),\ldots,(i_{k-1},j_{k-1})$ such that $(i_{k-1},j_{k-1})\neq (0,1)$:
    \begin{enumerate}
        \item  Fix an arbitrary vertex $v_{(
        i_1,j_1),(i_2,j_2),\ldots,(i_{k-1},j_{k-1}),(0,1)}\in V_{(
        i_1,j_1),(i_2,j_2),\ldots,(i_{k-1},j_{k-1})}$, and define $V_{(
        i_1,j_1),(i_2,j_2),\ldots,(i_{k-1},j_{k-1}),0}:=\{v_{(
        i_1,j_1),(i_2,j_2),\ldots,(i_{k-1},j_{k-1}),(0,1)}\}$.  If all the elements in $V^k$ are singletons, i.e. they are of the form $V_{(
        i_1,j_1),(i_2,j_2),\ldots,(i_{h-1},j_{h-1}),(0,1)}$, for some $1\leq h\leq k$, the procedure ends with $K:=k$.

        \item  Let
        $$n_k:=\textup{max}_{v\in V_{(
        i_1,j_1),(i_2,j_2),\ldots,(i_{k-1},j_{k-1})}}\{d(v,v_{(
        i_1,j_1),(i_2,j_2),\ldots,(i_{k-1},j_{k-1}),(0,1)})\}.$$ (Note that $n_k$ also depends upon $(
        i_1,j_1),(i_2,j_2),\ldots,(i_{k-1},j_{k-1})$). Fix an ordered partition 
        $$
        (V_{(
        i_1,j_1),(i_2,j_2),\ldots,(i_{k-1},j_{k-1}),0},V_{(
        i_1,j_1),(i_2,j_2),\ldots,(i_{k-1},j_{k-1}),1},\ldots,V_{(
        i_1,j_1),(i_2,j_2),\ldots,(i_{k-1},j_{k-1}),n_k})$$ of $V_{(
        i_1,j_1),(i_2,j_2),\ldots,(i_{k-1},j_{k-1})}$ such that $v\in V_{(
        i_1,j_1),(i_2,j_2),\ldots,(i_{k-1},j_{k-1}),i_k}$ if and only if $d(v,v_{(
        i_1,j_1),(i_2,j_2),\ldots,(i_{k-1},j_{k-1}),(0,1)})=i_k$, where $0\leq i_k\leq n_k$.
        \item Further partition each $V_{(
        i_1,j_1),(i_2,j_2),\ldots,(i_{k-1},j_{k-1}),i_k}$ into\\ $(V_{(
        i_1,j_1),(i_2,j_2),\ldots,(i_{k-1},j_{k-1}),(i_k,1)},\ldots,V_{(
        i_1,j_1),(i_2,j_2),\ldots,(i_{k-1},j_{k-1}),(i_k,l(i_k))})$, such that $$\Gamma(V_{(
        i_1,j_1),(i_2,j_2),\ldots,(i_{k-1},j_{k-1}),i_k})=\bigsqcup_{j_k=1,\ldots,l(i_k)}\Gamma(V_{(
        i_1,j_1),(i_2,j_2),\ldots,(i_{k-1},j_{k-1}),(i_k,j_k)})$$ is the decomposition into connected components, arbitrarily ordered. Clearly, we have $V_{(
        i_1,j_1),(i_2,j_2),\ldots,(i_{k-1},j_{k-1}),0}=V_{(
        i_1,j_1),(i_2,j_2),\ldots,(i_{k-1},j_{k-1}),(0,1)}$.
    \end{enumerate}
        \item Let $$V^k_\bullet:=(V_{(
        i_1,j_1),(i_2,j_2),\ldots,(i_{h},j_{h})})_{\{(
        i_1,j_1),(i_2,j_2),\ldots,(i_{h},j_{h}),1\leq h\leq k\}}$$ be the ordered partition of $V(\Gamma)$, where the ordering is defined by $V_{(
        i_1,j_1),(i_2,j_2),\ldots,(i_{h},j_{h})}<V_{(
        i'_1,j'_1),(i'_2,j'_2),\ldots,(i'_{h'},j'_{h'})}$ iff $(
        i_1,j_1),(i_2,j_2),\ldots,(i_{h},j_{h})>_{\textup{lex}}(
        i'_1,j'_1),(i'_2,j'_2),\ldots,(i'_{h'},j'_{h'})$, where $>_{\textup{lex}}$ is the lexicographic order on the string of indices.
    \end{enumerate}
\end{itemize}
    This sequence of ordered partitions induces a total order on  $V(\Gamma)$. Indeed, for each $v\in V(\Gamma)$, there exists a unique set $V_{(
        i_1,j_1),(i_2,j_2),\ldots,(i_{h-1},j_{h-1}),(0,1)}\in V^h$, for some $1\leq h\leq K$, such that $\{v\}=V_{(
        i_1,j_1),(i_2,j_2),\ldots,(i_{h-1},j_{h-1}),(0,1)}$. We associate the vector $\un v:=(
        i_1,j_1,i_2,j_2,\ldots,i_{h-1},j_{h-1})$ to the vertex $v$. We define the \emph{total order $\leq$ on $V(\Gamma)$} as follows. For every $v,w\in V(\Gamma)$, 
    $$
    v\leq w\quad \text{if and only if}\quad \un v\geq_{\textup{lex}}\un w.
    $$

    \vspace{0.1cm}

    This order also induces a total order on $\textup{Div}^d(\Gamma)$, by defining, for each $D,D'\in\textup{Div}^d(\Gamma)$, $D>D'$ if $D_v>D'_v$, where $v=\textup{max}\{w|D_w\neq D'_w\}$.

    We define the \emph{depth} of a set $V_{(
        i_1,j_1),(i_2,j_2),\ldots,(i_{k-1},j_{k-1}),(0,1)}$ to equal $k$,
    and the \emph{depth} of a vertex $v$ to equal $\textup{depth}(\{v\})$.  (This also equals half the length of the vector $\un{v}$ plus $1$).

    Similarly, we define the \emph{depth} of an edge $e$ to be the smallest positive integer $k$ such that $$e\notin \bigsqcup_{V\textrm{ appears in } V^k_\bullet}\Gamma(V).$$
\end{definition}

\begin{definition}\label{D:totalordinduced}

    Let $f:\Gamma\to \Gamma'$ be a morphism, and let $(V^1_\bullet,\ldots,V^K_\bullet)$ be a sequence of ordered partitions for $V(\Gamma)$ as in Definition~\ref{D:totalorderGamma}. We define the sequence of ordered partitions $(W^1_\bullet,\ldots,W^H_\bullet)$ \emph{induced} by $(V^1_\bullet,\ldots,V^K_\bullet)$ on $\Gamma'$ as follows:
    \begin{itemize}
        \item [Step 1.]
    \begin{enumerate}
        \item Fix $w_{(0,1)}:=f(v_{(0,1)})\in V(\Gamma')$, and define $W_0:=\{w_{(0,1)}\}$.
        \item Fix an ordered partition $(W_0,W_1,\ldots,W_m)$ such that $w\in W_{i_1}$ if and only if $d(w,w_{(0,1)})=i_1$.
        \item Further partition each $W_{i_1}$ into $(W_{(i_1,1)},\ldots,W_{(i_1,l(i_1))})$, such that $\Gamma(W_i)=\bigsqcup_{j_1=1,\ldots,l(i_1)}\Gamma(W_{(i_1,j_1)})$ is the ordered decomposition in connected components, with $j_1<j_1'$ iff 
        $$
        \textup{max}\{v\in f^{-1}(W_{(i_1,j_1)})\}>\textup{max}\{v\in f^{-1}(W_{(i_1,j_1')})\}.
        $$
        Clearly, $W_0=W_{(0,1)}$.
        \item Define $W^1_\bullet:=(W_{(0,1)},W_{(1,1)},\ldots,W_{(1,l(1))},\ldots,W_{(m,l(m))})$. In other words, $W_{(i_1,j_1)}<W_{(i_1',j_1')}$ iff $(i_1,j_1)>_{\textup{lex}}(i_1',j_1')$, for every $i_1,j_1,i_1',j_1'$.
    \end{enumerate}
    \item [Step h.]Assume the ordered partition $W^{h-1}_\bullet=(W_{(i_1,j_1),(i_2,j_2),\ldots,(i_{h-1},j_{h-1})})_{\{(i_1,j_1),(i_2,j_2),\ldots,(i_{k},j_{k}),1\leq k\leq h-1\}}$ of $V(\Gamma')$ has been defined.
    We define the ordered partition $W^h_\bullet$ of $V(\Gamma')$ as a refinement of $W^{h-1}_\bullet$, as follows. 
    
    \begin{enumerate}
        \item For every $(
        i_1,j_1),(i_2,j_2),\ldots,(i_{k-1},j_{k-1})$ such that $(i_{h-1},j_{h-1})\neq(0,1)$:
    \begin{enumerate}
        \item  Fix the vertex $w_{(
        i_1,j_1),(i_2,j_2),\ldots,(i_{h-1},j_{h-1}),(0,1)}\in W_{(
        i_1,j_1),(i_2,j_2),\ldots,(i_{h-1},j_{h-1})}$, such that 
        $$
        \textup{max}\{v\in f^{-1}(w_{(
        i_1,j_1),(i_2,j_2),\ldots,(i_{h-1},j_{h-1}),(0,1)})\}\geq\textup{max}_{w\in W_{(
        i_1,j_1),(i_2,j_2),\ldots,(i_{h-1},j_{h-1})}}\{v\in f^{-1}(w)\},
        $$
        and define $(W_{(
        i_1,j_1),(i_2,j_2),\ldots,(i_{h-1},j_{h-1}),0}:=\{w_{(
        i_1,j_1),(i_2,j_2),\ldots,(i_{h-1},j_{h-1}),(0,1)}\}$. If all the elements of $W^h_\bullet$ are singletons, i.e. they are of the form $W_{(
        i_1,j_1),(i_2,j_2),\ldots,(i_{k-1},j_{k-1}),(0,1)}$, for some $1\leq k\leq h$, the procedure ends with $H:=h$.
        \item  Let
        $$m_h:=\textup{max}_{w\in W_{(
        i_1,j_1),(i_2,j_2),\ldots,(i_{h-1},j_{h-1})}}\{d(w,w_{(
        i_1,j_1),(i_2,j_2),\ldots,(i_{h-1},j_{h-1}),(0,1)})\}.$$ 
        (Note that $m_h$ also depends upon $(
        i_1,j_1),(i_2,j_2),\ldots,(i_{h-1},j_{h-1})$).
        Fix the ordered partition 
        $$
        (W_{(
        i_1,j_1),(i_2,j_2),\ldots,(i_{h-1},j_{h-1}),0},W_{(
        i_1,j_1),(i_2,j_2),\ldots,(i_{h-1},j_{h-1}),1},\ldots,W_{(
        i_1,j_1),(i_2,j_2),\ldots,(i_{h-1},j_{h-1}),m_h})$$ of $W_{(
        i_1,j_1),(i_2,j_2),\ldots,(i_{h-1},j_{h-1})}$ such that $w\in W_{(
        i_1,j_1),(i_2,j_2),\ldots,(i_{h-1},j_{h-1}),i_h}$ if and only if $d(w,w_{(
        i_1,j_1),(i_2,j_2),\ldots,(i_{h-1},j_{h-1}),(0,1)})=i_h$, where $0\leq i_h\leq m_h$.
        \item Further partition each $W_{(
        i_1,j_1),(i_2,j_2),\ldots,(i_{h-1},j_{h-1}),i_h}$ into\\ $(W_{(
        i_1,j_1),(i_2,j_2),\ldots,(i_{h-1},j_{h-1}),(i_h,1)},\ldots,W_{(
        i_1,j_1),(i_2,j_2),\ldots,(i_{h-1},j_{h-1}),(i_h,l(i_h))})$, such that $$\Gamma(W_{(
        i_1,j_1),(i_2,j_2),\ldots,(i_{h-1},j_{h-1}),i_h})=\bigsqcup_{j_h=1,\ldots,l(i_h)}\Gamma(W_{(
        i_1,j_1),(i_2,j_2),\ldots,(i_{h-1},j_{h-1}),(i_h,j_h)})$$ is the ordered decomposition into connected components, with $j_h<j_h'$ iff
        $$
        \textup{max}\{v\in f^{-1}((W_{(
        i_1,j_1),(i_2,j_2),\ldots,(i_{h-1},j_{h-1}),(i_h,j_h)})\}>\textup{max}\{v\in f^{-1}((W_{(
        i_1,j_1),(i_2,j_2),\ldots,(i_{h-1},j_{h-1}),(i_h,j_h')})\}.
        $$

        Clearly, $W_{(
        i_1,j_1),(i_2,j_2),\ldots,(i_{h-1},j_{h-1}),0}=W_{(
        i_1,j_1),(i_2,j_2),\ldots,(i_{h-1},j_{h-1}),(0,1)}$.
    \end{enumerate}
        \item Let $$W^h_\bullet:=(W_{(
        i_1,j_1),(i_2,j_2),\ldots,(i_{k},j_{k})})_{\{(
        i_1,j_1),(i_2,j_2),\ldots,(i_{k},j_{k}),1\leq k\leq h\}}$$ be the ordered partition of $V(\Gamma')$, where the ordering is defined by $W_{(
        i_1,j_1),(i_2,j_2),\ldots,(i_{k},j_{k})}<V_{(
        i'_1,j'_1),(i'_2,j'_2),\ldots,(i'_{k'},j'_{k'})}$ iff $(
        i_1,j_1),(i_2,j_2),\ldots,(i_{k},j_{k})>_{\textup{lex}}(
        i'_1,j'_1),(i'_2,j'_2),\ldots,(i'_{k'},j'_{k'})$.
\end{enumerate}
\end{itemize}

    Similarly to the previous definition, a sequence of ordered partitions $(W^1,\ldots, W^h)$ induces a total order $\leq$ on $V(\Gamma')$ and, consequently, also on $\Div^d(\Gamma')$, for any integer $d$.
\end{definition}

\begin{proposition}\label{P:cardPTstab}
    Let $\D\in\Deg(\Gamma)$, and let $\P$ be an upper subset of $\OO_\D^d(\Gamma)$.
    \begin{enumerate}
     \item \label{P:cardPTstableq}

    Let $G\in\S\S_\D(\Gamma)$. If $\P$ satisfies Property~\ref{D:PTstab}(\ref{D:PTstabProp2}) for each $G'\in S_G(\P)$ (see Notation~\ref{N:NotationPT}(\ref{N:subposet})), we have 
    $$
    |\P(G')|\leq c_{\D(\P)}(G'),
    $$
    for each $G'\in S_G(\P)$. 
    
    In particular, if $\P$ satisfies Property~\ref{D:PTstab}(\ref{D:PTstabProp2}) for each $G\in\S\S_\D(\Gamma)$, then
    $$
    |\P(G)|\leq c_{\D(\P)}(G),
    $$
    for each $G\in\S\S_\D(\Gamma)$.
    \item If, moreover, $\P$ is a PT-assignment, then $$|\P(G)|= c_{\D(\P)}(G).$$
\end{enumerate}
\end{proposition}

\begin{remark}\label{R:cD-spec}
 Proposition~\ref{P:cardPTstab} reduces to two well-known formulas for PT-assignements with extreme degeneracy subsets:
\begin{enumerate}
\item If $\D(\P)=\emptyset$, then, for any connected spanning subgraph $G\leq \Gamma$, the natural quotient map $\Div(G) \to \Pic(G)$ induces a bijection $\pi_G: \P(G) \to \Pic^{d-E(G)^\mathsf{c}}(G)$ where $d$ is the degree of $\P$ (by Definition~\ref{D:PTstab}\eqref{D:PTstab1}), and then the statement reduces to the equality $|\Pic^0(G)|=c(G)$, which follows from the classical Kirchhoff's (matrix-tree) theorem.
\item If $\D(\P)=\BCon(\Gamma_X)$ then $\P(G)$ is in bijection with the outdedgree sequences of $G$ (by Theorem~\ref{T:V=PT}\eqref{T:V=PT2}, Proposition~\ref{P:maxdegcJ} and Corollary~\ref{C:posdeg}) which is known to be equal to $c_{\BCon(\Gamma_X)}(G)$, i.e. the number of spanning forests of $G$ (see \cite[Prop.~40]{Ber} and the references therein).
\end{enumerate}
\end{remark}

\begin{proof}
We show the first claim by providing an injection from $\P_\D(G)$ to the set of $\D$-admissible spanning forests $\S\F_\D(G)$ of $G$, for each $G\leq \Gamma$ such that $\P_\D(G)\neq\emptyset$. Without loss of generality, we assume $G=\Gamma$.

Consider the full subposet $S(\P):=S_\Gamma(\P)$ of $\S\S_{\D}(\Gamma)$, as in Notation~\ref{N:NotationPT}(\ref{N:subposet}), whose objects are subgraphs $G\leq \Gamma$ such that:
\begin{itemize}
    \item there exists an ordered partition $W_\bullet$ of $V(\Gamma)$ with $G=\Gamma\setminus E(W_\bullet)$;
    \item $\P(G)\neq\emptyset$.
\end{itemize} 
In particular for $G,G'$ objects of $S(\P)$, we have $G\leq G'$ if there exists an ordered partition $W_\bullet$ of $V(G')$ such that $G=G'\setminus E(W_\bullet)$.

We use a total order as given in Definition~\ref{D:totalorderGamma} on $\Div^d(\Gamma)$ restricted to $\P(\Gamma)$ to define a map $K_\Gamma:\P(\Gamma)\to\S\F_\D(\Gamma)$. For each graph $G$ in $S(\P)$, let $\delta_G$ be as in Remark~\ref{R:DegreemSS}.
By the Kirchhoff theorem, we have $|\textup{Pic}^{\delta_G}(G)|=c(G)$ where, if $G=\bigsqcup_i G_i$ is the decomposition into connected components, $c(G)=\prod_ic(G_i)$. Hence, we can choose a bijection $k'_G:\textup{Pic}^{\delta_G}\to\S\F_{\textup{con}}(G)$, where $\S\F_{\textup{con}}(G)$ denotes the set of spanning forests of $G$ whose connected components coincide with those of $G$.

Consider the natural map $\pi_G:\P(G)\to \textup{Pic}^{\delta_G}(G)$.
For each $[D]\in\textup{Pic}^{\delta_G}(G)$ such that $\pi_G^{-1}([D])\neq\emptyset$, let $\tilde{D}=\textup{max}\{D\in\pi_G^{-1}([D])\}$, and define $k_G(\tilde{D}):=k'_G([D])\in\S\F_{\textup{con}}(G)$.

For each object $G$ of $S(\P)$, we define a map $K_G$ extending $k_G$ to all $\P(G)$, by poset induction on $S(\P)$.
If $G$ is minimal in $S(\P)$, it follows from Definition~\ref{D:PTstab}(\ref{D:PTstabProp2}) and Remark~\ref{R:DegreemSS}, that $\pi_G:\P(G)\to \textup{Pic}^{\delta_G}(G)$ is injective, hence $K_G:=k_G$ is defined on $\P(G)$. 
    
Let $G$ be any graph in $S(\P)$, and assume $K_{G'}$ is defined for all $G'<G$ with $G'$ an object of $S(\P)$. Let $D\in\P(G)$. If $D$ is maximal in $\pi_G^{-1}([D])$, then set $K_G(D):=k_G(D)$. Otherwise, there exists a maximal divisor $\tilde{D}\in\pi_G^{-1}([D])$ and an ordered partition $V_\bullet$ such that $(G\setminus E(V_\bullet),D-D(\O(V_\bullet)))\in\P$ and $\tilde D=D-D(\O(V_\bullet))+D(\O(\overline{V}_\bullet))$. By the induction hypothesis, $K_{G\setminus{E(V_\bullet)}}$ is defined, and we set $K_G(D):=K_{G\setminus E(V_\bullet)}(D-D(\O(V_\bullet)))$.

    \vspace{0.1cm}

    Now, we show that the map $K_\Gamma$ is a bijection on its image. We do so by providing an inverse $H_\Gamma:\Im(K_\Gamma)\to\P(\Gamma)$.
    We observe that, from the construction of $K_\Gamma$, for each $D\in\P(\Gamma)$, there exists a unique sequence $V_{\bullet,\bullet}=(V_{\bullet,1},\ldots,V_{\bullet,t(D)})$ of ordered partitions of $V(\Gamma)$ (for some integer $t(D)$) such that the following hold: $$D-\sum_{i=1}^{t(D)}D(\O(V_{\bullet,i})) \textup{ is maximal in } \pi^{-1}_{\Gamma\setminus\bigsqcup_iE(V_{\bullet,i})}([D-\sum_{i=1}^{t(D)}D(\O(V_{\bullet,i}))])$$ and $$K_G(D)=K_{\Gamma\setminus\bigsqcup_iE(V_{\bullet,i})}(D-\sum_{i=1}^{t(D)}D(\O(V_{\bullet,i}))).$$

    We construct the map $H_\Gamma:\Im(K_\Gamma)\to\P(\Gamma)$ as follows.
    For each forest $F$ in $\S\F_\D(\Gamma)$, there exists at most one graph $G\in S(\P)$ such that $F\in\Im(K_G)$. Let $D_0= K_G^{-1}(F)$ be the unique maximal divisor in $\pi_G^{-1}([D_0])$, and let $V_{\bullet,\bullet}$ be the sequence of ordered partitions of $V(\Gamma)$ associated to $D_0$ by Lemma~\ref{L:Vbulletbullet} below (with $G_0=G \in S(\P)$). 
    We define $$H_\Gamma(F):=D_0+\sum_{i=1}^t D(\O(V_{\bullet,i})),$$
    for some $t$, and it is clear that $K_\Gamma$ and $H_\Gamma$ are mutual inverses. This completes the proof of Part~1.

    \vspace{0.1cm}

    We now prove Part~2. If we further assume that the map $\pi_G:\P_\D(G)\to \textup{Pic}^{\delta_G}(G)$ is surjective for every $G\in\S\S_\D(\Gamma)$, then the map $K_\Gamma$ is surjective on $\S\S_\D(\Gamma)$,  with inverse $H_\Gamma$. 
\end{proof}

\begin{lemma}\label{L:Vbulletbullet}
    Let $\Gamma$ be a connected graph, let $\D\in\Deg(\Gamma)$, and let $\leq$ be a total order on its vertices as defined in Definition~\ref{D:totalorderGamma}. Let $\P$ be an upper subset of $\OO^d_\D(\Gamma)$ satisfying Property~\ref{D:PTstab}(\ref{D:PTstabProp2}). Then, for each subgraph $G_0=\Gamma\setminus E(W_\bullet)\leq\Gamma$, for some ordered partition $W_\bullet$, and each divisor $D_0\in\P_\D(G_0)$, there exist a unique integer $k$ and a sequence of ordered partitions $V_{\bullet,\bullet}=(V_{\bullet,1},\ldots,V_{\bullet,k})$ such that the following hold for each $i=1,\ldots,k$:
    \begin{itemize}
        \item $D_i:=D_{i-1}+D(\O(V_{\bullet,i}))\in\P(G_i)$, where $G_i:=G_{i-1}\sqcup E(V_{\bullet,i})$;\\
        \item $D_{i-1}+D(\O(\overline{V}_{\bullet,i}))$ is maximal in $\pi_{G_i}^{-1}([D_i])$;\\
        \item $D_k\in\P(\Gamma)$.
    \end{itemize}
    
\end{lemma}
\begin{proof}
    Let $E^{(k)}$ denote the set of edges of $\Gamma$ of  depth $k$. 
    
    We observe that the datum of an ordered partition $V_\bullet$ on a graph $\Gamma$ is equivalent to that of a collection of edges $F\subseteq E(\Gamma)$ such that, if $\Gamma\setminus F=\bigsqcup_i \Gamma_i$ is the decomposition in connected components, for each $e\in F$, there exist two indices $i\neq j$ and two vertices $v\in V(\Gamma_i)$, and $w\in V(\Gamma_j)$ such that $e=\{v,w\}$, together with a total order of the components.

    Firstly, assume that $G_0=\bigsqcup_{v\in V(\Gamma)}\{v\}$ is the decomposition in connected components, i.e. that all the components are singletons, and let $k$ be the maximum depth of an edge $e\in\Gamma$. Let $V_{\bullet,\bullet}$ be such that $V_{\bullet,i}$ is given by $E^{(k-i+1)}$ and the order $\geq$ on the sets of vertices of depth at most $k-i+1$. Indeed, these correspond to the connected components of $G_i\setminus E^{(k-i+1)}$.
    
    Then $V_{\bullet,\bullet}$ satisfies the required conditions:
    \begin{itemize}
        \item For each $i=1,\ldots,k$, the divisor $D_{i-1}+D(\O(\overline{V}_{\bullet,i}))$ is maximal in its equivalence class. Indeed, the only possible chipfiring action could happen along the edges in $E^{(k-i+1)}$, but this would lead to a divisor $D'<D_{i-1}+D(\O(\overline{V}_{\bullet,i}))$, by the orientation imposed by $\overline{V}_{\bullet,i}$ on such edges.\\
        \item The sequence $V_{\bullet,\bullet}$ is the only one satisfying these conditions. Let $e=\{s,t\}\in E(V_{\bullet,i})$ and $e'\in E(V_{\bullet,i'})$, such that $\textup{depth}(e)<\textup{depth}(e')$ and $i<i'$. Then $D_{i'-1}+D(\O(\overline{V}_{\bullet,i'}))$ is not maximal in its equivalence class. In fact, let $V_t$ be the unique element such that $\textup{depth}(V_t)=\textup{depth}(e)$ and $t\in V_t$, and let 
        $$\ov V_t:=\bigcup\{V|V\leq V_t\textup{ and } \textup{depth}(V)=\textup{depth}(e)\}.$$ 
        Then,
    $$
    D_{i'-1}+D(\O(\overline{V}_{\bullet,i'}))<D_{i'-1}+D(\O(\overline{V}_{\bullet,i'}))+\div_{G_{i'}}(\chi_{\ov V_t})\in \P_\D(G_{i'}).
    $$

    Hence, we deduce that, for every pair of edges $e,e'$, if $\textup{depth}(e)<\textup{depth}(e')$, then $e\in E(V_{\bullet,i})$, and $e'\in E(V_{\bullet,i'})$, with $i\geq i'$.

    Similarly, we can see that if $e\in E(V_{\bullet,i})$, then all edges $e'$ with $\textup{depth}(e)=\textup{depth}(e')$ must be in $E(V_{\bullet,i})$ as well. 

    Moreover, notice that, for every $1\leq j<j'\leq k$, there doesn't exist an index $i$ such that all the edges of depth $j$ or $j'$ are in $E(V_{\bullet,i})$, as, in that case, we would have $D_{i-1}+D(\O(\overline{V}_{\bullet,i}))\notin \pi_{G_i}^{-1}([D_i])$. Therefore, we conclude that $V_{\bullet,\bullet}$ is unique.
    \end{itemize}

    \vspace{0.1cm}

    Now, let $G_0=\Gamma\setminus E(W_\bullet)\leq\Gamma$, for some ordered partition $W_\bullet$, be any graph. Let $f:\Gamma\to \Gamma'$ be the morphism that contracts all the edges of $G_0$. Construct $V_{\bullet,\bullet}$ on $G'_0:=f(G_0)$ as above, with respect to the order induced by $f$ on $\Gamma'$, by Definition~\ref{D:totalordinduced}. Then, arguing as above, the preimage $f^{-1}(V_{\bullet,\bullet}):=(f^{-1}(V_{\bullet,1}),\ldots,f^{-1}(V_{\bullet,k}))$ is the unique sequence of ordered partitions satisfying the required conditions.
\end{proof}

\subsection{Equivalence of V-stability conditions and PT-assignments}
\label{Sub:PT=V}

In this subsection we prove in Theorem~\ref{T:V=PT} the equivalence of V-stability conditions defined in Definition~\ref{D:Vstab} and PT-assignments defined in Definition~\ref{D:PTstab}, generalizing the results of \cite[Subsections 1.3,1.4,1.5]{viviani2023new}.

First of all, we prove that the semistable set of a V-stability satisfies Condition (\ref{D:PTstabProp2}) of Definition~\ref{D:PTstab} of a PT-assignment.

\begin{proposition} \label{p:Vstabtopic}
    Let $\n$ be a V-stability condition on a graph $\Gamma$ of degree $d$.
    Let $G=\Gamma\setminus S$ be $\D(n)$-admissible, and consider the composite map 
        \begin{equation*}
            \pi_{G}:\P_\n(G)\subset \Div^{d-|S|}(G)\twoheadrightarrow\Pic^{\delta_G}(G).
        \end{equation*}
        If $D, E\in\P_\n(G)$ are such that $\pi_G(E)=\pi_G(D)$, then there exists an ordered partition $V_\bullet=(V_1,\ldots,V_t)$ of $V(G)$, with $V_i\in\wh \D(\n), \text{ for all } 1\leq i \leq t$, such that
        \begin{equation} \label{uno}
        E=D+D(\O_G(V_\bullet))-D(\O_G(\ov V_\bullet))
        \end{equation}
        and
        \begin{equation} \label{due}
        E_{V_i}+e_S(V_i)=\n_{V_i}+\val_G\left(V_i,\bigcup_{j<i} V_j\right) \text{ for all } 1\leq i \leq t.
        \end{equation}
        
        In particular, the restriction of $\pi_G$ to $\P_\n^{ps}(G)$ is injective.
\end{proposition}

If $\D(n)=\emptyset$, this proposition specialises to \cite[Proposition~1.13(2)]{viviani2023new}.

\begin{proof}
Let $F:V(G)\to \ZZ$ be a map such that $E=D+\div_G(F)$.  

Let $$V_1:=\{v\in V(G):F_v=\textup{max}_{w\in V(\Gamma)}F_w=:M_1\}\subseteq V(G).$$
(Note that if $V_1= V(G)$, then $E=D$ and the statement holds trivially with $V_\bullet=(V(G))$. This happens if $\D(\n) = \emptyset$.)

For all $h\geq 1$ such that $\bigcup_{i=1}^h V_i \subsetneq V(G)$, for convenience we set 
 $$V_{h+1}:=\Big\{v\in V(G):F_v=\textup{max}_{w\in V(G)\setminus\bigcup_{i=1}^h V_i}F_w=:M_{h+1}\Big\}\subseteq V(G)\setminus\bigcup_{i=1}^h V_i.$$
For each $1\leq h\leq |\Im(F)|$, let $V_h^i\subseteq V_h$, for $i=1,\ldots,l_h$ be such that $\Gamma[V_h]=\bigsqcup_{i=1}^{l_h}\Gamma[V_h^i]$ is the decomposition into connected components. We also define \[G_{h}:=G\setminus \bigsqcup_{j\leq h}E(V_j), \quad  D_h:=D-\sum_{j\leq h} D(\O(V_j)), \quad E_h:=D_h+\div_{G_h}(F),\] for every $1\leq h\leq |\Im(F)|$.

Before we start our proof, we observe the following. Firstly, since $F_v\leq M_{h+1}$ for each $v\in V(G)\setminus \bigsqcup_{i=1}^{h}V_i$ and $\val_{G_h}(V_{h+1}^i,(V_{h+1}^i)^\mathsf{c}\setminus V_{h+1}^\mathsf{c})=\val_{G_h}(V_{h+1}^i,V_j)=0$, for every $j<h$ and $1\leq i\leq l_h$, we have:
\begin{align}
    \div_{G_h}(F)_{V_{h+1}^i}&=\sum_{v\in V_{h+1}^i}\left[ -F_v\val_{G_h}(v)+\sum_{v\neq w\in V_{h+1}^i}F_w\val_{G_h}(v,w)+\sum_{w\in (V_{h+1}^i)^\mathsf{c}}F_w\val_{G_h}(v,w)\right]=\notag\\
    &=\sum_{v\in V_{h+1}^i}\left[ -M_{h+1}\val_{G_h}(v)+M_{h+1}\val_{G_h}(v,V_{h+1}^i\setminus\{v\})+\sum_{w\in V_{h+1}^\mathsf{c}}F_w\val_{G_h}(v,w)\right]\leq\notag\\
    &
    \leq \sum_{v\in V_{h+1}^i}\left[ -M_{h+1}\val_{G_h}(v)+M_{h+1}\val_{G_h}(v,V_{h+1}^i\setminus\{v\})+(M_{h+1}-1)\val_{G_h}(v,V_{h+1}^\mathsf{c})\right]=\notag\\
    &=\sum_{v\in V_{h+1}^i}\left[M_{h+1}\val_{G_h}(v,(V_{h+1}^i)^\mathsf{c}+(M_{h+1}-1)\val_{G_h}(v,V_{h+1}^\mathsf{c})\right]=\notag \\
    &\label{E:Vcard0} = -\val_{G_h}(V_{h+1}^i,V_{h+1}^\mathsf{c})=-\val_{G_h}(V_{h+1}^i).
\end{align}

Secondly, for all $W\in\Con(G)$, denote by $G[W^\mathsf{c}]=\bigsqcup_{i=1}^kG[Z_i]$ the decomposition into connected components of $G[W^\mathsf{c}]$.
By applying Remark~\hyperref[R:prop-Pn2]{\ref*{R:prop-Pn}(\ref*{R:prop-Pn2})} to both $D$ and $E$, we obtain:
\begin{equation}\label{E:Vcard1}
    |\{Z_i: \: Z_i\not \in \D(\n)\}|-\val_G(W)\leq\div_G(F)_W.
\end{equation}

We will now prove the following claims by induction on $1\leq h\leq |\Im(F)|$:
\begin{itemize}
    \item every $V_h^i$ belongs in $\wh\D(\n)$,
    \item the equation  
    $$E_{V_h^i}+e_S(V_h^i)=\n_{V^i_h}+\val_G\left(V_h^i,\bigcup_{j<h} V_j\right)$$
    is satisfied for each $1\leq i \leq l_h$,
    \item the divisors $D_h,E_h$ belong in $\P_\n(G_h)$.
\end{itemize}

Let $h=1$. 
The combination of \eqref{E:Vcard1} applied to $W=V_1^i$ together with \eqref{E:Vcard0} shows that $V_1^i\in\wh\D(\n)$ and that 
$$
\div_G(F)_{V_1^i}=\val_G(V_1^i)
$$
for all $1 \leq i \leq l_1$. Moreover, since 
$$
E_{V_1^i}=(D+\div_G(F))_{V_1^i}=D_{V_1^i}-\val_G(V_1^i)
$$
and both $D$ and $E$ belong to $\P_\n(G)$, we deduce that 
\begin{equation}\label{E:Vcard2}
\begin{cases}
    D_{V_1^i}+e_S(V_1^i)=\n_{V_1^i}+\val_G(V_1^i),\\
    E_{V_1^i}+e_S(V_1^i)=\n_{V_1^i},
\end{cases}    
\end{equation}
for all $1\leq i\leq l_1$. In particular, each $V_1^i$ satisfies \eqref{due}.

Finally, since $V_1^i\in\wh\D(\n)$ for each $i$, \eqref{E:Vcard2} implies that $D_1,E_1\in\P_\n(G_1)$.
Indeed, let $W^1_1,\ldots,W^1_{k_1}$ be such that $G[V_1^\mathsf{c}]=\bigsqcup_{i=1}^{k_1}G[W^1_i]$ is the decomposition in connected components, and let $$W_\bullet^1:=(W_1,\ldots,W_{k_1},V_1^1,\ldots,V_1^{l_1})$$ be an ordered partition. Then, applying Lemma~\ref{L:ss-isotr2} to $D\geq_{W_\bullet^1}D_1$ and $E\geq_{\ov W_\bullet^1}E_1$ respectively, we obtain the claims for $h=1$.

Assume now, that the claims hold for every $1\leq j\leq h$. By the induction hypothesis, we have $D_h,E_h\in\P_\n(G_h)$.  The combination of \eqref{E:Vcard1} applied to $W=V_{h+1}^i$ together with \eqref{E:Vcard0} shows that $V_{h+1}^i\in\wh\D(\n)$ and that 
$$
\div_{G_h}(F)_{V_{h+1}^i}=\val_{G_h}(V_{h+1}^i).
$$
for all $1 \leq i \leq l_{h+1}$. Moreover, since 
$$
(E_h)_{V_{h+1}^i}=(D_h+\div_{G_h}(F))_{V_{h+1}^i}=(D_h)_{V_{h+1}^i}-\val_{G_h}(V_{h+1}^i)
$$
and both $D_h$ and $E_h$ belong to $\P_\n(G_h)$, we deduce that 
\begin{equation}\label{E:Vcard4}
\begin{cases}
    D_{V_{h+1}^i}+e_{\Gamma\setminus G_h}(V_{h+1}^i)=\n_{V_{h+1}^i}+\val_{G_h}(V_{h+1}^i),\\
    E_{V_{h+1}^i}+e_{\Gamma\setminus G_h}(V_{h+1}^i)=\n_{V_{h+1}^i},
\end{cases}    
\end{equation}
for all $1\leq i\leq l_{h+1}$. By iterating the second equation for every $j\leq h+1$, we obtain

$$
E_{V_{h+1}^i}+e_S(V_{h+1}^i)=\n_{V_{h+1}^i}+\val_G\left(V_{h+1}^i,\bigcup_{j<h+1} V_j\right),
$$
that is, each $V_{h+1}^i$ satisfy \eqref{due}.

As in the previous case, let $W^{h+1}_1,\ldots,W^{h+1}_{k_{h+1}}$ be such that $$G_h[V_{h+1}^\mathsf{c}]=\bigsqcup_{i=1}^{k_{h+1}}G_h[W^{h+1}_i]$$ is the decomposition in connected components, and let $$W_\bullet^{h+1}:=(W^{h+1}_1,\ldots,W^{h+1}_{k_{h+1}},V_{h+1}^1,\ldots,V_{h+1}^{l_{h+1}})$$ be an ordered partition. Then, Lemma~\ref{L:ss-isotr2}, applied to $(G_h,D_h)\geq_{W_\bullet^{h+1}}(G_{h+1},D_{h+1})$ and $(G_h,E_h)\geq_{\ov W_\bullet^{h+1}}(G_{h+1},E_{h+1})$ respectively, shows that $D_{h+1},E_{h+1}\in\P_\n(G_{h+1})$, by \eqref{E:Vcard4} and the fact that $V_{h+1}^i\in\wh\D(\n)$ for each $1\leq i\leq l_{h+1}$. This concludes the proof of the claims by induction.

Finally, we claim that the ordered partition
$$V_\bullet:=(V_1^1,\ldots,V_1^{l_1},V_2^1,\ldots,V_2^{l_2}, \ldots, \ldots, V_{\Im(F)}^1, \ldots, V_{|\Im(F)|}^{l_{|\Im(F)|}})$$
is given by sets in $\wh\D(\n)$ and satisfies Properties \eqref{uno} and \eqref{due}.
Given what was shown by the induction, it remains to check Property \eqref{uno}. Indeed, this immediately follows from the definition of $V_h^i$, for all $1\leq h\leq |\Im(F)|$, and all $1\leq i\leq l_h$.

\end{proof}

\begin{corollary}\label{c:cardVstab}
    The $\n$-semistable set $\P_\n$ associated to a V-stability condition $\n$ on a graph $\Gamma$ satisfies the condition of Definition~\hyperref[D:PTstabProp2]{\ref*{D:PTstab}(\ref*{D:PTstabProp2})}. Moreover, $|\P_\n(G)|\leq c_{\D(\n)}(G)$, for each $G\in\S\S_{\D(\n)}(\Gamma)$.
\end{corollary}
\begin{proof}
    The first claim follows from the combination of Proposition~\ref{p:Vstabtopic} and Lemma~\ref{L:ss-isotr2}. The second claim follows from the first and Proposition~\ref{P:cardPTstab}(\ref{P:cardPTstableq}).   
\end{proof}

\begin{lemma}\label{L:ss-isotr2}
    For any V-stability $\n$ on $\Gamma$ of degree $d$ and any $(G=\Gamma \setminus S,D)\in \P_\n$. Let $W_\bullet=(W_0,\ldots, W_q)$ be an ordered partition of $V(\Gamma)$ such that $W_i\in \Con(G)$ for each $i$, and define 
    $$(G=\Gamma\setminus S,D)\geq_{\O_G(W_\bullet)} (\ov G=\Gamma\setminus \ov S, \ov D).$$
    Then we have that 
     $$
    (\ov G,\ov D)\in \P_\n \Leftrightarrow 
    \begin{sis} 
    &  W_i\in \wh D(\n) \:  \text{ and } \\  
    & D_{W_i}+e_S(W_i)=\n_{W_i}+\val_G\left(W_i,\bigcup_{j<i} W_j\right) \text{ for each } 0\leq i \leq q.
    \end{sis}
    $$
\end{lemma}
\begin{proof}
    We prove separately the two implications.

    $\un{\Longrightarrow}$: note that $W_i\in \Con(\ov G)$ since $W_i\in \Con(G)$ by assumption and none of the edges in the support of $\O_G(W_\bullet)$ are contained in $G[W_i]$.
    Therefore $\ov G=\coprod_i \ov G[W_i]$ is the decomposition of $\ov G$ into connected components. We can now apply Proposition \ref{P:n-ss}\eqref{P:n-ss2} in order to get (for any $0\leq i \leq q$):
    $$
    \begin{sis}
     & W_i \in \wh \D(\n),\\ 
     & \ov D_{W_i}+e_{\ov S}(W_i)=\n_{W_i}.
    \end{sis}
    $$
    We now conclude by observing that:
    \begin{itemize}
        \item $e_S(W_i)=e_{\ov S}(W_i)$, since $G[W_i]$ does not contain any edge of $\ov S-S=\supp \O_G(W_\bullet)$;
        \item $D_{W_i}=\ov D_{W_i}+\D(\O_G(W_\bullet))_{W_i}=\ov D_{W_i}+\val_G(\bigcup_{j<i} W_j, W_i)$, since all the oriented edges of $\O_G(W_\bullet)$ go from $G[W_j]$ to $G[W_i]$ for $j<i$.
    \end{itemize}

     $\un{\Longleftarrow}$: We will proceed by induction on $q$ (the case $q=0$ being trivial). 

     Consider the element $(\wt G, \wt D)\in \OO^d_\Gamma$ defined by 
     $$
     (G,D)\geq_{\O_G((W_0,W_0^\mathsf{c}))} (\wt G=\Gamma\setminus \wt S, \wt D) \geq_{\O(\wt G, (W_1,\ldots, W_q))} (\ov G,\ov D).
     $$

     \un{Claim:} $(\wt G, \wt D)\in \P_\n$.

     First of all, note that $G=\Gamma\setminus S$ is $\n$-admissible by Proposition \ref{P:n-ss}\eqref{P:n-ss2}. Hence, using Proposition \ref{P:Pn-rest}, it is enough to prove the result for $G=\Gamma$. Set $W:=W_0$. 
    
    Consider the decomposition into connected components $\Gamma[W^\mathsf{c}]=\coprod_i \Gamma[Z_i]$. Since $Z_i\in \D(\n)$ and $(\Gamma,D)\in \P_\n$, Remark \ref{R:prop-Pn}\eqref{R:prop-Pn1} implies that
    $$
    D_{Z_i}\leq \n_{Z_i}+\val(Z_i).
    $$
    Summing over all the $Z_i$, we get 
    \begin{equation}\label{E:ineq-sum}
    d-D_W=\sum_i D_{Z_i}\leq \sum_i[\n_{Z_i}+\val(Z_i)]=d-\n_W,
    \end{equation}
    where in the first equality we used that $\deg(D)=d$ and in the last equality we used Lemma \ref{L:add-whn}. Since $D_W=\n_W$ by assumption, we infer that the first and last term of \eqref{E:ineq-sum} are equal, which then forces (for any $Z_i$)
    \begin{equation}\label{E:equ-Zi}
    D_{Z_i}= \n_{Z_i}+\val(Z_i).
    \end{equation}
    By construction, we have that 
    $$
    \wt G=\Gamma[W] \coprod_i \Gamma[Z_i] \text{ and } \wt D=D_{|_W}+\sum_i [D_{|_{Z_i}}-D(\O(W,Z_i))],
    $$
     where $\O(W,Z_i)$ is the partial orientation of the edges $E(W,Z_i)$ from $W$ to $Z_i$. 
    By our assumption and \eqref{E:equ-Zi}, we get that 
    \begin{equation}\label{E:deg-Dbar}
     \deg \wt D_W=\n_W=|\n(\Gamma[W])| \text{ and } \deg \wt D_{Z_i}=\n_{Z_i}=|\n(\Gamma[Z_i])|.  
    \end{equation}
    Hence, by Proposition \ref{P:Pn-rest}, we have to check that 
    \begin{equation}\label{E:ss-W}
    (\Gamma[W],D_{|_W})\in \P_{\n(\Gamma[W])},    
    \end{equation}
    \begin{equation}\label{E:ss-Zi}
    (\Gamma[Z_i],D_{|_{Z_i}}-D(\O(W,Z_i)))\in \P_{\n(\Gamma[Z_i])},    
    \end{equation}
     First, we observe that\eqref{E:ss-W} follows immediately from Lemma~\ref{LD:V-subgr}. Indeed, for every $V\in \BCon(\Gamma[W])\subseteq\Con(\Gamma)$, we have 
     \begin{equation*}
       (D_{|_W})_V=D_V\geq \n_{V}=\n(\Gamma[W])_V,
     \end{equation*}
     where the inequality is given by Remark~\ref{R:prop-Pn}(\ref{R:prop-Pn2}

    Let us now prove \eqref{E:ss-Zi}. Let $V\in \BCon(Z_i)\subseteq\Con(\Gamma)$. Then, Lemma-Definition~\ref{LD:V-subgr} gives
    \begin{equation*}
    (D_{|_{Z_i}}-D(\O(W,Z_i)))_V=D_V-\val(Z_i^\mathsf{c},V)\geq\n_V=\n(\Gamma[Z_i])_V,
    \end{equation*}
    where the inequality follows from Remark~\ref{R:prop-Pn}(\ref{R:prop-Pn2}, since $(\Gamma,D)$ is $\n$-semistable, and \eqref{E:equ-Zi}.

 \vspace{0.2cm}

We now conclude the proof of the Lemma by applying the induction hypothesis to 
$(\wt G, \wt D) \geq_{\O(\wt G, (W_1,\ldots, W_q))} (\ov G,\ov D)$ and using that
(for any $1\leq i \leq q$)
 \begin{itemize}
     \item $(\wt G, \wt D)\in \P_\n$ by the Claim.
     \item $W_i\in \wh D(\n)$ by hypothesis.
     \item $\wt D_{W_i}+e_{\wt S}(W_i)=D_{W_i}-\val_{G}(W_0,W_i)+e_S(W_i)=\n_{W_i}+\val_G(\bigcup_{1\leq j<i} W_j,W_i)$ by hypothesis.
 \end{itemize}
\end{proof}

We now introduce BD-sets (=generalized break divisors sets), generalizing \cite[Sec. 1.4]{viviani2023new}. They will allow us to bridge between V-stabilities and PT-assignments.

\begin{definition}\label{D:BD-set}
    Let $\Gamma$ be a connected graph and let $\D\in\Deg(\Gamma)$. Consider a $\D$-\emph{forest function of degree $d$}, i.e. a function 
    \begin{equation*}
        I:m\S\S_\D(\Gamma)\to \Div(\Gamma),
    \end{equation*}
    such that $I(F)\in \Div^{d-b_1(\Gamma)-b_0(F)+1}(\Gamma)$ for any $F\in m\S\S_\D(\Gamma)$.
    
    Then the \emph{BD-set} with respect to $I$ is the minimal upper closed subset of $\OO_\D^d(\Gamma)$ that contains $\{I(F)\}_{F\in m\S\S_\D(\Gamma) }$. More explicitly, $BD_I$ is the subset of $\OO^d_\D(\Gamma)$ defined by 
    \begin{equation*}
        BD_I(G):=\bigcup_{\substack{G\geq F\in m\S\S_\D(\Gamma)\\ \supp(\O)=E(G)-E(F)}} I(F)+D(\O)\subset \Div^{d-|E(G)^\mathsf{c}|}(G),
    \end{equation*}
    for any $G\in \S\S_\D(\Gamma)$.
An upper closed subset of the form $BD_I$ for  some $\D$-forest function $I$ is called a BD-set with degeneracy set $\D$.
\end{definition}

We now investigate the behaviour of BD-sets under restriction to 
admissible spanning subgraphs and morphisms that preserve the genus.

Let $G=\Gamma\setminus S$ be a $\D$-admissible spanning subgraph of $\Gamma$. By Lemma~\ref{L:operOD}\eqref{L:operOD2i}, we have an inclusion of posets
$$
\S\S_{\D}(\Gamma\setminus S)\hookrightarrow \S\S_\D(\Gamma),
$$
which restricts to an inclusion of admissible forests
    \begin{equation}\label{EQ:SubforestsInclusion}
        \S\F_{\D}(\Gamma\setminus S)\hookrightarrow \S\F_\D(\Gamma)
    \end{equation}
whose image are the $\D$-admissible spanning forests of $\Gamma$ that do not contain any edges in $S$.
Hence, given a $\D$-forest function $I$ on $\Gamma$, we can restrict it to the minimal elements 
\begin{equation*}
    m\S\S_{\D}(\Gamma\setminus S)\hookrightarrow m\S\S_\D(\Gamma)
\end{equation*}
in order to obtain a function
\begin{equation}\label{e:forestfundel}
    I_{\Gamma\setminus S}:=I_{|_{m\S\S_{\D}}(\Gamma\setminus S)}:m\S\S_{\D}(\Gamma\setminus S)\to \bigsqcup_{t\leq d- |S|-b_1(\Gamma)}\Div^t(\Gamma).
\end{equation}
We denote by $\BD_{I_{\Gamma\setminus S}}$ the smallest upper subset of $\OO_\D^{d-|S|}(\Gamma\setminus S)$ that contains the image of $I_{\Gamma\setminus S}$.

\begin{lemma}\label{L:BDrestr}
With respect to the inclusion $\iota_G:\OO^{d-|S|}(\Gamma\setminus S)\hookrightarrow \OO^d(\Gamma)$ of Lemma \ref{L:operOD}\eqref{L:operOD2ii}, we have $BD_{I_{\Gamma\setminus S}}=BD_I\cap \OO^{d-|S|}_{\D}(\Gamma\setminus S)$.
\end{lemma}
\begin{proof}
    An element $I(F)+D(\O)$ of $BD_I(G)$, with $\O$ a partial orientation of $G$ whose support is $E(G)-E(F)$, belongs to $\OO^{d-|S|}_{\D}(\Gamma\setminus S)$ if and only if $G$ does not contain any of the edges in $S$. Hence, $F\leq G\leq \Gamma\setminus S$, and, in particular, $F$ is a minimal spanning forest of $\Gamma\setminus S$. Thus, we can conclude since $I_{\Gamma\setminus S}(F)=I(F)$.
\end{proof}

Let $f:\Gamma\to \Gamma'$ be a morphism of graphs such that $b_1(\Gamma)=b_1(\Gamma')$ and $\D$ be a degeneracy subset for $\Gamma$. Then, by Lemma~\ref{L:operOD}(\ref{L:operOD1}), there exists an injective pullback map 
    \begin{align}\label{EQ:SubforestsPullback}
        f^*:\S\F_{f_*\D}(\Gamma')&\hookrightarrow \S\F_\D(\Gamma)\\
        \Gamma'\setminus R&\mapsto \Gamma\setminus f^E(R),\notag
    \end{align}
    whose image is given by the $\D$-admissible spanning forests of $\Gamma'$ that contain all the edges contracted by $f$.
    
\begin{lemma-definition}
    Assume that the spanning forest $\Gamma\setminus f^E(S)$ belongs in $\tilde{f}^*(m\S\S_{f_*\D}(\Gamma'))\setminus m\S\S_\D(\Gamma)$. Let $S'$ be the maximal subset of contracted edges such that $\Gamma\setminus (f^E(S)\cup S')$ is still $\D$-admissible. Then $\Gamma\setminus (f^E(S)\cup S')$ is the unique minimal spanning forest of $\Gamma$ such that $\Gamma\setminus f^E(S)\geq \Gamma\setminus (f^E(S)\cup S')$ and the following map is well defined:
    \begin{align*}
    f^*:m\S\S_{f_*\D}(\Gamma')&\to m\S\S_\D(\Gamma)\\
    \Gamma'\setminus S&\mapsto \Gamma\setminus (f^E(S)\cup S').
\end{align*}
\end{lemma-definition}

\begin{proof}
    It is enough to observe that, for each spanning forest $F$ of $\Gamma$ there exists a unique minimal spanning forest $F_m$ such that $F_m\leq F$. In fact, let $F$ be a spanning forest of $\Gamma$ and let $e,e'\in E(F)$ such that both $F\setminus \{e\}$ and $F\setminus \{e'\}$ are $\D$-admissible. Then, we observe that $F\setminus\{e,e'\}\in\S\S_\D(\Gamma)$ too.
\end{proof}

Following from this, we now observe that any given $\D$ forest function $I$ induces a $f_*\D$ forest function that we call $f_*(I)$ on $m\S\S_{f_*\D}(\Gamma')$:
\begin{align}\label{e:forestfuncontr}
    f_*(I):m\S\S_{f_*\D}(\Gamma')&\to \bigsqcup_{t\leq d-b_1(\Gamma')}\Div^t(\Gamma')\\
    F&\mapsto f_*(I(f^*(F))).\notag
\end{align}
In the particular case where $f:\Gamma\to \Gamma/S$ is the contraction of some set $S\subset E(\Gamma)$ that does not decrease the genus of $\Gamma$, then we set $I_{\Gamma/S}:=f_*(I)$.

\begin{lemma}
    With respect to the map $f_*: \OO^d(\Gamma) \to \OO^d(\Gamma')$ defined in \eqref{E:f*-O}, we have the inclusion
    \begin{equation}
        BD_{f_*(I)}\subseteq f_*(BD_I).
    \end{equation}
\end{lemma}
\begin{proof}
    An element of $BD_{f_*(I)}$ is written as $f_*(I)(\Gamma'\setminus S)+D(\O)$, for some minimal spanning forest $\Gamma'\setminus S$ of $\Gamma'$, some $T\subseteq S$ and some orientation $\O$ of $S\setminus T$. By definition of $f_*(I)$, we have that $f_*(I)(\Gamma'\setminus S)=f_*(I(\Gamma\setminus (f^E(S)\cup S')))$, with $S'\cap f^E(E(\Gamma'))=\emptyset$. Moreover, $\O$ induces an orientation $\tilde{\O}$ of $f^E(S)\setminus f^E(T)$ such that $f_*(D(f^E(S)\setminus f^E(T),\tilde{\O}))=D(S\setminus T,\O)$. Thus, we conclude that
    $$
    f_*(I)(\Gamma'\setminus S)+D(\O)=f_*(I(\Gamma\setminus (f^E(S)\cup S'))+D(\tilde{\O})).
    $$
\end{proof}

In what follows, we will need a generalisation of the usual deletion/contraction property of the complexity of a graph (recall the definition of $\D$-complexity given in Definition~\ref{D:Dadmissible}).

\begin{lemma}\label{L:ContractionDeletion}
    Let $\Gamma$ be a connected graph and let $\D\in\Deg(\Gamma)$. Then for each  spanning subgraph $G\in\S\S_\D(\Gamma)$ and each $e\in E(G)$ such that $G\setminus \{e\}\in\S\S_\D(\Gamma)$ and $b_1(G)=b_1(G/\{e\})$ (i.e. $e$ is not a loop of $G$), we have 
    \begin{equation*}
    c_\D(G)=c_{\D}(G \setminus  e )+c_{(c_e)_*\D}(c_e(G)),
    \end{equation*}
    where $c_e:G\to G/\{e\}$ is the contraction of $e$.
\end{lemma}
\begin{proof}
    It is sufficient to observe that the maps \eqref{EQ:SubforestsInclusion} for $G\setminus \{e\}\leq G$ and \eqref{EQ:SubforestsPullback} for $c_e:G\to G/\{e\}$ have disjoint images whose union is the set $\S\F_\D(G)$.
\end{proof}

\begin{proposition}\label{BDthm}
    Let $\Gamma$ be a connected graph and let $I$ be a $\D$ forest function, for some $\D\in\Deg(\Gamma)$. 
    \begin{enumerate}
        \item \label{BDthm1} We have that $|BD_I(\Gamma)|\geq c_\D(\Gamma)$.
        \item \label{BDthm2} If $|BD_I(\Gamma)|= c_\D(\Gamma)$, then
        \begin{enumerate}
            \item for any $\D$-admissible spanning subgraph $\Gamma\setminus S$, we have $|BD_I(\Gamma\setminus S)|= c_\D(\Gamma\setminus S)$;
            \item for any morphism $f:\Gamma\to \Gamma'$ that preserves the genus, we have that $|BD_{f_*(I)}(\Gamma')|=c_{f_*(\D)}(\Gamma')$.
        \end{enumerate}
    \end{enumerate}
\end{proposition}
This is proved in \cite[Theorem~1.17]{viviani2023new} for $\D=\emptyset$.
\begin{proof}
    The proof of \cite[Theorem~1.17]{viviani2023new} can be adapted to our more general setting: one argues by induction on the edges of $\Gamma$ such that $\Gamma\setminus \{e\}\in \S\S_{\D}(\Gamma)$ and $b_1(\Gamma/\{e\})=b_1(\Gamma)$, using the functions $I_{\Gamma\setminus \{e\}}$ of \eqref{e:forestfundel} and $I_{\Gamma/\{e\}}$ of \eqref{e:forestfuncontr}, and Lemma \ref{L:ContractionDeletion}.
\end{proof}

\begin{corollary}\label{C:BDsetContractions}
    If $BD_I$ is a weak numerical PT-assignment, then:
    \begin{enumerate}
        \item $BD_{I_{\Gamma\setminus S}}$ is a numerical PT-assignment for every $\D$-admissible spanning subgraph $\Gamma\setminus S$ of $\Gamma$. In particular, $BD_I$ is a numerical PT-assignment.
        \item $BD_{f_*(I)}$ a numerical PT-assignment for each genus preserving morphism $f:\Gamma\to\Gamma'$. In particular $BD_{I_{\Gamma/S}}$ is a numerical PT-assignment for each contraction $\Gamma/S$ such that $b_1(\Gamma)=b_1(\Gamma/S)$.
    \end{enumerate}
\end{corollary}
\begin{proof}
    It follows from Proposition~\hyperref[BDthm2]{\ref*{BDthm}(\ref*{BDthm2})}.
\end{proof}

\begin{corollary}\label{C:PTstabisBD}
If $\P$ is a numerical PT-assignment, then $\P=BD_I$ for a unique $\D(\P)$-forest function $I$.
\end{corollary}
\begin{proof}
    Similar to \cite[Corollary~1.19]{viviani2023new}.
\end{proof}

We now move towards the main results of this subsection, giving a bijection between V-stabilities and  PT-assignments, which are special BD-sets by the above Corollary.

We will first show how to construct a BD-set from a given V-stability.

\begin{lemma-definition}\label{L:BDInwelldef}
Let $\n$ be a V-stability of degree $d$ on $\Gamma$. 

Let $F \in m \S\S_{\D(\n)}(\Gamma)$ and let $F=\bigsqcup_{i=1}^kF_i$ be its decomposition in connected components. For any edge $e\in F$, denote by $F_{i_e}$ the connected component of $F$ containing $e$, and write \[F\setminus \{e\}=:F[W_e]\sqcup F[V(F_{i_e})\setminus W_e]\bigsqcup_{i\neq i_e}F_i,\] (where $F[W_e]$ and $F[V(F_{i_e})\setminus W_e]$ are connected). 

The \emph{$\D(\n)$-forest function of degree $d$ associated to $\n$} is the function 
$$I_\n:\m\S\S_{\D(\n)}(\Gamma)\to \Div(\Gamma)$$
such that $I_\n(F)$ is  the unique divisor on $\Gamma$ verifying
\begin{equation}\label{E:In}
\begin{cases}
    I_\n(F)_{W_e}=\n(F_{i_e})_{W_e},\\
    I_\n(F)_{V(F_{i_e})\setminus W_e}=\n(F_{i_e})_{V(F_{i_e})\setminus W_e},\\
    I_\n(F)_{V(F_i)}=|\n(F_i)|.
\end{cases}     
\end{equation}
\end{lemma-definition}
\begin{proof}
    We observe that $I_\n$ is well defined, as it is uniquely determined on every forest $F$ by  Condition~\eqref{E:In} applied to every edge $e$ of $F$.
    It remains to check that $I(F)\in \Div^{d-b_1(\Gamma)-b_0(F)+1}(\Gamma)$ for any $F\in m\S\S_\D(\Gamma)$.
    
    Fix a forest $F$ and an edge $e\in E(F)$. We observe that, for any $i=1\ldots,k$ and any $W\in\BCon(F_i)$, $\n(F_i)_W=\n(\Gamma[V(F_i)])_W-b_1(\Gamma[W])$. Then, we have 
    \begin{multline*}
    I_\n(F)_{W_e}+I_\n(F)_{W_e^\mathsf{c}}
    =I_\n(F)_{W_e}+I_\n(F)_{V(F_{i_e})\setminus W_e}+\sum_{i\neq i_e}I_\n(F)_{V(F_i)}=\n(F_{i_e})_{W_e}+\n(F_{i_e})_{V(F_{i_e})\setminus W_e}+\sum_{i\neq i_e}|\n(F_i)|=\\=\n_{W_e}-b_1(\Gamma[W_e])+\n_{V(F_{i_e})\setminus W_e}-b_1(\Gamma[V(F_{i_e})\setminus W_e])+\sum_{i\neq i_e}(\n_{V(F_i)}-b_1(\Gamma[V(F_i)]))=\\=d+1-b_1(\Gamma[W_e])-b_1(\Gamma[V(F_{i_e})\setminus W_e])-\val(W_e,V(F_{i_e})\setminus W_e)+\\-\sum_{i\neq i_e}b_1(\Gamma[V(F_i)])-\sum_{i<j}\val(\Gamma[V(F_i)],\Gamma[V(F_j)])=d+1-b_1(\Gamma)-b_0(F),
    \end{multline*}
    where the third equality follows from Lemma-Definition~\ref{LD:V-subgr}, and the fourth by Definition~\ref{D:Vstab}(\ref{D:Vstab1}), since $W_e\notin\wh\D(\n)$ and $V(F_i)\in\wh\D(\n)$ for each $0\leq i\leq k$.
\end{proof}

Conversely, we now show how to construct a V-stability from a given BD-set which is a weak numerical PT-assignment. 

\begin{lemma-definition}\label{L:nIwelldef}
Let $\D\in \Deg(\Gamma)$ and let $I$ be a $\D$-forest function of degree $d$ such that the corresponding BD-set $BD_I$ is a weak numerical PT-assignment. 

For any $W\in\BCon(\Gamma)$, pick a minimal spanning forest $F\in m\S\S_\D(\Gamma)$ that is \emph{adapted} to $W$, i.e. that satisfies
\begin{equation}\label{E:Fadapted}
    \begin{cases}
        F[W] \textup{ is a spanning forest of $\Gamma[W]$},\\
        F[W^\mathsf{c}] \textup{ is a spanning forest of $\Gamma[W^\mathsf{c}]$},\\
        \val_F(W)=1-\mathbbm{1}_\D(W),
    \end{cases}
\end{equation}
where $\mathbbm{1}_\D(W)$ equals $1$ if $W \in \D$ and it equals $0$ otherwise. Then 
\begin{equation}\label{E:nI}
    \begin{aligned} 
    \n^I:\BCon(\Gamma) & \to \ZZ\\
    W & \mapsto \n^I_W:=I(F)_W+b_1(\Gamma[W])+b_0(F[W])-1.
    \end{aligned}
\end{equation}
is a V-stability condition of degree $d$  on $\Gamma$ with $\D(\n^I)=\D$.
\end{lemma-definition}
\begin{proof}
    We follow the outline of the proof of \cite[Lemma~1.22]{viviani2023new}.
    Firstly, we show that $\n^I$ is well defined. This follows from 
    
    \underline{Claim 1:} For each pair of minimal spanning forests $F_1,F_2\in m\S\S_\D(\Gamma)$ that are adapted to $W$, we have
    $$
    I(F_1)_W+b_0(F_1[W])=I(F_2)_W+b_0(F_2[W]).
    $$

    We prove our claim by induction on the number $c_\D(\Gamma)$ of $\D$-admissible spanning forests of $\Gamma$. If $c_\D=1$, then $F_1=F_2$ and the claim is obvious.
    Before proving the general case, we consider the special case $c_\D(\Gamma[W])=c_\D(\Gamma[W^\mathsf{c}])=1$. We notice that if $\Gamma[W]$ (resp.$\Gamma[W^\mathsf{c}]$) has only one spanning forest $F_W$ (resp. $F_{W^\mathsf{c}}$), then it must be a tree. We distinguish two cases:
    \begin{itemize}
        \item If $W\in\D$, then $F_W\cup F_{W^\mathsf{c}}$ is the only minimal $\D$-admissible spanning forest. Hence the claim is obvious.
        \item If $W\notin\D$, we are in the situation of \cite[Lemma~1.22, Special case]{viviani2023new}, hence we conclude using Corollary~\ref{C:BDsetContractions}$(2)$ and Definition~\ref{D:BD-set}.
    \end{itemize}

    We are now ready to prove the general case, i.e. either $c_\D(\Gamma[W])>1$ or $c_\D(\Gamma[W^\mathsf{c}])>1$. Since $I(F_i)_W+I(F_i)_{W^\mathsf{c}}=d-b_1(\Gamma)-b_0(F_i)+1$, up to switching $W$ with $W^\mathsf{c}$ we can assume that $c_\D(\Gamma[W])>1$. We distinguish two cases:
    \begin{itemize}
        \item There exists only one spaning forest $F_W$ of $\Gamma[W]$. We can assume $W\notin\D$, since again the case of $W$ degenerate is trivial. Let $T_W$ be a spanning tree of $\Gamma[W]$ such that $F_W\leq T_W$. Then, as above, the reasoning of \cite[Lemma~1.22, Special case]{viviani2023new} applies with respect to $T_W$ and $F_{W^\mathsf{c}}$, the only spanning forest of $\Gamma[W^\mathsf{c}]$, which is a tree. Hence we have the claim by Corollary~\ref{C:BDsetContractions}$(2)$ and Definition~\ref{D:BD-set}.
        \item Let $F_1,F_2$ be two minimal spanning forests adapted to $W$, such that $F_1[W]\neq F_2[W]$. Then there exists an edge $e$ of $\Gamma[W]$ that is not a loop and such that $\Gamma\setminus \{e\}$ is still $\D$-admissible. 
        Consider the graph $\Gamma\setminus \{e\}$ (resp. $\Gamma/\{e\}$) and the $\D$ forest function $I_{\Gamma\setminus \{e\}}$ (resp. $I_{\Gamma/\{e\}}$) defined in \eqref{e:forestfundel} (resp. \eqref{e:forestfuncontr}). By Corollary~\ref{C:BDsetContractions}, since $BD_I$ is a weak numerical PT-assignment, both $BD_{I_{\Gamma\setminus \{e\}}}$ and $BD_{I_{\Gamma/\{e\}}}$ are numerical PT-assignments. Moreover, since $c_\D(\Gamma\setminus \{e\}),c_\D(\Gamma/\{e\})<c_\D(\Gamma\setminus \{e\})+c_\D(\Gamma/\{e\})=c_\D(\Gamma)$, Claim 1 holds for $I_{\Gamma\setminus \{e\}}$ and $I_{\Gamma/\{e\}}$ by the induction hypothesis. We now distinguish three cases:
        \begin{itemize}
            \item If neither $F_1$ nor $F_2$ contain $e$, then they are minimal spanning forests of $\Gamma\setminus \{e\}$ and, by the induction hypothesis,
            $$
            I(F_1)_W+b_0(F_1[W])=I_{\Gamma\setminus \{e\}}(F_1)_W+b_0(F_1[W])=I_{\Gamma\setminus \{e\}}(F_2)_W+b_0(F_2[W])=I(F_2)_W+b_0(F_2[W]).
            $$
            \item If both $F_1$ and $F_2$ contain $e$, then their respective images $\ov F_1$ and $\ov F_2$ in $\Gamma/\{e\}$ are minimal spanning forests. Hence, by the induction hypothesis,
            $$
            I(F_1)_W+b_0(F_1[W])=I_{\Gamma/\{e\}}(\ov F_1)_{\ov W}+b_0(\ov F_1[\ov W])=I_{\Gamma/\{e\}}(\ov F_2)_{\ov W}+b_0(\ov F_2[\ov W])=I(F_2)_W+b_0(F_2[W]),
            $$
            where $\ov W\subset V(\Gamma/\{e\})$ is the image of $W\subset V(\Gamma)$.
            \item Suppose $F_1$ contains $e$ and $F_2$ does not (the opposite case being analogous). 

            We recall that, for each spanning forest $F$ of $\Gamma$ there exists a unique minimal spanning forest $F_m$ such that $F_m\leq F$.

            Consider the graph $F_1\cup F_2$. Since $F_1,F_2\leq F_1\cup F_2$, with $F_1\neq F_2\in m\S\S_\D(\Gamma)$, we have $b_1(F_1\cup F_2)\geq 1$. Moreover, we see that every edge $f\in F_1\setminus F_2$ is not separating. In fact, if such a separating edge $f$ existed, the connected components of the graph $(F_1\cup F_2)\setminus \{f\}$ would be elements of $\wh \D$. Hence, since $F_1\leq F_1\cup F_2$ is spanning, we would have $F_1\setminus \{f\}\in\S\S_\D(\Gamma)$. But this is in contradiction with the minimality of $F_1$.

            Therefore, there exist a (possibly empty) set of edges $\{g_1,\ldots,g_s\}\subset E(F_1)\setminus\{e\}$, such that the graph $G:=F_2\cup\{e,g_1,\ldots,g_s\}$ has exactly a cycle $\{e,f_1,\ldots,f_r,g_1,\ldots,g_s\}$, for some $f_1,\ldots,f_r\in E(F_2)$, with $r\geq 1$. 
            We observe that $m\S\S_\D(G)=\{F_2,\tilde{F}_1,\ldots,\tilde{F}_r\}$, where $\tilde{F}_i$ denotes the unique minimal spanning forest of the graph $G\setminus \{f_i\}$. Indeed, the minimal spanning forest of $G\setminus g_j$ is $F_2$, for any $j=1,\ldots,s$. 

            Since $e\in E(\tilde{F}_i)$, for each $i=1,\ldots,r$, by the previous case we have that 
            \begin{equation}\label{e:I(tildeF)}
                I(F_1)_W+b_0(F_1[W])=I(\tilde F_i)_W+b_0(\tilde F_i[W]),\quad \textup{for any } i=1,\ldots,r.
            \end{equation}
            By definition of $BD_I$, we have that
            \begin{equation*}
                BD_I(G)\supseteq \bigcup_{i=1}^rI(\tilde F_i)+D(\O(G\setminus \tilde F_i)),
            \end{equation*}
            and, by arguing as in \cite[Lemma~3.7]{PTgenus1}, we see that the inclusion is actually an equality. Therefore, since $\{f_1,\ldots,f_r\}\subset E(\Gamma[W])$, by \eqref{e:I(tildeF)}, we deduce that for each $D\in BD_I(G)$,
            \begin{equation}\label{e:DWtildeF}
                D_W=I(\tilde F_i)_W+|E(G)\setminus E(\tilde F_i)|+b_0(\tilde F_i[W])=I(F_1)_W+|E(G)\setminus E(F_1)|+b_0(F_1[W]).
            \end{equation}
            Since $BD_I(G)\supset I(F_2)+D(\O(G\setminus F_2))$ and $\{e,g_1,\ldots,g_s\}\subset E(\Gamma[W])$, we have that for each $D\in BD_I(G)$,
            \begin{equation}\label{e:DWF2}
                D_W=I(F_2)+|E(G)\setminus E(F_2)|=I(F_2)+s+1.
            \end{equation}
            Finally, observing that, for each $i=1,\ldots,r$, 
            \begin{equation*}
                b_0(F_2[W])+|E(F_2)|=b_0(\tilde F_i[W])+|E(\tilde F_i)|,
            \end{equation*}
            the combination of \eqref{e:DWtildeF} and \eqref{e:DWF2} gives
            $$
            I(F_1)_W+b_0(F_1[W])=I(F_2)_W+b_0(F_2[W])
            $$
            as required.
        \end{itemize}
    \end{itemize}
    Next, we show that $\n^I$ is a V-stability condition by checking that it satisfies the two conditions of Definition~\ref{D:Vstab}.

    \underline{Claim 2}: For any $W\in\BCon(\Gamma)$, we have
    $$
    \n^I_W+\n^I_{W^\mathsf{c}}+\val(W)-d=\begin{cases}
        0,\quad\textup{if }W\in\D,\\
        1,\quad\textup{if }W\notin\D.
    \end{cases}
    $$
    Let $F\in m\S\S_\D(\Gamma)$, such that $F[W]\in m\S\S_\D(\Gamma[W])$, $F[W^\mathsf{c}]\in m\S\S_\D(\Gamma[W^\mathsf{c}])$ and $\val_F(W)=1-\mathbbm{1}_\D(W)$. Then
    \begin{multline*}
        \n^I_W+\n^I_{W^\mathsf{c}}=I(F)_W+b_1(\Gamma[W])+b_0(F[W])-1+I(F)_{W^\mathsf{c}}+b_1(\Gamma[W^\mathsf{c}])+b_0(F[W^\mathsf{c}])-1=\\
        =\deg(I)+b_1(\Gamma)-\val(W)+b_0(F)-\mathbbm{1}_\D(W)=d-\val(W)-\mathbbm{1}_\D(W)+1,
    \end{multline*}
    where we have used that $b_1(\Gamma)=b_1(\Gamma[W])+b_1(\Gamma[W^\mathsf{c}])+\val(W)-1$, that $I(F)\in\Div^{d-b_1(\Gamma)-b_0(F)+1}(\Gamma)$ and that $b_0(F)=b_0(F[W])+b_0(F[W^\mathsf{c}])-1+\mathbbm{1}_\D(W)$.

    \underline{Claim 3}: For any pairwise disjoint $W_1,W_2,W_3\in\BCon(\Gamma)$ such that $W_1\cup W_2\cup W_3=V(\Gamma)$, we have that
    \begin{equation*}
            \sum_{i=1}^{3}\n^I_{W_i}+\sum_{1\leq i<j\leq 3} \val(W_i,W_j)-d \in\begin{cases}
                \{1,2\}, \textup{ if $W_i\notin\D$ for all $i$};\\
                \{1\}, \textup{ if there exists a unique $i$ such that $W_i\in\D$};\\
                \{0\}, \textup{ if $W_i\in\D$ for all $i$}.
            \end{cases}
        \end{equation*}
    For $i=1,2,3$, fix a minimal spanning forest $F_i\leq \Gamma[W_i]$. We distinguish three cases:
    \begin{itemize}
        \item Assume $W_i\notin \D$ for each $i=1,2,3$. Fix an edge $e_{ij}\in E(W_i,W_j)$ for each $1\leq i\neq j\leq 3$, so that $F_i\cup F_j\cup\{e_{ij}\}\in m\S\S_\D(\Gamma[W_i\cup W_j])$, and consider the graph 
        $$
        C:=F_1\cup F_2\cup F_3\cup \{e_{12},e_{23},e_{13}\}.
        $$
        We have 
        \begin{multline}
            \sum_{i=1}^{3}\n^I_{W_i}+\sum_{1\leq i<j\leq 3} \val(W_i,W_j)-d=\sum_{i=1}^3(I(C\setminus \{e_{(i-1)i}\})_{W_i}+b_1(\Gamma[W_i])+b_0(F_i)-1)+\\+\sum_{1\leq i<j\leq 3} \val(W_i,W_j)-d
            =\sum_{i=1}^2(I(C\setminus \{e_{(i-1,i)}\})_{W_i}+b_1(\Gamma[W_i])+b_0(F_i)-1)+\\+\sum_{1\leq i<j\leq 3} \val(W_i,W_j)-b_1(\Gamma)-b_0(C\setminus \{e_{23}\})+b_1(\Gamma[W_3])+b_0(F_3)+2-I(C\setminus \{e_{23}\})_{W_1\cup W_2}=\\
            =-I(C\setminus \{e_{23}\})_{W_1\cup W_2}+I(C\setminus \{e_{13}\})_{W_1}+I(C\setminus \{e_{12}\})_{W_2}+2,
        \end{multline}
        where we have used that $b_1(\Gamma)=\sum_{i=1}^3b_1(\Gamma[W_i])+\sum_{1\leq i<j\leq3}\val(W_i,W_j)-3+1$, and $b_0(C)=\sum_{i=1}^3b_0(F_i)$.

        Fix a spanning tree $T$ of $C$, such that $F_i\leq T[W_i]=:T_i$. The claim follows by arguing as in Claim~3 of the proof of \cite[Lemma~1.22]{viviani2023new}. 
        \item Assume $W_i\in\D$ iff $i=3$. Let $e_{12}\in E(W-i,W_j)$ be such that $F_1\cup F_2\cup \{e_{12}\}\in m\S\S_\D(\Gamma[W_i\cup W_j])$ and consider the graph 
        $$
        C:=F_1\cup F_2\cup F_3\cup \{e_{12}\}
        $$
        in $m\S\S_\D(\Gamma)$. Reasoning as above, we have 
        $$
        \sum_{i=1}^{3}\n^I_{W_i}+\sum_{1\leq i<j\leq 3} \val(W_i,W_j)-d=-I(C)_{W_1\cup W_2}+I(C)_{W_1}+I(C)_{W_2}+1=1,
        $$
        where we have used that $b_0(C)=\sum_{i=1}^3b_0(F_i)-1$.
        \item Assume $W_i\in\D$ for each $i=1,2,3$. Then the graph 
        $$
        C:=F_1\cup F_2\cup F_3
        $$
        is in $m\S\S_\D(\Gamma)$. Again as above, we obtain 
        $$
        \sum_{i=1}^{3}\n^I_{W_i}+\sum_{1\leq i<j\leq 3} \val(W_i,W_j)-d=-I(C)_{W_1\cup W_2}+I(C)_{W_1}+I(C)_{W_2}=0,
        $$
        since $b_0(C)=\sum_{i=1}^3b_0(F_i)$.
    \end{itemize}
\end{proof}

We are now ready to state and prove the main result of this section. 
\begin{theorem}\label{T:V=PT} 
Let $\Gamma$ be a connected graph.
    \begin{enumerate}
        \item\label{T:V=PT1}For any degeneracy subset $\D$ on $\Gamma$ and any $d\in \ZZ$, the  maps 
        \begin{align*}
            \left\{\begin{aligned}\textup{V-stability conditions on $\Gamma$}
            \end{aligned}\right\} &\leftrightarrow   \left\{\begin{aligned}\textup{BD-sets on $\Gamma$ that are} \\ \textup{weak numerical PT-assignments}\end{aligned}\right\}\\
            \n&\mapsto \BD_{I_\n}\\
            \n^I & \mapsfrom \BD_I 
        \end{align*}
        (defined in Lemma-Definitions~\ref{L:BDInwelldef}  and \ref{L:nIwelldef} respectively) are bijections, one the inverse of the other, that preserve the degree and the degeneracy subset, and such that  $\P_\n=\BD_{I_\n}$.
        \item\label{T:V=PT2} For a degeneracy subset $\D\in \Deg(\Gamma)$ and a subset $\P\subset\OO^d_\D(\Gamma)$,  the following are equivalent:
        \begin{enumerate}
            \item\label{T:V=PT2a} There exists a V-stability condition $\n$ such that $\D=\D(\n)$ and $\P=\P_\n$;
            \item\label{T:V=PT2b} $\P$ is a PT-assignment with degeneracy set $\D$;
            \item\label{T:V=PT2c} $\P$ is a numerical PT-assignment with degeneracy set $\D$;
            \item\label{T:V=PT2d} $\P$ is a BD-set with degeneracy set $\D$, which is also a weak PT-assignment;
            \item\label{T:V=PT2e} $\P$ is a BD-set  with degeneracy set $\D$, which is also a weak numerical PT-assignment.
        \end{enumerate}
    \end{enumerate}
\end{theorem}
This is proved in \cite[Theorem~1.20]{viviani2023new} for $\D=\emptyset$.
\begin{proof}
    We will follow the strategy of the proof of  \cite[Theorem~1.20]{viviani2023new}. 
    
    Let us first prove Part \eqref{T:V=PT1}. First of all, observe that the two maps  preserve the degree and the degeneracy subset by Lemmas \ref{L:BDInwelldef} and \ref{L:nIwelldef}.
    We conclude the proof of Part \eqref{T:V=PT1} by virtue of the  following two Claims. 
    
    \underline{Claim 1}: For any V-stability $\n$ on $\Gamma$, $BD_{I_\n}$ is a weak numerical PT-assignment and $\P_\n=BD_{I_\n}$.

    Proof of the claim. For any spanning forest $F=\Gamma\setminus S\in m\S\S(\Gamma)$, Lemma~\ref{L:Vstabs1divisor} below gives that 
    $$
    \P_\n(F)=\{D(F)\},
    $$
    for some divisor $D(F)\in\Div^{d-b_1(\Gamma)-b_0(F)+1}(F)$. Let $F=\bigsqcup_{i=1}^k F_i$ be the decomposition in connected components. For every $i$ and every edge $e\in F_i$, let $$F\setminus \{e\}=:F_i[W_e]\sqcup F_i[V(F_i)\setminus W_e]\bigsqcup_{j\neq i}F_j.$$ Notice that this identifies the sets of vertices $W_e$ and $V(F_i)\setminus W_e$ up to swapping. Then, we have that $F[W_e]$ and $F[V(F_i)\setminus W_e]$ are minimal spanning forests respectively for $\Gamma[W_e]$ and for $\Gamma[V(F_i)\setminus W_e]$, and we have that $\val_F(W_e)=1$. By Lemma-Definition~\ref{LD:V-subgr} and Remark~\ref{R:prop-Pn}(\ref{R:prop-Pn2}), we see that
    \begin{equation}\label{E:V=PT1}
        D_{W_e}\geq\n_{W_e}-e_S(W_e)=\n(F_i)_{W_e},
    \end{equation}
    and, analogously, Remark~\ref{R:prop-Pn}(\ref{R:prop-Pn1}) gives
    \begin{equation}\label{E:V=PT2}
        D_{W_e}\leq\n_{W_e}-e_S(W_e)=\n(F_i)_{W_e},
    \end{equation}
    since $W_e\notin\wh\D(\n)$.
    Moreover, by Remark~\ref{R:prop-Pn}(\ref{R:prop-Pn3}), we see that
    \begin{equation}\label{E:V=PT3}
        D_{V(F_i)}+e_S(V(F_i))=\n_{V(F_i)},
    \end{equation}
    and, combining \eqref{E:V=PT1}, \eqref{E:V=PT2} and \eqref{E:V=PT3}, up to swapping $W_e$ with $V(F_i)\setminus W_e$, we deduce that 
    $$
    \P_\n(F)=\{I_\n(F)\}.
    $$

    Hence, since $\P_\n$ is an upper subset by Proposition~\ref{P:n-ss}(\ref{P:n-ss1}), by the definition~\ref{D:BD-set} of BD-sets, we have that 
    \begin{equation}\label{E:V=PT5}
        BD_{I_\n}\subset\P_\n.
    \end{equation}
    Now, Lemma~\ref{L:BDrestr} and Proposition~\ref{BDthm} imply that 
    \begin{equation}\label{E:V=PT6}
        |BD_{I_\n}(G)|\geq c_{\D(\n)}(G),
    \end{equation}
    for each $G\in\S\S_{\D(\n)}(\Gamma)$, while Corollary~\ref{c:cardVstab} gives
    \begin{equation}\label{E:V=PT7}
        |\P_\n(G)|\leq c_{\D(\n)}(G).
    \end{equation}
    The combination of \eqref{E:V=PT5}, \eqref{E:V=PT6} and \eqref{E:V=PT7} proves Claim 1.

    \underline{Claim 2}: Any BD set that is also a weak numerical PT-assignment is of the form $BD_{I_\n}$, for a unique V-stability condition $\n$.

    Let $I$ be a $\D$ forest function, and let $\n^I$ be the V-stability condition with $\D(\n^I)=\D$ defined in \eqref{E:nI}. Let $F\in m\S\S_\D(\Gamma)$, with decomposition into irreducible components $F=\bigsqcup_{i=1}^kF_i$. Firstly, we observe that, since $V(F_i)\in\wh\D(\n^I)$, for any $i=1,\ldots,k$, Definition~\ref{D:ext-n}(\ref{D:ext-n2}) shows that 
    \begin{equation*}
        \n_{V(F_i)}^I=I(F)_{V(F_i)}+b_1(\Gamma[V(F_i)]),
    \end{equation*}
    and, by Lemma-Definition~\ref{LD:V-subgr}, that
    \begin{equation}\label{E:V=PT8}
        I_{\n^I}(F)_{V(F_i)}=|\n^I(F_i)|=I(F)_{V(F_i)}.
    \end{equation}

    For each connected component $F_i$, and each $e\in F_i$, let $F\setminus \{e\}=F_i[W_e]\sqcup F_i[V(F_i)\setminus W_e]\bigsqcup_{j\neq i}F_j$ be as above.
    Then, combining \eqref{E:nI} and \eqref{E:In}, we see that
    
    \begin{equation}\label{E:V=PT9}
        I_{\n^I}(F)_{W_e}=\n^I(F_{i})_{W_e}=\n^I_{W_e}-e_S(W_e)=\n^I_{W_e}-b_1(\Gamma[W_e])=I(F)_{W_e}.
    \end{equation}
    Therefore, combining \eqref{E:V=PT8} and \eqref{E:V=PT9}, we obtain $I=I_{\n^I}$, and thus
    $$
    BD_I=BD_{I_{\n^I}}.
    $$
    We are left to prove that $\n^I$ is indeed the unique V-stability condition such that $I=I_{\n^I}$. Let $\n$ be a V-stability condition such that $I_\n=I$. Let $W\in\BCon(\Gamma)$ and let $F\in m\S\S_\D(\Gamma)$ be adapted to $W$, as in \eqref{E:Fadapted}. Then, by applying \eqref{E:In} and using that $I_\n=I=I_{\n^I}$, we obtain
    $$
    \n_W-b_1(\Gamma[W])-b_0(F[W])+1=I(F)_W=\n^I_W-b_1(\Gamma[W])-b_0(F[W])+1,
    $$
    and hence $\n=\n^I$.

    Now, we prove Part \eqref{T:V=PT2} via the following cycle of implications:
    \begin{itemize}
        \item (\ref{T:V=PT2a}$)\implies$(\ref{T:V=PT2b}): By Proposition~\ref{P:n-ss}, $\P_\n$ is an upper subset of $\OO_\n^d(\Gamma)$. Consider a $\D(\n)$-admissible spanning subgraph $G=\Gamma\setminus S\leq\Gamma$. Corollary~\ref{c:cardVstab} implies that the map 
        \begin{equation}\label{E:piG}
            \pi_{G}:\P_\n(G)\subset \Div^{d-|S|}(G)\twoheadrightarrow\Pic^{\delta_G}(G)
        \end{equation}
        satisfies Property~\ref{D:PTstabProp2} of Definition~\ref{D:PTstab}. Hence, by the proof of Proposition~\ref{P:cardPTstab}, the map
        $$
        K_G:\P_\n(G)\hookrightarrow\S\F_{\D(\n)}(G)
        $$
        is injective for each $G\in\S\S_\D(\Gamma)$. Since $\P_\n$ is a BD-set and a weak numerical PT-assignment by part \eqref{T:V=PT1}, Proposition~\ref{BDthm}(\ref{BDthm2}a) implies that 
        $$
        |\P_\n(G)|=c_{\D(\n)}(G)=|\S\F_{\D(\n)}(G)|,
        $$
        hence the map $K_G$ is a bijection. We observe that, by construction of $K_G$, for every pair of divisors $D_1,D_2\in\P_\n(G)$ such that $\pi_G(D_1)=\pi_G(D_2)$, if $K_G(D_1)\in\S\T(G)$, then  $K_G(D_2)\notin\S\T
        (G)$, and viceversa. By surjectivity of $K_G$, for every spanning tree $T\in\S\T(G)$, there exists a divisor $D(T)\in\P_\n$ such that $K_G(D(T))=T$, and for every pair of trees $T,T'\in\S\T(G)$, we have $\pi_G(D(T))\neq\pi_G(D(T'))$. Therefore, we deduce that the map \eqref{E:piG} is surjective, and $\P_\n$ is a PT assignment.
        \item (\ref{T:V=PT2b}$)\implies$(\ref{T:V=PT2c}):  follows from Proposition~\ref{P:cardPTstab}.
        \item (\ref{T:V=PT2c}$)\implies$(\ref{T:V=PT2d}):  follows from Corollary~\ref{C:PTstabisBD}.
        \item (\ref{T:V=PT2d}$)\implies$(\ref{T:V=PT2e}): follows from Proposition~\ref{P:cardPTstab}.
        \item (\ref{T:V=PT2e}$)\implies$(\ref{T:V=PT2a}): follows from Part~\ref{T:V=PT1}.

    \end{itemize}
    \end{proof}
The next Lemma completes the proof of the above.
\begin{lemma}\label{L:Vstabs1divisor}
    Let $\n$ be a V-stability condition on $\Gamma$. If $F\in m\S\S_{\D(\n)}(\Gamma)$, then $\P_\n(F)$ contains exactly one element.
\end{lemma}
\begin{proof}
    The strategy of the proof is similar to \cite[Lemma~1.21]{viviani2023new}.

    By Corollary~\ref{c:cardVstab}, and the fact that $c_{\D(\n)}(F)=1$ for a minimal spanning forest $F$, it is enough to show that $\P_\n(F)\neq\emptyset$. 

    We will prove this by induction on the number of edges of $\Gamma\setminus S:=F$. 
    \begin{itemize}
        \item[Case 1:] $|E(F)|=0$. Let $V(\Gamma)=\{v_1,\ldots,v_k\}$. Then $F=\bigsqcup_{i=1}^k\{v_i\}$ and $\{v_i\}\in\wh\D(\n)$ for each $i=0,\ldots,k$.
        We define the divisor $D$ on $\Gamma$ by $D_v=\n_v$ for each $v\in V(\Gamma)$, and we show that $D\in\P_\n(F)$. 
        
        For each $W\in\BCon(\Gamma)$, we have that $W\in\D(\n)$. We prove that $D_W=\n_W-e_S(W)$ by induction on $|V(W)|$, the case of $W=\{v\}$ being trivial. Pick a vertex $\ov v\in W$ and write $W=\{\ov v\}\cup W\setminus\{\ov v\}$. By Lemma~\ref{L:add-whn}, we have 
        $$
        D_W=D_{\ov v}+D_{W\setminus\ov \{v\}}=\n_{\ov v}+\n_{W\setminus\{\ov v\}}-e_S(W\setminus\{\ov v\})=
        $$
        $$=\n_W-\val(\{\ov v\},W\setminus\{\ov v\})-e_S(W\setminus\{\ov v\})=\n_W-e_S(W).
        $$

        \item[Case 2:] $|E(F)|=h+1$ for some $h \in \mathbb{N}$. We proceed in complete analogy to \cite[Lemma~1.21]{viviani2023new}. Pick an end vertex $v$ of $\Gamma\setminus S$, let $e$ be the unique edge of $\Gamma\setminus S$ incident to $v$ and let $w$ be the other vertex incident to $e$. 
        Consider the contraction of $e$ 
$$
f:\Gamma\to \Gamma/\{e\}:=\wt \Gamma,
$$
which sends the two vertices $v,w$ into the vertex $\wt w$. We will identify the set $V(\Gamma)\setminus \{v,w\}$ with the set $V(\wt \Gamma)\setminus \{\wt w\}$. 

The image $f_*(\Gamma\setminus S)=\wt \Gamma\setminus \wt S$, where $\wt S=(f^E)^{-1}(S)$,  is a minimal spanning forest of $\wt \Gamma$ and the restriction 
$$
f_{\Gamma\setminus S}:\Gamma\setminus S\to \wt \Gamma \setminus \wt S
$$
is again the contraction of the edge $e$. Proposition~\ref{P:Pn-func} gives a well-defined map 
$$
f_*:\P_{\mathfrak n}(\Gamma\setminus S)\to \P_{f_*\mathfrak n}(\wt \Gamma \setminus \wt S).
$$
By our induction assumption, the set $\P_{f_*\mathfrak n}(\wt \Gamma \setminus \wt S)$ consists of a unique divisor $\wt D$. 
We want to show $\wt D$ is the image via $f_*$ of an element of $\P_{\mathfrak n}(\Gamma\setminus S)$. 

With this aim, we introduce the following divisor on $\Gamma$: 
$$
D_t=
\begin{cases}
\mathfrak \n(\Gamma[V(F_v)])_v & \text{ if } t=v, \\
\wt D_w-\mathfrak \n(\Gamma[V(F_v)])_v & \text{ if } t=w, \\
\wt D_t & \text{ if } t\neq v, w,
\end{cases}
$$
where $F_v$ denotes the connected component of $F$ containing $v$. Notice that if $v\in\BCon(\Gamma)$, then $\n(\Gamma(V(F_v)))_v=\n_v$.
    \end{itemize}

We conclude by showing the following

\un{Claim:} $D\in \P_{\mathfrak n}(\Gamma\setminus S)$.

Indeed, first of all we have that 
$|D|+|S|=|\wt D|+|\wt S|=|f_*\mathfrak n|=|\mathfrak n|$. 
Consider next a non-trivial biconnected subset $W\subset V(\Gamma)$ and let us check that 
\begin{equation}\label{E:stabV}
D_W+e_S(W) \geq \mathfrak n_W.
\end{equation}
We will distinguish several cases according to whether or not $v$ and $w$ belong to $W$. Using Remark~\ref{R:prop-Pn}\eqref{R:prop-Pn1}, we can assume that either $v,w\in W$ or $v\in W\not\ni w$, provided that we prove both the inequalities of \eqref{E:stabV} and of \eqref{E:uppbound} for such $W$'s.

$\bullet$ Case I: $v,w\in W$

We have that $W=f_V^{-1}(\wt W)$ with $\wt w\in \wt W\subset V(\wt \Gamma)$ biconnected non-trivial, and $W$ is $\n$-degenerate if and only if $\wt W$ is $f_*\n$-degenerate. We compute 
\begin{multline}
\mathfrak n_W= (f_*\mathfrak n)_{\wt W}\leq D_W+e_S(W)=\wt D_{\wt W}+e_{\wt S}(\wt W)\leq
\\ \leq \begin{cases}
    (f_*\mathfrak n)_{\wt W}-1+\val_{\Gamma\setminus S}(\wt W)=\mathfrak n_W-1-\val_{\Gamma\setminus S}(W), & \text{if $W$ is $\n$-nondegenerate}\\
    (f_*\mathfrak n)_{\wt W}+\val_{\Gamma\setminus S}(\wt W)=\mathfrak n_W-\val_{\Gamma\setminus S}(W), & \text{if $W$ is $\n$-degenerate}
\end{cases}.
\end{multline}

$\bullet$ Case II: $W=\{v\}$. If $W$ is biconnected, then it is $\n$-nondegenerate, since $F$ is minimal.

We have that 
$$
\mathfrak  n_v=D_v+e_S(v)=\mathfrak n_v-1+\val_{\Gamma\setminus S}(v)=\mathfrak n_v.
$$

$\bullet$ Case III: $\{v\}\subsetneq W\not\ni w$.

Let $Z:=W\setminus \{v\}$ and consider the decomposition $\Gamma[Z]=\Gamma[Z_1]\coprod \ldots \coprod \Gamma[Z_k]$ into connected components.
Note that $Z_i^\mathsf{c}=\bigcup_{j\neq i} Z_j \cup \{v\} \cup W^\mathsf{c}$ is connected since $W^\mathsf{c}$ is connected, there is at least one edge between $v$ and $W^\mathsf{c}$ (namely $e$) and there is at least one edge between $v$ and $Z_j$ (since $W$ is connected). Hence each $Z_i$ is biconnected (and non-trivial).  Moreover, we can write $Z_i=f_V^{-1}(\wt Z_i)$ for some $\wt Z_i\subset V(\Gamma)$ non-trivial and biconnected. 
Up to reordering the sets, let $1\leq r\leq k$ be an index such that $Z_i\in\D(\n)$ iff $i> r$. We notice that $Z_i\in\D(\n)$ iff $\wt Z_i\in\D(f_*\n)$.
We set $\wt Z=\bigcup_i \wt Z_i$ and observe that $Z=f_V^{-1}(\wt Z)$. 

Now we compute 
\begin{equation}\label{E:equa1}
D_W+e_S(W)=D_v+D_Z+e_S(Z)+\val_S(v,Z)=\mathfrak n_v+\wt D_{\wt Z}+e_{\wt S}(\wt Z)+\val_{\Gamma}(v,Z),
\end{equation}
where we have used that $D_v=\mathfrak n_v-e_S(v)$ and that all the edges joining $v$ with $Z$ belong to $S$.

Since $\wt D$ belongs to $\P_{f_*\mathfrak n}(\wt \Gamma \setminus \wt S)$, we have that (for any $1\leq i \leq k$):
 \begin{equation*}\label{E:equa2}
\mathfrak n_{Z_i}=f_*(\mathfrak n)_{\wt Z_i}\leq \wt D_{\wt Z_i}+e_{\wt S}(\wt Z_i)\leq 
\begin{cases}
    f_*(\mathfrak n)_{\wt Z_i}-1+\val_{\wt \Gamma\setminus \wt S}(\wt Z_i)=\mathfrak n_{Z_i}-1+\val_{\Gamma\setminus S}(Z_i) & \text{for $i\leq r$}\\
    f_*(\mathfrak n)_{\wt Z_i}+\val_{\wt \Gamma\setminus \wt S}(\wt Z_i)=\mathfrak n_{Z_i}+\val_{\Gamma\setminus S}(Z_i) & \text{for $i>r$}.
\end{cases}
\end{equation*}
Summing the above inequalities over all indices $i=1,\ldots, k$, and using that $\val_{\Gamma}(Z_i,Z_j)=0$ for $i\neq j$, we find 
 \begin{equation}\label{E:equa3}
\sum_i \mathfrak n_{Z_i}\leq \wt D_{\wt Z}+e_{\wt S}(\wt Z)\leq \sum_i \mathfrak n_{Z_i}-r+\val_{\Gamma\setminus S}(Z).
\end{equation}
Since $\{v\}$ is $\n$-nondegenerate, arguing as in \cite[Lemma~4.6(b)]{FPV1} for each of the subgraphs $\{v\}\bigcup_{1\leq i\leq s} Z_i$ for $1\leq s\leq r$, and using again that $\val_{\Gamma}(Z_i,Z_j)=0$ for $i\neq j$, 
we get that 
\begin{equation}
-r\leq \mathfrak n_{\{v\}\bigcup_{i\leq r}Z_i}-\mathfrak n_v-\sum_{i=1}^r \mathfrak n_{Z_i}-\val_{\Gamma}(v,\bigcup_{i=1}^rZ_i)\leq 0. 
\end{equation}
Finally, we observe that $\{v\}\bigcup_{i\leq r}Z_i\in\D(\n)$ if and only if $W\in\D(\n)$. Either way, arguing as in \cite[Lemma~4.6(b)]{FPV1} for all the subgraphs $\{v\}\bigcup_{i\leq r}Z_i\bigcup_{r+1\leq i\leq s}Z_i$, for $r+1\leq s\leq k$, we obtain 
\begin{equation}\label{E:equa5}
    -r\leq \n_W-\n_v-\sum_{i=1}^k\n_{Z_i}-\val_\Gamma(v,Z)\leq 0.
\end{equation}
Combining \eqref{E:equa1}, \eqref{E:equa3} and \eqref{E:equa5}, and using that $\val_{\Gamma\setminus S}(v,Z)=0$ and $\val_{\Gamma\setminus S}(v,W^\mathsf{c})=1$, 
we get the desired inequalities
\begin{equation*}\label{E:equa6}
\mathfrak n_W\leq D_W+e_S(W)\leq \mathfrak n_W+\val_{\Gamma\setminus S}(Z)=\mathfrak n_W-1+\val_{\Gamma\setminus S}(W).
\end{equation*}
\end{proof}

\section{Compactified Jacobians}\label{Sec:cJ}
In this section we define \emph{compactified Jacobians} and \emph{smoothable compactified Jacobians} for nodal curves. We start the section by fixing the notation, and recalling some geometric properties of the moduli space of rank~$1$ torsion free sheaves on a fixed nodal curve. In particular, we study how such sheaves specialize. 

Then we introduce, in Definition~\ref{D:VcJ}, compactified Jacobians arising from a V-stability condition (introduced in Definition~\ref{D:Vstab}), and prove in Theorem~\ref{T:VcJ-smoo} that these are indeed smoothable compactified Jacobians. Finally, in Theorem~\ref{T:cla-smoo} we prove the converse:  all smoothable compactified Jacobians arise from some V-stability condition.

\subsection{Notation on nodal curves}\label{sub:not-nodal}

Let $X$ be a nodal curve over $k=\ov k$, i.e. a projective and reduced curve over an algebraically closed field $k$  having only nodes as singularities. We denote by $g(X):=1-\chi(\O_X)$ the arithmetic genus of $X$.

The \emph{dual graph} of a nodal curve $X$, denoted by $\Gamma_X$, is the graph having one vertex for each irreducible component of $X$, one edge for each node of $X$ and such that an edge is adjacent to a vertex if the corresponding node belongs to the corresponding irreducible component. 
We will denote the irreducible components of $X$ by 
$$\{X_v\::\: v\in V(\Gamma_X)\},$$ 
and the nodes of $X$ by 
$$\Xsing:=\{n_e\: :\: e\in E(\Gamma_X)\}.$$ 
Note that $X$ is connected if and only if $\Gamma_X$ is connected. In general, we will denote by $\gamma(X)$ the number of connected components of $X$ (or of its dual graph $\Gamma_X$). 

A \emph{subcurve} $Y\subset X$ is a closed subscheme of $X$ that is a curve, or in other words $Y$ is the union of some irreducible components of $X$. 
Hence the subcurves of $X$ are in bijection with the subsets of $V(\Gamma_X)$:
\begin{equation}\label{E:sub-Vert}
 \begin{aligned}
   \left\{\text{Subsets of } V(\Gamma_X)\right\} & \leftrightarrow \left\{\text{Subcurves of} X \right\}  \\
   W &\mapsto X[W]:=\bigcup_{v\in W} X_v \:  \: \text{ for } W\subseteq V(\Gamma_X) \\
   V(\Gamma_Y) & \mapsfrom Y.
 \end{aligned}   
\end{equation}
We say that a subcurve $X[W]$ is non-trivial if $X[W]\neq \emptyset, X$, which happens if and only if $W$ is non-trivial.
The dual graph of $X[W]$ is equal to the induced subgraph $\Gamma_X[W]$. Hence, a subcurve $X[W]$ is connected if and only if $W\subseteq V(\Gamma_X)$ is connected. 

The complementary subcurve of $Y$ is 
$$Y^\mathsf{c}:=\ov{X\setminus Y}.$$
Note that $X[W]^\mathsf{c}=X[W^\mathsf{c}]$.  We say that a subcurve $X[W]$ is biconnected it if is connected and its complementary subcurve $X[W]^\mathsf{c}$ is connected, which happens if and only if $W$ is biconnected. The set of biconnected subcurves of $X$ is denoted by $\BCon(X)$, and it is in canonical bijection with $\BCon(\Gamma_X)$. We will often identify $\BCon(X)$ and $\BCon(\Gamma_X)$ in what follows, using the bijection \eqref{E:sub-Vert}.

We define the join and the meet of two subcurves by 
$$
\begin{sis}
& X[W_1]\vee X[W_2]:=X[W_1\cup W_2], \\
& X[W_1]\wedge X[W_2]:=X[W_1\cap W_2].
\end{sis}
$$
In other words, the join of two subcurves is simply their union, while the meet of two subcurves is the union of their common irreducible components. 

Given two subcurves $Y_1,Y_2$ of $X$ without common irreducible components (i.e. such that $Y_1\wedge Y_2=\emptyset$), we define 
by $|Y_1\cap Y_2|$ the cardinality of their intersection (which is a subset of $\Xsing$).  Note that 
$$
|X[W_1]\cap X[W_2]|=\val_{\Gamma_X}(W_1,W_2) \:\: \text{ for any  } \: W_1\cap W_2=\emptyset.
$$

Given a subset $S\subset E(\Gamma_X)$, we denote by  $X_S$ the \emph{partial normalization} of $X$ at the nodes corresponding to $S$ and by $\nu_S:X_S\to X$ the partial normalization morphism. The dual graph of $X_S$ is equal to 
$$\Gamma_{X_S}=\Gamma_X\setminus S.$$ 

\subsection{Torsion-free, rank-1 sheaves on nodal curves}\label{Sub:sheaves}

Let $X$ be a connected nodal curve over an algebraically closed field $k$. Let $I$ be a coherent sheaf on $X$. We say that $I$ is:
\begin{itemize}
\item  \emph{torsion-free}  if its associated points are generic points of $X$. Equivalently, $I$ is pure of dimension one (i.e. it does not have torsion subsheaves) and it has support equal to $X$.
\item \emph{rank-$1$} if $I$ is invertible on a dense open subset of $X$. 
\item  \emph{simple} if $\End(I) = k $.
\end{itemize}
Note that each line bundle on $X$ is torsion-free, rank-$1$ and simple. 

The \emph{degree} of a rank-$1$ sheaf $I$ is defined to be
\begin{equation}\label{E:deg-I}
\deg(I)=\chi(I)-\chi(\O_X)=\chi(I)-1+g(X).
\end{equation}

A torsion-free sheaf $I$ is locally free away from the nodes of $X$. We will denote by $\NF(I)$, and call it the \emph{non-free locus} of $I$, the set of nodes of $X$ at which $I$ is not free. We will denote by 
$G(I)$ the spanning subgraph $\Gamma_X\setminus \NF(I)$, and we will refer to it as the \emph{free subgraph} of $I$. 
If $I$ is a rank-$1$ torsion-free sheaf, then the stalk at a point $p\in X$ is equal to 
$$
I_p=
\begin{cases}
\O_{X,p} & \text{ if }p\not \in \NF(I), \\
\m_p & \text{ if }p\in \NF(I),
\end{cases}
$$ 
where $\m_p$ is the maximal ideal of the local ring $O_{X,p}$. A rank-$1$, torsion-free sheaf $I$ is equal to 
$$I=\nu_{\NF(I),*}(L_I),$$
for a uniquely determined line bundle $L_I$ on the partial normalization $\nu_{\NF(I)}:X_{\NF(I)}\to X$ of $X$ at $\NF(I)$. Indeed, the line bundle $L_I$ is given by the pull-back $\nu_{\NF(I)}^*(I)$ quotient out by its torsion subsheaf, and its degree is equal to 
$$
\deg L_I=\deg I-|\NF(I)|. 
$$

The  \emph{stack of torsion-free rank-$1$ sheaves}  on $X$ is denoted by $\TF_X$. More precisely, $\TF_X$ is the stack over $k$-schemes such that the fiber over a $k$-scheme $T$ is the groupoid of $T$-flat coherent sheaves $\I$ on $X\times_k T$ such that $\I|_{X\times t}$ is torsion-free, rank-$1$ sheaf on $X$ for any geometric point $t$ of $T$. The stack $\TF_X$ comes equipped with a universal sheaf $\I$ on $X\times_k \TF_X$. The automorphism group of $I\in \TF_X(k)$ is equal to   
\begin{equation}\label{E:Aut-I}
\Aut(I)=\Gm^{\gamma(X_{\NF(I)})}=\Gm^{\gamma(G(I))},
\end{equation}
where each copy of $\Gm$ acts by scalar multiplication on the corresponding connected component of $X_{\NF(I)}$.

In particular, $\Gm$ sits in the automorphism group of every $I\in \TF_X$ as the group of scalar multiplication and the $\Gm$-rigidification 
\begin{equation}\label{E:rig-TF}
    \TF_X\to \TF_X\fatslash \Gm
\end{equation}
is a trivial $\Gm$-gerbe.

The stack $\TF_X$ contains two open substacks
$$\PIC_X\subseteq \Simp_X\subseteq \TF_X,$$ 
where $\Simp_X$ is the open substack parametrizing simple  sheaves (i.e. those sheaves $I$ such that $\Aut(I)=\Gm$), and $\PIC_X$ is the Picard scheme of $X$, which parametrizes line bundles. 
 It follows from \eqref{E:Aut-I} that 
\begin{equation}\label{E:simp-I}
\begin{sis}
& I\in \Simp_X \Leftrightarrow X_{\NF(I)} \text{ is connected } \Leftrightarrow G(I)=\Gamma_X\setminus \NF(I) \text{ is connected, }\\
& I\in \PIC_X \Leftrightarrow X_{\NF(I)}=X  \Leftrightarrow G(I)=\Gamma_X\Leftrightarrow \NF(I)=\emptyset.\\
\end{sis}
\end{equation}
Note that $\Simp_X\fatslash \Gm$ is the biggest open algebraic subspace of $\TF_X\fatslash \Gm$.

We have a decomposition into connected components
\begin{equation}\label{E:Torsd}
\TF_X=\coprod_{d\in \ZZ} \TF^d_X
\end{equation}
where $\TF^\chi_X$ parametrizes sheaves of degree $d$. The decomposition \eqref{E:Torsd} induces the following decompositions 
$$
\Simp_X=\coprod_{d\in \ZZ} \Simp^d_X \: \text{ and }\: \PIC_X=\coprod_{d\in \ZZ} \PIC^d_X,
$$
although $\Simp^d_X$ and $\PIC^d_X$ are not necessarily connected.  

Here are some basic geometric properties of the stack $\TF_X$.

\begin{fact}\label{F:diagTF}
\noindent 
    \begin{enumerate}
        \item \label{F:diagTF1} The stack $\TF_X$ is quasi-separated and locally of finite type over $k$.
        \item \label{F:diagTF2} The diagonal of $\TF_X$ is affine and of finite presentation.
        \item \label{F:diagTF3}$\TF_X$ is a reduced stack  of pure dimension $g(X)-1$ with locally complete intersection singularities and its smooth locus is $\PIC_X$. 
    \end{enumerate}
\end{fact}
\begin{proof}
 The stack $\TF_X$ is an open substack of the stack of coherent sheaves on $X$. Hence the first two properties follow from \cite[\href{https://stacks.math.columbia.edu/tag/0DLY}{Lemma 0DLY}]{stacks-project} and \cite[\href{https://stacks.math.columbia.edu/tag/0DLZ}{Lemma 0DLZ}]{stacks-project}.

 The last property follows from the fact that the semiuniversal deformation ring of a sheaf $I\in \TF_X$ is given by (see \cite[Sec. 3]{CMKVlocal})
 $$\wh \bigotimes_{n\in\NF(I)} \frac{k[[x_n,y_n]]}{(x_ny_n)}\wh \bigotimes k[[t_1,\ldots, t_{g(X_{\NF(I)})}]],$$
 and the well-known fact that the stack $\PIC_X$ has dimension $g(X)-1$ (the minus ones comes from the fact that its generic stabilizer is $\mathbb{G}_m$).
\end{proof}

For each subcurve $Y$ of $X$, let $I_Y$ be the restriction $I_{|Y}$ of $I$ to $Y$ modulo torsion.
If $I$ is a torsion-free (resp. rank-$1$) sheaf on $X$, so is $I_Y$ on $Y$.
We let $\deg_Y (I)$ denote the degree of $I_Y$, that is, $\deg_Y(I) := \chi(I_Y )-\chi(\O_Y)$. The \emph{multidegree} of a rank-$1$ torsion-free sheaf $I$ on $X$ is the divisor on $\Gamma_X$ defined as the multidegree of the line bundle $L_I$ on $X_{\NF(I)}$, i.e.
$$D(I):=D(L_I)=\{D(I)_v:=\deg_{|(X_{\NF(I)})_v}(L_I)\: : \: v\in V(\Gamma)\}.$$
where, as usual, we have made the identification $V(\Gamma_{X_{\NF(I)}})=V(\Gamma_X\setminus \NF(I))=V(\Gamma_X)$.
It turns out that for the subcurve $X[W]$ of $X$ associated to $W\subset V(\Gamma_X)$, we have
\begin{equation}\label{E:deg-sub}
\deg_{X[W]}(I)=D(I)_{W}+e_{\NF(I)}(W). 
\end{equation}
In particular, the degree of $I$ on an irreducible component $X_v$ is given by 
$$
\deg_{X_v}(I)=D(I)_v+e_{\NF(I)}(v),
$$
and the total degree of $I$ is given by
\begin{equation}\label{E:deg-mdeg}
\deg(I)=|D(I)|+|\NF(I)|. 
\end{equation}

The \emph{generalized Jacobian} $\PIC^{\un 0}_X$ of $X$, i.e. the semiabelian variety parametrizing  line bundles of multidegre zero, acts on $\TF_X$ via tensor product and its orbits are described in the following well-known

\begin{fact}\label{F:orbits} (see e.g. \cite[Sec. 5]{meloviviani})
Let $X$ be a connected nodal curve. 
\begin{enumerate}
\item \label{F:orbits1}
The orbits of $\PIC_X^{\un 0}$ on $\TF_X$ are given by 
\begin{equation}\label{E:orbit1}
\TF_X(G,D):=\{I\in \TF_X\: :\: G(I)=G \: \text{ and }\: D(I)=D\}.
\end{equation}
In particular, we get a decomposition into disjoint $\PIC_X^{\un 0}$-orbits 
$$
\TF_X^d=\coprod_{(G,D)\in \OO^{d}(\Gamma_X)} \TF_X(G,D) \subset \TF_X=\coprod_{(G,D)\in \OO(\Gamma_X)} \TF_X(G,D) 
$$
\item \label{F:orbits3}
Each $\TF_X(G,D)$ is a locally closed irreducible substack of $\TF_X$ and, if we endow it with the reduced stack structure, there is an isomorphism (if we write $G=\Gamma_X\setminus S$)
$$
(\nu_S)_*: \PIC^{D}_{X_S}:=\{L\in \PIC(X_{S}): \:  D(L)=D\} \xrightarrow{\cong} \TF_X(\Gamma_X\setminus S, D). 
$$
Under this isomorphism, the action of $\PIC_X^{\un 0}$ factors through the quotient $\PIC_X^{\un 0}\twoheadrightarrow \PIC_{X_S}^{\un 0}$ followed by the tensor product action of $\PIC_{X_S}^{\un 0}$ on $ \PIC^{D}_{X_S}$.
\item \label{F:orbits4}
The closure of $\TF_X(G,D)$ is equal to 
$$
\ov{\TF_X(G,D)}=\coprod_{(G,D)\geq (G',D')} \TF_X(G',D'),
$$
where $\geq $ is defined in \eqref{E:leq-GD}. In particular, the poset of $\PIC_X^{\un 0}$-orbits of $\TF_X$ (resp. $\TF_X^d$) is isomorphic to $\OO(\Gamma_X)$ (resp. $\OO^d(\Gamma_X)$).
 \end{enumerate}
\end{fact}

\begin{remark}\label{R:op=up}
It follows from \eqref{E:simp-I} that the decomposition in Fact \ref{F:orbits}\eqref{F:orbits1} induces the following decomposition into disjoint $\PIC_X^{\un 0}$-orbits
$$
\Simp_X^d=\coprod_{(G,D)\in \OO_{\con}^d(\Gamma_X)} \TF_X(G,D) \subset \Simp_X=\coprod_{(G,D)\in \OO_{\con}(\Gamma_X)} \TF_X(G,D).
$$
\end{remark}

The previous Fact implies that any upper subset $\P$ of $\OO(\Gamma_X)$  determines an open subset 
\begin{equation}\label{E:opTF}
U(\P):=\coprod_{(G,D)\in P} \TF(G,D) \subset \TF_X.
\end{equation}
Conversely, any open subset $U$ of $\TF_X$ which is a union of orbits is of the form $U=U(\P)$ for a unique upper subset $\P:=\P(U)$. 
Note that:
\begin{itemize}
    \item $U(\P)$ is connected if and only if $\P$ is connected. 
    \item $U(\P)\subset \TF^d_X$ if and only if $\P\subset \OO^d(\Gamma_X)$. 
    \item $U(\P)$ is of finite type over $k$ if and only if $\P$ is finite.
 \end{itemize}

\subsubsection{Specializations}\label{S:spec}

In this subsection, we examine the specializations of torsion-free rank-$1$ sheaves on a fixed nodal curve $X$.

Let $R$ be  a discrete valuation $k$-ring with residue field $k$ and quotient field $K$, and we denote by $\val:K\to \ZZ\cup\{\infty\}$ the associated valuation. Set $B:=\Spec R$ with generic point $\eta:=\Spec K$ and special point $o:=\Spec k$.  

Consider a relative torsion-free rank-$1$ sheaf $\I\in \TF_X(B)$ on $X_B:=X\times_k B$ and denote by $\I_{\eta}\in \TF_X(\eta)$ its generic fiber and by $\I_o\in \TF_X(k)$ its special fiber.  
We now recall from \cite[Sec. 2.2]{viviani2023new} the definition of the $1$-cochain associated to $\I$.

Consider  the double dual $\L(\I)$ of the pull-back of $\I$ via the normalization map $\nu_B:\wt X_B:=\wt X\times_k B\to X_B$, which is a line bundle since it is reflexive on a regular $2$-dimensional scheme. 
As shown in \cite[Prop. 12.7]{Oda1979CompactificationsOT}, there is a canonical presentation of $\I$ of the form 
\begin{equation}\label{E:pres-I}
0\to \I \to (\nu_B)_*(\L(\I))\xrightarrow{a} \bigoplus_{n\in \NF(\I_{\eta})^\mathsf{c}}\O_{\{n\}\times B}\to 0.
\end{equation}
For any $n\in \NF(\I_{\eta})^\mathsf{c}=E(G(\I_\eta))$, choose an oriented edge $\ee$ whose underlying edge $e$ corresponds to the node $n$. 
Denote by $n_\ee^s$ and $n_\ee^t$ the two inverse images of $n$ under the normalization map $\nu:\wt X\to X$ in such a way that $n_\ee^s\in \wt X_{s(\ee)}$ and $n_\ee^t\in \wt X_{t(\ee)}$.
The restriction of $a$ to  $\{n\}\times B$ induces a surjection
$$
a_{|\{n\}\times B}: (\nu_B)_*(\L(\I))_{|\{n\}\times B}=\L(\I)_{|\{n_\ee^s\}\times B}\oplus \L(\I)_{|\{n_{\ee}^t\}\times B}\twoheadrightarrow \O_{\{n\}\times B}
$$
which defines an element $[x_{\ee}^s,x_{\ee}^t]\in \PP^1(R)$. Since the restriction $a_{|\{n\}\times \eta}$ is surjective on each of the two factors (see \cite[Prop. 12.7]{Oda1979CompactificationsOT}), the element $[x_\ee^s,x_{\ee}^t] \in \PP^1(R)$ is different from  $0$ and $\infty$. We set 
\begin{equation}\label{E:gammaI}
\gamma(\I)(\ee):=\val\left(\frac{x_{\ee}^t}{x_{\ee}^s}\right)\in \ZZ. 
\end{equation}
Note that $\gamma(\I)(\ee)=-\gamma(\I)(\ov \ee)$ by construction, and hence we get a well-defined element $\gamma(\I)\in \CC^1(G(\I_{\eta}),\ZZ)$, called the \emph{$1$-cochain associated to $\I$.} Intuitively, the integer $\gamma(\I)(\ee)$ measures the ''speed'' at which $\I_o$ is smoothened out in the direction of $\ee$.

We define the support of $\gamma(\I)$ as 
\begin{equation}\label{E:supp-gamma}
\supp \gamma(\I):=\{e\in E(G(\I_\eta))\: : \gamma(\I)(\ee)\neq 0 \text{ for some orientation } \ee  \text{ of } e\},
\end{equation}
and the \emph{orientation} associated to $\I$ (or to $\gamma(\I)$) as 
\begin{equation}\label{E:OI}
\O(\I):=\bigcup\{\ee: \gamma(\I)(\ee)>0\}.
\end{equation}
Note that $\supp \O(\I)=\supp \gamma(\I)$.

The orientation $\O(\I)$ associated to $\I$ allows us to 
describe the combinatorial type $(G(\I_o),D(\I_o))$ of the special fiber $\I_o$ in terms of the combinatorial type $(G(\I_\eta),D(\I_\eta))$.

\begin{lemma}\label{L:Io-Ieta}(see \cite[Lemma 2.6]{viviani2023new})
With the above notation, we have that 
$$\begin{aligned}
&G(\I_o)=G(\I_{\eta})\setminus \supp \O(\I), \\
&D(\I_o)=D(\I_{\eta})-D(\O(\I)).
\end{aligned}$$

In other words, we have that $(G(\I_\eta),D(\I_\eta))\geq_{\O(\I)} (G(\I_o),D(\I_o))$. 
\end{lemma}

 We now describe all the relative torsion-free rank-$1$ sheaves  $X_B$ that have the same generic fiber. 

With this aim, we define an action of $C^0(\Gamma_X,\ZZ)$ on the set $|\TF_X(B)|$ of isomorphism classes of the groupoid $\TF_X(B)$. Take $\I\in \TF_X(B)$ and $g\in C^0(\Gamma_X,\ZZ)$. The exact sequence of abelian groups
$$
0\to (R^*,\cdot) \to (K^*,\cdot) \xrightarrow{\val} \ZZ\to 0
$$
induces an exact sequence of $0$-cochains
$$
0\to C^0(\Gamma_X, R^*) \to C^0(\Gamma_X,K^*) \xrightarrow{C^0(\val)} C^0(\Gamma_X,\ZZ)\to 0.
$$
Pick a lift $\wt g\in C^0(\Gamma_X,K^*)$ of $g$. The element $\wt g$ induces an automorphism $\wt g_*$ of the line bundle $\L(\I_{\eta})$ which is the scalar multiplication by $\wt g(v)\in K^*$ on the irreducible component $(\wt X_v)_{\eta}:=\wt X_v\times_k K$ of $\wt X_{\eta}$ corresponding to the vertex $v$ of $\Gamma_X$. The automorphism $\wt g_*$ extends uniquely to an automorphism of the line bundle $\L(\I)$ on $\wt X_B$ and 
hence to an automorphism of the sheaf $(\nu_B)_*(\L(\I))$ on $X_B$, that we will also denote by $\wt g_*$. Consider now the presentation \eqref{E:pres-I} of $\I$ and define a new element $\wt g(\I)$ of $\TF_X(B)$ as  follows:
\begin{equation}\label{E:pres-gI}
0\to \wt g(\I):=\ker(a\circ \wt g_*) \to (\nu_B)_*(\L(\I))\xrightarrow{a\circ \wt g_*} \bigoplus_{n\in \NF(\I_{\eta})^\mathsf{c}}\O_{\{n\}\times B}\to 0.
\end{equation}
Arguing as in the proof of \cite[Prop. 12.3]{Oda1979CompactificationsOT}, it follows that, if $\wt g'$ is another lift of $g$ (so that $\wt g-\wt g'\in C^0(\Gamma_X,R^*)$), then $\wt g(\I)$ is isomorphic to $\wt g'(\I)$ in $X_B$. 
Hence, there is a well-defined action 
\begin{equation}\label{E:action}
\begin{aligned}
C^0(\Gamma_X,\ZZ)\times |\TF_X(B)|& \longrightarrow |\TF_X(B)|\\
(g,\I) & \mapsto g(\I):=\wt g(\I).
\end{aligned}
\end{equation}
Indeed, this action coincides with the action given by \cite[Eq. (3.7)]{FPV1}.

\begin{proposition}\label{P:non-sepa}(\cite[Prop. 2.7]{viviani2023new}, \cite[Thm. 3.8]{FPV1})
\noindent 
\begin{enumerate}
\item \label{P:non-sepa1} 
The action \eqref{E:action} induces a bijection
$$
\begin{aligned}
|\TF_X(B)|/C^0(\Gamma_X,\ZZ) & \xrightarrow{\cong} |\TF_X(\eta)|\\
[\I] &\mapsto \I_{\eta}.
\end{aligned}
$$
\item  \label{P:non-sepa2}
The $1$-cochains associated to conjugate elements under the action \eqref{E:action}  satisfy the relation
$$
\gamma(g(\I))=\gamma(\I)+\delta_{G(\I_{\eta})}(g)\in \CC^1(G(\I_\eta),\ZZ),
$$
where $g$ is interpreted as an element of $C^0(G(\I_{\eta}),\ZZ)$ using that $V(G(\I_{\eta}))=V(\Gamma_X)$. 
 \end{enumerate}
\end{proposition}

We end this subsection by examining the relation between two different specialization of the same sheaf over $\eta\in B$.

\begin{proposition}\label{P:2limits}
Let  $\I^1,\I^2\in \TF_X(B)$ such that $\I^1_\eta=\I^2_\eta$. Fix $g\in C^0(\Gamma_X,\ZZ)$ such that $\I^2=g(\I^1)$ (such a $g$ exists by Proposition \ref{P:non-sepa}). Let $Y$ be a subcurve of $X$ such that $g$ is constant on $V(\Gamma_Y)\subseteq V(\Gamma_X)$. Then we have that 
$$
(\I^1_0)_Y\otimes \O_Y(\sum_{\substack{\ee \in \O(\I^1):\\ t(\ee)\in V(\Gamma_Y)\\ s(\ee)\not \in V(\Gamma_Y)}} n_{|\ee|})=
(\I^2_0)_Y\otimes \O_Y(\sum_{\substack{\ee \in \O(\I^2):\\ t(\ee)\in V(\Gamma_Y)\\ s(\ee)\not \in V(\Gamma_Y)}} n_{|\ee|}).
$$
\end{proposition}
\begin{proof}
 Consider the presentations \eqref{E:pres-I} of $\I^1$ and \eqref{E:pres-gI} of $g(\I^1)=\I^2$: 
\begin{equation}\label{E:pres-I12}
\begin{sis}
& 0\to \I^1 \to (\nu_B)_*(\L(\I^1))\xrightarrow{a} \bigoplus_{n\in \NF(\I^1_{\eta})^\mathsf{c}}\O_{\{n\}\times B}\to 0, \\
& 0\to \I^2=g(\I^1) \to (\nu_B)_*(\L(\I^1))\xrightarrow{a\circ \wt g_*} \bigoplus_{n\in \NF(\I^1_{\eta})^\mathsf{c}}\O_{\{n\}\times B}\to 0,
\end{sis}
\end{equation}
   where $\wt g_*$ is the automorphism of $\L(\I^1)$ which is the scalar multiplication by $\pi^{g(v)}$ on $\wt X_v\times B$. 
  
By restricting to the special fiber and using that $\I^1_o$ and $\I^2_o$ are torsion-free, we get the two presentations
\begin{equation}\label{E:pres-I12o}
\begin{sis}
& 0\to \I^1_o \to \nu_*(\L(\I^1)_o)\xrightarrow{a_o} \bigoplus_{n\in \NF(\I^1_{\eta})^\mathsf{c}}\O_{\{n\}}\to 0, \\
& 0\to \I^2_o \to \nu_*(\L(\I^1)_o)\xrightarrow{(a\circ \wt g_*)_o} \bigoplus_{n\in \NF(\I^1_{\eta})^\mathsf{c}}\O_{\{n\}}\to 0.
\end{sis}
\end{equation}
By restricting these two presentations to $Y\subset X$ and using that $\wt g$ is constant on $\wt Y$, we deduce that $(\I^1_o)_Y$ and $(\I^2_0)_Y$ must differ from a line bundle on $Y$ supported on $Y\cap Y^\mathsf{c}$. 

We conclude using that (see \cite[Lemma 12.6]{Oda1979CompactificationsOT} and \cite[Lemma 2.6]{viviani2023new}):
\begin{equation}\label{E:pull-I12}
\frac{\nu^*(\I^1_o)}{\Tors(\nu^*(\I^1_o))}(\sum_{\ee\in \O(\I^1)}n_{\ee}^t)=\L(\I^1)_o=\L(\I^2)_o= \frac{\nu^*(\I^2_o)}{\Tors(\nu^*(\I^2_o))}(\sum_{\ee\in \O(\I^2)}n_{\ee}^t)
\end{equation}    
\end{proof}

\subsubsection{Isotrivial specializations}\label{S:isospec}

In this subsection we describe a special class of specializations.

Consider the quotient stack $\Theta=\Theta_k:=[\AA_k^1/\Gm]$. This stack has two $k$-points: the open point $1:=[\AA^1\setminus\{0\}/\Gm]$ with trivial stabiizer and the closed point $0:=[0/\Gm]=B\Gm$ with stabilizer equal to $\Gm$. 
Given two sheaves $I,J\in \TF_X(k)$, we say that $J$ is an \emph{isotrivial (or very close) specialization} of $\I$ if there exists a morphism $f:\Theta\to \TF_X$ such that $f(1)=I$ and $f(0)=J$. The morphism $f:\Theta\to \TF_X$ is called an isotrivial (or very close) specialization from $f(1)$ to $f(0)$. 

In the next Proposition, we describe isotrivial specializations in $\TF_X$. We will need the following definition.
Given a sheaf $I\in \TF_X(k)$ and an ordered partition of $X$ by subcurves
    $$
    Y_{\bullet}:=(Y_0, \ldots, Y_q),
    $$
    i.e. a collection of subcurves covering $X$ and without common pairwise irreducible components,   we set 
\begin{equation}\label{E:GrY}
\Gr_{Y_\bullet}(I):=\bigoplus_{i=0}^q I_{Y_i}\left(-\NF(I)^\mathsf{c}\cap Y_i\cap (\bigcup_{0\leq j<i} Y_j)\right).
\end{equation}
The sheaf $\Gr_{Y_\bullet}(I)$ belongs to $\TF_X(k)$ and Formulas \eqref{E:deg-sub} and \eqref{E:deg-mdeg} imply that $\deg \Gr_{Y_\bullet}(I)=\deg I$.

\begin{proposition}\label{P:iso-spec}
    A sheaf $I\in \TF_X(k)$ isotrivially specializes to $J\in \TF_X(k)$ if and only if $J=\Gr_{Y_\bullet}(I)$ for some ordered partition $Y_\bullet$ of $X$.     
\end{proposition}
\begin{proof}
This follows from \cite[Prop. 3.10 and Example 3.4]{FPV1}. 
\end{proof}

The above result allows to express the change of combinatorial type in an isotrivial degeneration. For that purpose, given an ordered partition of $V(\Gamma_X)$
    $$
    W_{\bullet}:=(W_0, \ldots, W_q),
     $$
     consider the associated ordered partiton of $X$
    $$
    X[W_{\bullet}]:=(X[W_0], \ldots, X[W_q]).
    $$
For any spanning subgraph $G\leq \Gamma_X$, consider the partial orientation of $G$ 
\begin{equation}
    \O_G(W_\bullet):=\{\ee\in \EE(G): s(\ee)\in W_i  \text{ and } t(\ee)\in W_j \text{ for some } i<j\}.
    \end{equation}

\begin{corollary}\label{C:iso-spec}
   Consider a sheaf $I\in \TF_X(k)$ that isotrivially specializes to  $J=\Gr_{X[W_\bullet]}(I)\in \TF_X(k)$, for some ordered partition $W_\bullet$ of $V(\Gamma_X)$.    Then 
   $$(G(I),D(I))\geq_{\O_{G(I)}(W_\bullet)} (G(J),D(J)).$$ 
\end{corollary}

\subsection{Compactified Jacobians}\label{Sub:cJ}

We are now ready for the main definitions of this paper. Let $X$ be a connected nodal curve over $k=\ov k$. 

\begin{definition}(See \cite[Def. 6.1]{FPV1}) \label{D:compJac}
 A \emph{compactified Jacobian stack} of $X$ is an open substack $\ov \J_X$ of $\TF_X$ admitting a $k$-proper good moduli space $\ov J_X$ (called the associated \emph{compactified Jacobians space}).
 
 If $\ov \J_X$ is contained in $\TF_X^d$ for some $d$ (which is always the case for connected compactified Jacobian stacks), then we say that $\ov \J_X$ is a \emph{degree-$d$} compactified Jacobian stack and we denote it by $\ov \J_X^d$.  

 If $\ov \J_X$ is contained in $\Simp_X$ (or, equivalently, if $\ov J_X=\ov \J_X\fatslash \Gm$) then $\ov \J_X$ is called \emph{fine}.
\end{definition}
The above definition was introduced for fine compactified Jacobian in \cite[Def. 3.1]{pagani2023stability} (see also \cite[Def. 2.22]{viviani2023new}) with the additional assumption that they should also be \emph{connected}.

We are interested in compactified Jacobians that are limits of Jacobians of smooth curves in the following sense. 

 \begin{definition}\label{D:compJac-spec}
An open substack $\ov \J_X^d$ of $\TF_X^d$ is called a \emph{smoothable degree-$d$ compactified Jacobian stack}
    if, for any one-parameter smoothing $\X/\Delta$ of $X$, the open substack $\ov \J_{\X}^d:=\ov \J_X^d\cup \J_{\X_{\eta}}^d\subset \TF_{\X/\Delta}^d$ admits a good moduli space $\ov J_{\X}^d$ that is proper over $\Delta$.
\end{definition}

\begin{remark}\label{R:cJ-compare}
Observe that a  smoothable  degree-$d$ compactified Jacobian is a \emph{connected} degree-$d$ compactified Jacobian.  

    \hspace{0.2cm} Indeed, if $\ov \J^d_{\X}$ admits a good moduli space $\ov J^d_{\X}$ which is proper over $\Delta$, then, by the functoriality of good moduli spaces, the proper algebraic space $\ov J^d_X:=(\ov J^d_{\X})_o$ is a good moduli space for $\ov \J^d_{X}$. Moreover, since $\ov J^d_{\X}$ is proper over $\Delta$ and its generic fiber is the degree-$d$ Jacobian of $\X_{\eta}$ which is geometrically connected, it follows that $\ov J^d_X$ is connected, which then implies that $\ov \J_X^d$ is connected.
\end{remark}

The following result will imply that any compactified Jacobian stack is a (finite) union of orbits.

\begin{proposition}\label{P:unionorb}
 Let $\ov \J_X$ be an open substack of $\TF_X$ that satisfies the existence part of the valuative criterion.  Then  $\ov \J_X$ is a union of $\PIC_X^{\un 0}$-orbits.   
\end{proposition}
Note that a universally closed open substack $\ov \J_X\subset \TF_X$ satisfies the existence part of the valuative criterion by \cite[\href{https://stacks.math.columbia.edu/tag/0CLX}{Lemma 0CLX}]{stacks-project}, using that any such   stack $\ov\J_X$ is quasi-separated by Fact~\ref{F:diagTF}\eqref{F:diagTF1}. The above result was proved for fine compactified Jacobians in \cite[Lemma~7.2]{pagani2023stability}.

\begin{proof}
    We have to show that if there is an orbit $\TF_X(G,D)$ of $\TF_X$ that has nonempty intersection with $\ov \J_X$, then $\TF_X(G,D)$ is entirely contained in $\ov \J_X$.
    
    Fix a sheaf $I\in \TF(G,D)$ and let us show that $I\in \ov \J_X$.
    Since $\ov \J_X$ is open in $\TF_X$ and $\TF_X(G,D)$ is irreducible, then the intersection $\ov \J_X\cap \TF(G,D)$ is an open and dense substack of $\TF(G,D)$.
    Hence, there exists a discrete valuation $k$-ring $R$ with residue field $k$ and a relative torsion rank-$1$ sheaf $\I\in \TF_X(B:=\Spec R)$ whose special fiber $\I_o$ is $I$ and whose general fiber $\I_{\eta}$ belongs to $\ov \J_X\cap \TF(G,D)$.
    
    Since $\ov \J_X$  satisfies the existence part of the valuative criterion by assumption, we can find, up to replacing $R$ with a finite extension,  another relative torsion rank-$1$ sheaf $\wt \I\in \TF_X(B)$ whose general fiber $\wt \I_{\eta}$ coincides with $\I_{\eta}$ and whose special fiber $\wt I:=\wt \I_o$ belongs to $\ov \J_X$.
    
    By Proposition \ref{P:non-sepa}\eqref{P:non-sepa1}, there exists $g\in C^0(\Gamma_X,\ZZ)$ such that $\wt \I=g(\I).$ Set $q+1:=|\Im(g)|$ and let $\Im(g)=\{m_0<m_1<\ldots<m_q\}$. For every $0\leq i\leq q$, define the subcurve 
    $$Y_i:=\bigcup_{v\in V(\Gamma_X): g(v)=m_i}  X_v,$$
    and consider the ordered partition $Y_{\bullet}:=(Y_0,\ldots, Y_q)$. 

We now make the following
    
   \un{Claim:} $\Gr_{Y_{\bullet}}(I)=\wt I$. 

   \vspace{0.1cm}
   
   The Claim concludes the proof since it implies that  $I$ isotrivially specializes to $\wt I$ by Proposition \ref{P:iso-spec}, which in turn implies that $I\in \ov \J_X$ since $\wt I\in \ov \J_X$. 

     \vspace{0.1cm}
     
    It remains to prove the Claim. 
    We first  compare the $1$-cochains associated to $\I$ and $\wt I$. 
    Since $\I_{\eta}$ and $\I_o$ belongs to the same orbit of $\TF_X$, namely $\TF(G,D)$, Lemma \ref{L:Io-Ieta} implies that $\gamma(\I)=0$. Hence, Proposition \ref{P:non-sepa}\eqref{P:non-sepa2} gives that $\gamma(\wt \I)=\delta_G(g)$. Therefore,
    the support of $\gamma(\wt \I)$ is equal to 
    $$\supp \gamma(\wt \I)=\bigcup_{0\leq i<j\leq q} E_G(g^{-1}(m_i),g^{-1}(m_j)).$$
    By applying Lemma \ref{L:Io-Ieta}, we deduce that  $\wt I=\wt \I_o$ is non-free at all the nodes between $Y_i$ and $Y_j$, for any $0\leq i<j\leq q$. This, together with Fact \ref{F:orbits}\eqref{F:orbits3}, gives the decomposition
    \begin{equation}\label{E:dec-wtI}
    \wt I=\bigoplus_{i=0}^q \wt I_{Y_i}.    
    \end{equation}

By comparing the above decomposition \eqref{E:dec-wtI} with the definition \eqref{E:GrY} of $\Gr_{Y_\bullet}(I)$, the Claim becomes equivalent to the equality 
(for any $0\leq i \leq q$):
\begin{equation}\label{E:IYi}
    \wt I_{Y_i}=I_{Y_i}\left(-\NF(I)^\mathsf{c}\cap Y_i\cap (\bigcup_{0\leq j<i} Y_j)\right).
\end{equation}

This follows from Proposition \ref{P:2limits} using that $\O(\I)=\emptyset$ (since $\gamma(\I)=0$) and that 
$$
\O(\wt \I)=\{\ee\in \EE(G): \: s(\ee) \in V(\Gamma_{Y_i}), t(\ee)\in V(\Gamma_{Y_j})\: \text{ for some } i<j\},
$$
which follows from $\gamma(\wt \I)=\delta_G(g)$ and the definition of $Y_\bullet$.
\end{proof}

The above Proposition allows us to define the poset of orbits of a compactified Jacobian stack.

\begin{definition}\label{D:orb-fcJ}
Let $\ov \J_X$ be a open substack of $\TF_X$ that is universally closed (e.g. a compactified Jacobian stack of $X$). The \emph{poset of orbits} of $\ov \J_X$ is the following upper subset 
$$\P(\ov J_X):=\{(G,D)\in \OO(\Gamma_X)\: : \TF(G,D)\subset \ov \J_X^d\}\subset \OO(\Gamma_X).$$ 
\end{definition}
The poset of orbits recovers the open substack $\ov \J_X$, since, by Proposition~\ref{P:unionorb}, we have that 
(with the notation of \eqref{E:opTF}
\begin{equation}\label{E:U-cJ}
\ov \J_X=U(\P(\ov J_X)).
\end{equation}

The smooth locus and the irreducible components of a compactified Jacobian stack are described in the following

\begin{corollary}\label{C:sm-irr}
Let $\ov \J_X$ be a open substack of $\TF_X$ that is universally closed (e.g. a compactified Jacobian stack of $X$).
\begin{enumerate}
\item \label{C:sm-irr1}  $\ov \J_X$ is a reduced stack  of pure dimension $g(X)-1$  with locally complete intersection singularities and its smooth locus is equal to 
$$
(\ov \J_X)_{\sm}=\coprod_{D\in \P(\ov J_X)(\Gamma_X)} \TF(\Gamma_X,D).
$$
In particular, the smooth locus of $\ov \J_X$ is isomorphic to a disjoint union of  copies of the generalized Jacobian $\PIC^{\un 0}(X)$
\item \label{C:sm-irr2} The irreducible components of $\ov \J_X$ are given by 
$$
\Big\{\ov{\TF(\Gamma_X,D)}\: : D\in \P(\ov \J_X)(\Gamma_X)\Big\}.
$$
\end{enumerate}
\end{corollary}
\begin{proof}
Part \eqref{C:sm-irr1} follows from Proposition \ref{P:unionorb} and Fact \ref{F:diagTF}\eqref{F:diagTF3}. Part \eqref{C:sm-irr2} follows from Proposition \ref{P:unionorb} and Fact \ref{F:orbits}.
\end{proof}

We now investigate when an open substack $\ov \J_X$ of $\TF_X$ admits a good moduli space. 

\vspace{0.1cm}

Recall that $\ov \J_X$ is $\Theta$-complete if and only if, for any DVR $R$ with residue field $k$, any map $f: \Theta_R\setminus \{0\}\to \ov \J_X$  can be extended to a map $F: \Theta_R\to \ov \J_X$, where $\Theta_R:=[\AA^1_R/\Gm]$ and $0=[\Spec k/\Gm]$ is the closed point of $\Theta_R$. The four points of $\Theta_R$ are related by the following specializations. 
 
 \begin{figure}[hbt!]
 \[
\begin{tikzcd}
    1_{\eta}:=[\AA^1_\eta\setminus \{0\}/\Gm] \arrow[r] \arrow[d, rightsquigarrow] &  1:=[\AA^1_k\setminus \{0\}/\Gm] \arrow[d, rightsquigarrow] \\
    0_\eta:=[\eta/\Gm] \arrow[r] & 0:=[\Spec k/\Gm]
\end{tikzcd}
\]
\caption{The four points of $\Theta_R$ and their specializations: the horizontal arrows are ordinary specializations while the vertical arrows are isotrivial specializations.\label{F:ThetaR}}
\end{figure}

\begin{proposition}\label{P:gms}
 Let $\ov \J_X^d$ be an open substack of $\TF_X^d$. The following conditions are equivalent:
 \begin{enumerate}
     \item \label{P:gms1} for any one-parameter smoothing $\X/\Delta$ of $X$, the stack $\ov \J_{\X}^d:=\ov \J_X^d\cup \J_{\X_{\eta}}^d\subset \TF_{\X/\Delta}^d$ admits a good moduli space $\ov J_{\X}^d$ over $\Delta$.
    \item \label{P:gms2} there exists a regular one-parameter smoothing $\X/\Delta$ of $X$, such that the stack $\ov \J_{\X}^d:=\ov \J_X^d\cup \J_{\X_{\eta}}^d\subset \TF_{\X/\Delta}^d$ admits a good moduli space $\ov J_{\X}^d$ over $\Delta$.
   \item \label{P:gms3} $\ov \J_X^d$ admits a good moduli space $\ov J_X^d$.
     \item \label{P:gms4} $\ov \J_X^d$ is $\Theta$-complete.  
 \end{enumerate}
\end{proposition}
\begin{proof}
   We will prove a cyclic chain of implications.

$\eqref{P:gms1}\Rightarrow \eqref{P:gms2}$ is obvious. 

$\eqref{P:gms2}\Rightarrow \eqref{P:gms3}$ follows from the fact that good moduli spaces satisfy base change.

$\eqref{P:gms3}\Rightarrow \eqref{P:gms4}$ follows from the necessary conditions for the existence of good moduli spaces in \cite[Thm. 4.1]{alper-goodmodspaces}.

$\eqref{P:gms4}\Rightarrow \eqref{P:gms1}$: we apply the sufficient conditions for the existence of good moduli spaces in \cite[Thm. 4.1]{alper-goodmodspaces}. Note that the stack $\ov\J_{\X}^d$ is locally of finite type and it has affine diagonal over $\Delta$ by \cite[\href{https://stacks.math.columbia.edu/tag/0DLY}{Lemma 0DLY}]{stacks-project} and \cite[\href{https://stacks.math.columbia.edu/tag/0DLZ}{Lemma 0DLZ}]{stacks-project}. Moreover, the stabilizers of all points of $\ov\J_{\X}^d$ are isomorphic to $\mathbb{G}_m^r$ for some $r\geq 0$ (see \eqref{E:Aut-I}) and hence the stack  $\ov\J_{\X}^d$ satisfies the conditions (1) and (3) of loc. cit. (see also \cite[Prop. 3.55]{alper-goodmodspaces}). Hence it remains to show that $\ov\J_{\X}^d$ satisfies condition (2) of loc. cit., i.e. that $\ov J_{\X}^d\to \Delta$ is $\Theta$-complete, if and only if its central fiber $\ov J_{X}^d$ is $\Theta$-complete. This follows from the fact that the generic fiber $\ov J_{\X_{\eta}}^d$ of $\ov J_{\X}^d$ is a scheme, and hence any morphism $\Theta_\eta\to \ov J_{\X_{\eta}}^d$ is constant. 
\end{proof}

We now give a necessary combinatorial criterion for the $\Theta$-completeness of an open substack $\ov \J_X$ of $\TF_X$ that is union of orbits.

\begin{proposition}\label{P:Theta-red}   
Let  $\ov \J_X$ be an open substack of $\TF_X$ that is union of orbits. 
If $\ov \J_X$ is $\Theta$-complete, then,  for any three orbits 
$\TF_X(G,D)$, $\TF_X(G',D')$ and $\TF_X(G'',D'')$ contained in $\ov \J_X$ such that 
\begin{equation}\label{E:3orient}
    \begin{sis}
    & (G,D)\geq_{\O_G(W_\bullet)} (G',D') \text{ for some ordered partition } W_\bullet \text{ of } V(\Gamma_X);\\
    & (G,D)\geq_{\O} (G'',D'') \text{ for some partial orientation $\O$ of $G$},  
    \end{sis}
    \end{equation}
it holds that:
\begin{itemize}
    \item $\O_G(W_\bullet)$ and $\O$ are  concordant (i.e. any edge of $G$ which is both in the support of $\O_G(W_\bullet)$ and of $\O$ is oriented in the same way in $\O$ and in $\O_G(W_\bullet)$), or equivalently $\O\cup \O_G(W_\bullet)$ is a partial orientation of $G$; 
    \item if we set $(G,D)\geq_{\O\cup \O_G(W_\bullet)} (\wt G, \wt D)$, then the orbit $\TF_X(\wt G,\wt D)$ is contained in $\ov \J_X$. 
    \end{itemize}
\end{proposition}

\begin{proof}
Suppose first that the combinatorial condition in the statement does not hold and let us prove that $\ov \J_X$ is not $\Theta$-complete. By assumption, there exists  three orbits $\TF_X(G,D), \TF_X(G',D'), \TF_X(G'',D'')\subset \ov \J_X$ satisfying \eqref{E:3orient} and such that either $\O_G(W_\bullet)$ and $\O$ are not concordant or, they are concordant but if we set $(G,D)\geq_{\O\cup \O_G(W_\bullet)} (\wt G, \wt D)$  then $\TF_X(\wt G,\wt D)$ is not contained in $\ov \J_X$. 

By Fact \ref{F:orbits}\eqref{F:orbits4}, there exists a $k$-DVR $R$ with residue field $k$ and a relative rank-$1$ torsion-free sheaf $\I\in \ov \J_X(\Spec R)$ such that  its generic fiber $\I_\eta\in \TF_X(\eta)$ has combinatorial type $(G,D)$ and its special fiber $\I_o\in \TF_X(k)$ has combinatorial type $(G'',D'')$. Moreover, by Proposition \ref{P:iso-spec}, there exists an isotrivial specialization from $\I_\eta$ to $\Gr_{X[W_\bullet]}(\I_\eta)$, or in other words  a map $f_\eta:[\AA^1_\eta/\Gm]\to \ov \J_X$ such that $f_\eta(1_\eta)=\I_\eta$ and $f_\eta(0_\eta)=\Gr_{Y_\bullet}(\I_\eta)$. The map $f_\eta$ and the map $f_1:\Spec R\to \ov \J_X$ associated to $\I$ glue together to give rise to a map $f:\Theta_R\setminus \{0\}\to \ov \J_X$ such that $f_{|[\AA^1_R\setminus\{0\}/\Gm]}=f_1$ and $f_{[\AA^1_\eta/\Gm]}=f_1$. 

Now we claim that the above map $f:\Theta_R\setminus \{0\}\to \ov \J_X$ does not extend to $\Theta_R$. Indeed, by contradiction, assume that there exists an extension $F:\Theta_R\to \ov \J_X$ of $f$. Then the combinatorial type $(\wh G,\wh D)$ of the sheaf $F(0)\in \ov \J_X$ is such that 
$$
\begin{sis}
    & (G,D)\geq_{\O_G(W_\bullet)} (G',D')\geq (\wh G,\wh D),\\
    & (G,D)\geq_{\O} (G'',D'') \geq (\wh G,\wh D).
\end{sis}
$$
Therefore, Fact \ref{F:orbits}\eqref{F:orbits4} implies that there exists a partial orientation $\wh O$ of $G$ such that $(G,D)\geq_{\wh O} (\wh G,\wh D)$ and such that $\wh O$ restricts to both $\O$ and $\O_G(W_\bullet)$. In particular, $\O$ and $\O_G(W_\bullet)$ are concordand and we must have that $(\wh G,\wh D)\leq (\wt G,\wt D)$. 
Since $F(0)\in \ov \J_X\cap \TF_X(\wh G,\wh F)$ and $\ov \J_X$ is a union of orbits, we must have that $\TF_X(\wh G,\wh D)\subset \ov \J_X$. Finally, since $\ov \J_X$ is open, we have also the inclusion $\TF_X(\wt G,\wt D)\subset \ov \J_X$ and this is absurd.  
\end{proof}

\subsection{V-compactified Jacobians}\label{Sub:VcJ}

In this subsection, we recall the definition and main properties of the V-compactified Jacobians, introduced in \cite{viviani2023new} and \cite{FPV1}. 

Let us first recall the definition of V-stability conditions for a connected nodal curve.  

\begin{definition}\label{D:VStabX}
Let $X$ be a connected nodal curve over $k=\ov k$. 
   A \emph{stability condition of vine type} (or simply a \textbf{V-stability condition)}    of characteristic $\chi\in \ZZ$ on  $X$ is a function
    \begin{align*}
        \s:\BCon(X)&\to \ZZ\\
        Y&\mapsto \s_Y
    \end{align*}
    satisfying the following properties:
\begin{enumerate}
\item \label{D:VStabX1} for any $Y\in \BCon(X)$, we have 
\begin{equation}\label{E:sum-s}
\mathfrak s_Y+\mathfrak s_{Y^\mathsf{c}}-\chi
\in \{0,1\}.
\end{equation}
A subcurve $Y\in \BCon(X)$ is said to be \emph{$\s$-degenerate} if $\s_Y+\s_{Y^\mathsf{c}}-\chi=0$,
and \emph{$\s$-nondegenerate} otherwise.

\item  \label{D:VStabX2} given subcurves $Y_1,Y_2,Y_3\in \BCon(X)$ without pairwise common irreducible components such that $X=Y_1\cup Y_2\cup Y_3$, we have that:
\begin{enumerate}
 \item if two among the subcurves $\{Y_1,Y_2,Y_3\}$ are $\s$-degenerate, then so is  the third.
            \item the following condition holds
            \begin{equation} \label{E:tria-s}
            \sum_{i=1}^{3}\s_{Y_i}-\chi
            \in \begin{cases}
                \{1,2\} \textup{ if $Y_i$ is $\s$-nondegenerate for all $i=1,2,3$};\\
                \{1\} \textup{ if there exists a unique   } i\in \{1,2,3\} \text{ such that $Y_i$ is $\s$-degenerate};\\
                \{0\} \textup{ if $Y_i$ is $\s$-degenerate for all $i=1,2,3$}.
            \end{cases}
        \end{equation}
\end{enumerate}
\end{enumerate}

The characteristic $\chi$ of $\s$ will also be denoted by $|\s|$. The \emph{degeneracy set} of $\s$ is the collection
\begin{equation*}
\D(\s):=\{Y\in \BCon(X): Y \text{ is $\s$-degenerate}\}.
\end{equation*}
    A V-stability condition $\s$ is called \emph{general}  if every $Y\in \BCon(X)$ is $\s$-nondegenerate, i.e. if $\D(\s)=\emptyset$.

The collection of all V-stability conditions of characteristic $\chi$ on $X$ is denoted by $\VStab^\chi(X)$ and the collection of all V-stability conditions on $X$ is denoted by 
$$\VStab(X)=\coprod_{\chi \in \ZZ}\VStab^\chi(X).$$
\end{definition}

The easiest way of producing V-stability conditions is via numerical polarizations, as we now explain following \cite[Subsec. 4.1]{FPV1}. 

\begin{example}(\textbf{Classical V-stability conditions})
Let $X$ be a connected nodal curve. A \emph{numerical polarization} on $X$ of characteristic $\chi\in \ZZ$ is a function  
$$
\begin{aligned}
\psi:\left\{\text{Subcurves of } X\right\} & \longrightarrow \RR\\
Y & \mapsto \psi_Y
\end{aligned}
$$
that is \emph{additive}, i.e. if $Y_1,Y_2$ are subcurves of $X$ such that $Y_1\wedge Y_2=\emptyset$ then $\psi_{Y_1\cup Y_2}=\psi_{Y_1}+\psi_{Y_2}$, and such  that $|\psi|:=\psi_X=\chi$. The function 
\begin{equation}\label{E:s-psi}
\begin{aligned}
        \s(\psi) :\BCon(X)&\to \ZZ\\
        Y&\mapsto \s(\psi)_Y:=\left\lceil\psi_Y\right\rceil.
    \end{aligned}
    \end{equation}
is a V-stability condition on $X$ of characteristic $\chi$ (called the V-stability condition associated to $\psi$), such that 
$$
\D(\s(\psi))=\{Y\in \BCon(X): \: \psi_Y\in \ZZ\}.
$$
 This follows  by taking the upper integral parts in the following two equalities
$$
\begin{sis}
& \psi_Y+\psi_{Y^\mathsf{c}}-\chi=0 \text{ for any } Y\in \BCon(X),\\    
& \sum_{i=1}^3 \psi_{Y_i}-\chi=0 \text{ for any } \{Y_1,Y_2,Y_3\} \text{ as in  Definition \ref{D:VStabX}\eqref{D:VStabX2}}. 
\end{sis}
$$
The V-stabilities of the form $\s(\psi)$ are called \emph{classical}.

Because of the additivity property, a numerical polarization is completely determined by its values  on the irreducible components of $X$. Hence, the space of numerical polarizations on $X$ of characteristic $\chi\in \ZZ$, denoted by $\Pol^{\chi}(X)$, is a real affine subspace of $\RR^{I(X)}$ (with $I(X)$ the set of irreducible components of $X$) whose underlying real vector space is $\Pol^0(X)$. 

Consider the arrangement of hyperplanes in $\Pol^{\chi}(X)$ given by  
\begin{equation}\label{E:arr-hyperX}
\A_{X}^\chi:=\left\{\psi_Y=n\right\}_{Y\in \BCon(X), n\in \ZZ}.
\end{equation}
This arrangement yields an induced wall and chamber decomposition on $\Pol^\chi(X)$ such that two numerical polarizations $\psi, \psi'$ belong to the same region, i.e. they have the same relative positions with respect to all the hyperplanes, if and only if $\s(\psi)=\s(\psi')$. In other words, the set of regions induced by $\A_X^\chi$ on $\Pol^\chi(X)$ is the set of classical V-stability conditions on $X$.  Note also that $\s(\psi)$ is a general V-stability condition if and only if $\psi$ belongs to a chamber (i.e. a maximal dimensional region), or equivalently if it does not lie on any of the hyperplanes of $\A_X^\chi$, in which case we say that $\psi$ is a \emph{general} numerical polarization.
\end{example}

We now show that V-stability conditions on a nodal curve $X$ and on its dual graph $\Gamma_X$ are in canonical bijection, and that many of the construction introduced in \cite{FPV1} for V-stability conditions on curves correspond to V-stability conditions on graphs introduced in Section \ref{Sec:Vstab}.

\begin{proposition}\label{P:n-s}
Let $X$ be a connected nodal curve over $k=\ov k$ and let $\Gamma_X$ be its dual graph. 
For any $\chi\in \Z$, we have a bijection 
\begin{equation}\label{E:n-s}
\begin{aligned}
\VStab^{\chi}(X)& \xrightarrow{\cong} \VStab^{\chi+g(X)-1}(\Gamma_X)\\
\s& \mapsto \n \: \text{ such that } \n_{V(\Gamma_Y)}=\s_Y+g(Y)-1 \text{ for any } Y\in \BCon(X).
\end{aligned}
\end{equation}
If $\s\in \VStab^{\chi}(X)$ and $\n\in \VStab^{\chi+g(X)-1}(\Gamma_X)$ are in correspondence under the above bijection, we have that
\begin{enumerate}[(i)]
    \item \label{P:n-s1} $Y\in \D(\s)$ if and only if $V(\Gamma_Y)\in \D(\n)$. In particular, $\s$ is general if and only if $\n$ is general.
    \item \label{P:n-s2} $Y$ belongs to the extended degeneracy subset $\wh \D(\s)$ of $\s$ (see \cite[Def. 4.3]{FPV1}) if and only if $V(\Gamma_Y)\in \wh \D(\n)$.
    \item \label{P:n-s3} The extended V-function associated to $\s$ (see \cite[Def. 4.4]{FPV1}) and the extended V-function associated to $\s$ are related by the same formula \eqref{E:n-s}.  
    \item \label{P:n-s4} The bijection \eqref{E:n-s} is an isomorphism of posets (where the poset structure on $\VStab^\chi(X)$ is defined in \cite[Def. 4.1]{FPV1}).
    \item \label{P:n-s5} A sheaf $I\in \TF_X(k)$ is $\s$-semistable (resp. $\s$-polystable, resp. $\s$-stable) in the sense of \cite[Def. 5.1]{FPV1} if and only if $(G(I),D(I))$ is $\n$-semistable (resp. $\n$-polystable, resp. $\n$-stable).
    \item \label{P:n-s6} There is an isomorphism of real affine spaces
    \begin{equation}\label{E:Pol-Div}
\begin{aligned}
\Pol^{\chi}(X)& \xrightarrow{\cong} \Div^{\chi+g(X)-1}(\Gamma_X)_{\R}\\
\psi& \mapsto \phi\: \text{ such that } \phi_{V(\Gamma_Y)}=\psi_Y+\frac{\deg_Y(\omega_X)}{2} \text{ for any } Y\in \BCon(X),
\end{aligned}
\end{equation}
under which $\s(\psi)=\n(\phi)$.
\end{enumerate}
\end{proposition}
\begin{proof}
Let us first show the bijection in \eqref{E:n-s}. Clearly, it is enough to show that if $\s=\{\s_Y: Y\in \BCon(X)\}$ and $\n=\{\n_W: W\in \BCon(\Gamma_X)\}$ are related by the formula in \eqref{E:n-s}, then $\s$ is a V-stability on $X$ of characteristic $\chi$ if and only if $\n$ is a V-stability on $\Gamma_X$ of degree $d:=\chi+g(X)-1$. 

Let $Y\in \BCon(X)$ and set $W:=V(\Gamma_Y)\in \BCon(\Gamma_X)$, so that $V(\Gamma_{Y^\mathsf{c}})=W^\mathsf{c}$. By applying \eqref{E:n-s} to $Y$ and $Y^\mathsf{c}$, and using the formula $g(X)=g(Y)+g(Y^\mathsf{c})+\val(W)-1$, we compute 
$$
\n_{W}+\n_{W^\mathsf{c}}+\val(W)-d=[\s_Y+g(Y)-1]+[\s_{Y^\mathsf{c}}+g(Y^\mathsf{c})-1]+\val(W)-[\chi+g(X)-1]=\s_Y+\s_{Y^\mathsf{c}}-\chi,
$$
which shows that Condition \eqref{D:VStabX1} of Definition \ref{D:VStabX} holds for $\s$ if and only if Condition \eqref{D:Vstab1} of Definition \ref{D:Vstab} holds for $\n$, and that $Y\in \D(\s)$ if and only if $W\in \D(\n)$.

Consider now three subcurves $Y_1,Y_2,Y_3\in \BCon(X)$ without pairwise common irreducible components such that $X=Y_1\cup Y_2\cup Y_3$, and set $W_i:=V(\Gamma_{Y_i})\in \BCon(\Gamma_X)$ for any $i=1,2,3$, so that $\{W_1,W_2,W_3\}$ form a pairwise disjoint partition of $V(\Gamma_X).$ By applying \eqref{E:n-s} to each $Y_i$, and using the formula $g(X)=\sum_{1\leq i\leq 3} g(Y_i)+\sum_{1\leq i<j\leq 3}\val(W_i,W_j)-2$, we compute 
$$\begin{aligned}
&\sum_{1\leq i\leq 3}\n_{W_i}+\sum_{1\leq i<j\leq 3} \val(W_i,W_j)-d =\\
&\sum_{1\leq i\leq 3}[\s_{Y_i}+g(Y_i)-1]+\sum_{1\leq i<j\leq 3} \val(W_i,W_j)-[\chi+g(X)-1] =\sum_{1\leq i\leq 3}\s_{Y_i}-\chi,
\end{aligned}
$$
which shows that Condition \eqref{D:VStabX2} of Definition \ref{D:VStabX} holds for $\s$ if and only if Condition \eqref{D:Vstab2} of Definition \ref{D:Vstab} holds for $\n$,
This concludes the proof of the bijection and also the proof of part \eqref{P:n-s1}. 

Parts \eqref{P:n-s2}, \eqref{P:n-s3} and \eqref{P:n-s4} follow straightforwardly by comparing the definitions of the objects involved. 

Part \eqref{P:n-s5} for semistability and stability follows by comparing the definitions of $\s$-semistability (resp. $\s$-stability) from \cite[Def. 5.1]{FPV1} and  $\n$-semistability (resp. $\n$-stability) of Definition \ref{D:Pn}, and using that for any $Y\in \BCon(X)$ with $W:=V(\Gamma_Y)\in \BCon(\Gamma_X)$ we have
$$
\begin{aligned}
& \chi(I_Y)-\s_Y=\deg(I_Y)+1-g(Y)-\s_Y= & \text{ by } \eqref{E:deg-I}\\   & =\deg_{X[W]}(I)-\n_W= & \text{ by } \eqref{E:n-s} \\
&=D(I)_W+e_{\NF(I)}(W)-\n_W & \text{ by } \eqref{E:deg-sub}.
\end{aligned}
$$
The equivalence of $\s$-polystability and $\n$-polystability follows from the above relation together with the fact that $E(W,W^\mathsf{c})\subseteq \NF(I)$ if and only if $I=I_{X[W]}\oplus I_{X[W^\mathsf{c}]}$.

Part \eqref{P:n-s6}: the fact that \eqref{E:Pol-Div} is an isomorphism of real affine spaces follows from the fact that the function 
$$
\begin{aligned}
 \left\{\text{Subcurves of } X\right\} & \longrightarrow \R\\
 Y & \mapsto \frac{\deg_Y(\omega_X)}{2}
\end{aligned}
$$
is an additive function such that $\displaystyle \frac{\deg_X(\omega_X)}{2}=g(X)-1$. The last assertion follows by taking the upper integral parts of the following equality
$$
\phi_{V(\Gamma_Y)}-\frac{\val(V(\Gamma_Y))}{2}=\psi_Y+\frac{\deg_Y(\omega_X)}{2}-\frac{\val(V(\Gamma_Y))}{2}=\psi_Y+g(Y)-1.
$$
\end{proof}

To any V-stability condition on a nodal curve (or on its dual graph), we can associate a compactified Jacobian, as we now recall.

\begin{definition}\label{D:VcJ}
 Let $\s$ be a V-stability condition  on $X$. The \emph{V-subset} associated to $\s$ is 
 $$
 \ov \J_X(\s):=\{I\in \TF_X: \: \chi(I)=|\s|, \chi(I_Y)\geq \s_Y \text{ for any } Y\in \BCon(X)\}\subset \TF_X^{|\s|+g(X)-1}. $$
\end{definition}

\begin{remark}\label{R:Pn-cJ}
 Let $\s$ be a V-stability condition  on $X$ and let $\n$ be the corresponding V-stability condition on $\Gamma_X$ as in Proposition \ref{P:n-s}. Then, by comparing \eqref{E:n-s} and \eqref{E:deg-I}, we have that 
 $$ \ov \J_X(\s)=\ov \J_X(\n):=\{I\in \TF_X: \: \deg(I)=|\n|, \deg(I_{X[W]})\geq \n_W \text{ for any } W\in \BCon(\Gamma_X)\}\subset \TF_X^{|\n|}.$$
 
 It follows from Proposition \ref{P:n-s}\eqref{P:n-s5}  \eqref{E:deg-sub} that $I\in \ov \J_X^d$ if and only if $(G(I),D(I))\in \P_\n$. Therefore, since $\P_\n$ is a finite upper subset of $\OO^d(\Gamma_X)$ by Proposition \ref{P:n-ss}\eqref{P:n-ss1}, we have (using \eqref{E:opTF}):
 $$
 \ov \J_X(\n)=U(\P_\n)\subset \TF_X^d \text{ is open and a finite union of orbits.}
 $$
 In particular, $\ov \J_X(\n)$ is of finite type. 
\end{remark}

\begin{theorem}\label{T:VcJ-smoo}
    For any V-stability $\s$ on $X$, the associated V-subset $\ov \J_X(\s)$ is a smoothable compactified Jacobian of $X$ of degree $|\s|+g(X)-1$. Moreover:
    \begin{enumerate}[(i)]
    \item \label{T:VcJ-smoo1} $\ov \J_X(\s)$ is fine if and only if $\s$ is general.
    \item \label{T:VcJ-smoo2} $\s$ is uniquely determined by $\ov \J_X(\s)$.
    \end{enumerate}
\end{theorem}
In particular, $\ov \J_X(\s)$ is a connected compactified Jacobian of degree $|\s|+g(X)-1$ (see Remark \ref{R:cJ-compare}), which will be called the \emph{V-compactified Jacobian} associated to the V-stability condition $\s$.

\begin{proof}
Let $\X/\Delta$ be a one-parameter smoothing of $X$. Since the geometric generic fiber of $\X/\Delta$ is smooth, the V-stability condition $\s$ on $X$ extends uniquely to a V-stability condition on the family $\X/\Delta$ (in the sense of \cite[Subsection 4.3]{FPV1}), which we will continue to denote by $\s$. Moreover, we have that 
$$
\ov \J_{\X/\Delta}(\s)=\ov \J_{\X}^d=\ov \J_X(\s) \cup \J_{\X_{\eta}}^{d}\subset \TF_{\X/\Delta}^d,
$$
where $d:=|\s|+g(X)-1$. It follows from \cite[Thm. 6.6]{FPV1} that $\ov \J_{\X/\Delta}(\s)\to \Delta$ admits a relative proper good moduli space, which shows that $\ov \J_X(\s)$ is a smoothable compactified Jacobians stack of $X$. 
Finally: Part~\eqref{T:VcJ-smoo1} follows from the fact that $\s$-stability coincides with $\s$-semistability if and only if $\s$ is general (see \cite[Sec. 5]{FPV1}); Part~\eqref{T:VcJ-smoo2} follows from \cite[Cor. 8.11]{FPV1}. 
\end{proof}

\subsection{Smallest compactified Jacobians}\label{Sub:smallcJ}

The aim of this subsection is to characterize V-compactified Jacobians as the ones having the smallest number of irreducible components with a given degeneracy subset.

We begin by defining the degeneracy (resp. extended degeneracy) subset of any open subset of $\TF_X^d$.

\begin{definition}\label{D:deg-cJ}
 Let $\ov \J_X^d$ be an open substack of $\TF_X^d$.
 \begin{enumerate}
     \item \label{D:deg-cJ1} 
     The \emph{degeneracy subset of} $\ov \J_X^d$ is
     $$
     \D(\ov \J_X^d):=\{W\in \BCon(\Gamma_X)\: : \text{ there exists a sheaf } I=I_{X[W]}\oplus I_{X[W^\mathsf{c}]}\in \ov \J_X^d\}.$$
     \item \label{D:deg-cJ2}
     The \emph{extended degeneracy subset of} $\ov \J_X^d$ is
     $$
     \wh \D(\ov \J_X^d):=\{W\in \Con(\Gamma_X)\: : \text{ there exists a sheaf } I=I_{X[W]}\oplus I_{X[W^\mathsf{c}]}\in \ov \J_X^d\}.$$
 \end{enumerate}
\end{definition}

\begin{remark}\label{R:obs-deg-cJ}
    Let $\ov \J_X^d$ be an open substack of $\TF_X^d$. 
     \begin{enumerate}
         \item \label{R:obs-deg-cJ1} For any $W\in \BCon(\Gamma_X)$, we have that 
         $$W\in \D(\ov \J_X^d) \Leftrightarrow W^\mathsf{c}\in \D(\ov \J_X^d).$$ 
         \item \label{R:obs-deg-cJ2}
         Let $W\in \Con(\Gamma_X)$ and let $\Gamma_X[W^\mathsf{c}]=\coprod \Gamma_X[V_i]$ be the decomposition into connected component. Then 
         $$W\in \wh \D(\ov \J_X^d) \Rightarrow V_i\in \D(\ov \J_X^d) \text{ for any } i.$$
     \end{enumerate}

Indeed, Part~\eqref{R:obs-deg-cJ1} is obvious from the definition of $\D(\ov \J_X^d)$, while Part~\eqref{R:obs-deg-cJ2} follows from the definition of $\wh \D(\ov \J_X^d)$  together with the fact that for any sheaf $I\in \TF_X^d$ we must have that 
     $$
     I_{X[W^\mathsf{c}]}=\bigoplus_i I_{X[V_i]}. 
     $$
\end{remark}

\begin{remark}\label{R:deg-VcJ}
    Let $\n$ be a V-stability of degree $d$ on $\Gamma_X$. By combining Remark~\ref{R:Pn-cJ} and Proposition~\ref{P:n-ss}\eqref{P:n-ss2} with Theorem~\ref{T:V=PT}, we deduce that 
    $$\D(\ov \J_X(\n))= \D(\n) \text{ and } \wh \D(\ov \J_X(\n))= \wh \D(\n).$$
\end{remark}

\begin{proposition}\label{P:many-orbits}
    Let $\ov \J_X^d$ be a non-empty open substack of $\TF_X^d$ that is universally closed, of finite type, connected and $\Theta$-complete (e.g. a connected compactified Jacobian stack of $X$ of degree $d$). If $\Gamma_X\setminus S$ is a spanning subgraph of $\Gamma_X$ that is $\D(\ov \J_X^d)$-admissible, i.e. such that the decomposition of $\Gamma_X\setminus S=\coprod G_i$ into connected components is such that $V(G_i)\in \wh \D(\n)$, then we have that
    \begin{equation}\label{E:many-orbits}
    \TF_X(\Gamma_X\setminus S,D)\subseteq \ov \J_X^d \: \text{ for some }
    D\in \Div^{d-|S|}(\Gamma_X)
    \end{equation}
\end{proposition}
This result was proved in the case where $\ov \J_X^d$ is a non-empty open substack of $\Simp_X^d$ that is of finite type and universally closed in \cite[Thm. 2.16]{viviani2023new}\footnote{Note that \cite[Thm. 2.16]{viviani2023new} was stated for fine compactified Jacobians, but the same proof works for any non-empty open substack of $\TF_X^d$ that is of finite type and universally closed.}.
The four assumptions on $\ov \J_X^d$ are all needed: if one of them is missing, we can construct examples that do not satisfy the conclusion of the above Proposition.
\begin{proof} 
During the proof, we are going to freely use that, since $\ov \J_X^d$ is universally closed, it is a union of $\PIC_X^{\underline{0}}$-orbits  and that it satisfies the existence part of the valuative criterion (see  Proposition \ref{P:unionorb} and the comment shortly below it). 

We set $\Gamma:=\Gamma_X$. Note that $G_i$ is a connected spanning subgraph of $\Gamma[V(G_i)]$ and that $V(\Gamma_X)=\coprod_i V(G_i)$. Hence the proof  follows from the  two steps below.

\un{Step I:} If we have an inclusion
$$\TF_X\left(\coprod_i \Gamma[V_i], E\right)\subseteq \ov \J_X^d \text{ where } V(\Gamma_X)=\coprod_i V_i \text{ with } V_i \in \Con(\Gamma) \text{ and } D\in \Div(\Gamma)$$
then, for any connected spanning subgraph $H_i$ of $\Gamma[V_i]$, there exists  a divisor $D\in \Div(\Gamma)$ such that 
$$\TF_X\left(\coprod_ i H_i,D\right)\subseteq \ov \J_X^d.$$

\vspace{0.1cm}

Indeed, write $\coprod_i \Gamma[V_i]=\Gamma\setminus \wt S$ and $E=\sum_i E_i$ with $E_i\in \Div^{e_i}(\Gamma[V_i])$. The partial normalization of $X$ at $\wt S$ is  given by $X_{\wt S}=\coprod_i X_i$ with $X_i:=X[V_i]$ connected and it comes with a morphism 
 $$
 \nu_{\wt S}:X_{\wt S}=\coprod_i X_i \to X.
 $$
 The pushforward along the morphism $\nu_{\wt S}$ induces a closed embedding 
 $$
 \begin{aligned}
 (\nu_{\wt S})_*=\bigtimes_i \TF_{X_i}^{e_i}& \hookrightarrow \TF_X^d\\
 (I_i)_i & \mapsto \bigoplus_i I_i,
 \end{aligned}
 $$
 with the property that 
$$(\nu_{\wt S})_*(\bigtimes_i \TF_{X_i}(\Gamma[V_i]=\Gamma_{X_i},E_i))=\TF_X( \coprod_i \Gamma[V_i],E)\subset \ov \J_X^d.
$$
 Since $\ov \J_X^d$ is a open substack of $\TF_X^d$ which is a finite union of orbits, we get that the connected component of $(\nu_{\wt S})_*^{-1}(\ov \J_X^d)$ that contains $\bigtimes_i \TF_{X_i}(\Gamma_{X_i},E_i)$
is given by 
$$
\bigtimes_i \ov \J_{X_i}^{e_i}
$$
for some  non-empty open substacks $\ov \J_{X_i}^{e_i}\subseteq \TF_{X_i}^{e_i}$ which are  finite union of orbits (and in particular of finite type).  Moreover, since $\ov \J_X^d$ is universally closed  and $(\nu_{\wt S})_*$ is a closed embedding, we have that $\ov \J_{X_i}^{e_i}$ are also universally closed. Hence, the same proof of \cite[Thm. 2.16]{viviani2023new} gives the existence of a divisor $D_i\in \Div(\Gamma_{X_i})$ such 
$$
\TF_{X_i}(H_i,D_i)\subseteq \ov \J_{X_i}^{e_i}.$$
 From this, we deduce that 
$$
\TF_X(\coprod_i H_i,\sum_i D_i)=(\nu_{\wt S})_*(\bigtimes_i \TF_{X_i}(H_i,D_i)) \subseteq (\nu_{\wt S})_*(\bigtimes_i \ov \J_{X_i}^{s_i})\subseteq \ov \J_X^d,
$$
 and we are done. 

\vspace{0.1cm}

\un{Step II:} If $V(\Gamma)=\coprod_i V_i$ with $V_i\in \wh \D(\ov \J_X^d)$ then 
$$
\TF_X\left(\coprod_i \Gamma[V_i], E\right)\subseteq \ov \J_X^d \text{ for some } E\in \Div(\Gamma).
$$

\vspace{0.1cm}

Indeed, it is enough to show that, for any $k\geq 1$, there exists a divisor $E_k\in \Div(\Gamma)$ such that 
\begin{equation}\label{E:orbit-k}
    \TF_X\left(\coprod_{1\leq i \leq k}\Gamma[V_i]\coprod \Gamma[(V_1\cup\ldots \cup V_k)^\mathsf{c}],E_k\right)\subset \ov \J_X^d.
\end{equation}
We reorder the subsets $V_i$ in such a way that 
\begin{equation*}\tag{*}
    \Gamma[(V_1\cup\ldots \cup V_k)^\mathsf{c}] \text{ is connected for any } k\geq 1. 
\end{equation*}

We will prove \eqref{E:orbit-k} by induction on $k$. If $k=1$ then the result follows from the fact that $V_1\in \wh \D(\ov \J_X^d)$. Let us assume that \eqref{E:orbit-k} is true for some $k$ and let us prove for $k+1$. Since $V_{k+1}\in \wh \D(\ov \J_X^d)$, there exists a divisor $B_k\in \Div(\Gamma)$ such that 
\begin{equation}\label{E:orbit-Vk1}
\TF_X(\Gamma[V_{k+1}]\coprod \Gamma[V_{k+1}^\mathsf{c}],B_k)\subset \ov \J_X^d.
\end{equation}
Using the assumption (*), if we choose an edge $e\in E(\Gamma[V_{k+1}],\Gamma[(V_1\cup \ldots \cup V_{k+1})^\mathsf{c}])$ then we have that 
$ \Gamma[(V_1\cup \ldots \cup V_{k})^\mathsf{c}]\setminus \{e\}^\mathsf{c}$ is a connected spanning subgraph of $\Gamma[(V_1\cup \ldots \cup V_{k})^\mathsf{c}]$, where $\{e\}^\mathsf{c}:=E(\Gamma[V_{k+1}],\Gamma[(V_1\cup \ldots \cup V_{k+1})^\mathsf{c}])\setminus \{e\}$.
By applying Step I to condition \eqref{E:orbit-k} for $k$, we can find a divisor $\wt E_k$ such that 
\begin{equation}\label{E:orbit-new}
     \TF_X\left(\coprod_{1\leq i \leq k}\Gamma[V_i]\coprod \Gamma[(V_1\cup \ldots \cup V_{k})^\mathsf{c}]\setminus \{e\}^\mathsf{c},\wt E_k\right)\subset \ov \J_X^d,  
\end{equation}
Since $\ov \J_X^d$ is connected, there exists a divisor $A$ on the graph 
$$G:=\Gamma\setminus (E(\Gamma[V_1\cup \ldots \cup V_k], \Gamma[V_{k+1}])\cup \{e\}^\mathsf{c}) $$
such that 
     \begin{equation}\label{E:maggiore}
     \begin{aligned}
     & (G,A)\geq_{\O_1} \left(\coprod_{1\leq i \leq k}\Gamma[V_i]\coprod \Gamma[(V_1\cup \ldots \cup V_{k})^\mathsf{c}]\setminus \{e\}^\mathsf{c},\wt E_k\right), \\ 
     & (G,A)\geq_{\O_2} (\Gamma[V_{k+1}]\coprod \Gamma[V_{k+1}^\mathsf{c}],B_k),
     \end{aligned}
     \end{equation}
     for some partial orientations $\O_1$ and $\O_2$ of $G$.
     Since $G\setminus \{e\}= \Gamma[V_{k+1}]\coprod \Gamma[V_{k+1}^\mathsf{c}]$, then the partial orientation $\O_2$ has support equal to $\{e\}$ and it must be of the form  
     \begin{equation}\label{E:O2}
      \O_2={\O_G(W_\bullet)} \text{ for } W_\bullet=(V_{k+1},V_{k+1}^\mathsf{c}) \text{ or } (V_{k+1}^\mathsf{c},V_{k+1}).
     \end{equation}
     where $W_\bullet$ is the ordered partition of $V(\Gamma)$ which is equal to either $(V_{k+1},V_{k+1}^\mathsf{c})$ or $(V_{k+1}^\mathsf{c},V_{k+1})$. 

     Since $\ov \J_X^d$ is $\Theta$-complete,  Proposition \ref{P:Theta-red} implies that $\O_1$ and $\O_2=\O_G(W_\bullet)$ must be concordant and if we set  
     $$
     (G,A)\geq_{\O_1\cup \O_2} (\wt G, \wt A) 
     $$
     then $\TF_X(\wt G,\wt A)\subset \ov \J_X^d$. Since we have that 
     $$
     \wt G=G\setminus \supp(\O_1\cup \O_2)=\left(\coprod_{1\leq i \leq k}\Gamma[V_i]\coprod \Gamma[(V_1\cup \ldots \cup V_{k})^\mathsf{c}]\setminus \{e\}^\mathsf{c}\right)\setminus\{e\}=$$
     $$=\coprod_{1\leq i \leq k+1}\Gamma[V_i]\coprod \Gamma[(V_1\cup \ldots \cup V_{k+1})^\mathsf{c}],
     $$
     we have proved \eqref{E:orbit-k} for $k+1$ with $E_{k+1}:=\wt A$.
\end{proof}

As a corollary of the above Proposition (or rather its proof), we now show that if $\ov\J_X^d$ is a connected compactified Jacobian stack then $\D(\ov \J_X^d)$ is a degeneracy subset, and $\wh \D(\ov \J_X^d)$ is the associated extended degeneracy subset (see Definition \ref{D:Deg}).

\begin{corollary}\label{C:prop-deg-cJ}
Let $\ov \J_X^d$ be a non-empty open substack of $\TF_X^d$ that is universally closed, of finite type, connected and $\Theta$-complete (e.g. a connected compactified Jacobian stack of $X$ of degree $d$).
     \begin{enumerate}
         \item \label{C:prop-deg-cJ1}
         Let $W\in \Con(\Gamma_X)$ and let $\Gamma_X[W^\mathsf{c}]=\coprod_i \Gamma_X[V_i]$ be the decomposition into connected component. Then 
         $$W\in \wh \D(\ov \J_X^d) \Leftrightarrow V_i\in \wh \D(\ov \J_X^d) \text{ for any } i.$$
         \item \label{C:prop-deg-cJ2}
          Let $W_1,W_2\in \BCon(\Gamma_X)$ be disjoint and such that $W_1\cup W_2\in \BCon(\Gamma_X)$. Then 
         $$W_1,W_2\in \D(\ov \J_X^d)\Rightarrow W_1\cup W_2\in \D(\ov \J_X^d).$$
     \end{enumerate}
     In particular, $\D(\ov \J_X^d)\in \Deg(\Gamma_X)$ and $\wh D(\ov \J_X^d)$ is its associated extended degeneracy subset.  
\end{corollary}
 \begin{proof}
     Part \eqref{C:prop-deg-cJ1}: the implication $\Rightarrow$ is easy and always true (see Remark \ref{R:obs-deg-cJ}\eqref{R:obs-deg-cJ2}). Conversely, suppose that $V_i\in \D(\ov \J_X^d)$ for every $i$. Arguing as in Step II of the proof of the above Proposition, we can find a divisor $D\in \Div(\Gamma_X)$ such that 
     $$
     \TF_X\left(\coprod_i \Gamma_X[V_i]\coprod \Gamma_X[(\bigcup_i V_i)^\mathsf{c}], D\right)\subseteq \ov \J_X^d.
     $$
     We conclude using that $(\bigcup_i V_i)^\mathsf{c}=W$.
     
     Part \eqref{C:prop-deg-cJ2}: arguing as in Step II of the proof of the above Proposition, we can find a divisor $D\in \Div(\Gamma_X)$ such that 
     $$
     \TF_X(\Gamma_X[W_1]\coprod \Gamma_X[W_2]\coprod \Gamma_X[(W_1\cup W_2)^\mathsf{c}], D]\subseteq \ov \J_X^d.
     $$
     This implies that $(W_1\cup W_2)^\mathsf{c}\in \D(\ov \J_X^d)$ and hence that also $W_1\cup W_2\in \D(\ov \J_X^d)$ (see Remark \ref{R:obs-deg-cJ}\eqref{R:obs-deg-cJ1}). 
 \end{proof}

As an other corollary of the above Proposition, we now prove a lower bound on the number of orbits of a compactified Jacobian stack $\ov\J_X^d$, in terms of the $\D(\ov \J_X^d)$-complexity 
(as in Definition \ref{D:Dadmissible}).  

\begin{corollary}\label{C:bound-orb}
    Let $\ov \J_X^d$ be a non-empty open substack of $\TF_X^d$ that is universally closed, of finite type, connected and $\Theta$-complete (e.g. a connected compactified Jacobian stack of $X$ of degree $d$). Then  
    \begin{enumerate}[(i)]
   \item \label{C:bound-orb1} $\P(\ov \J_X^d)$ contains some BD-set with degeneracy subset $\D(\ov \J_X^d)$;
    \item  \label{C:bound-orb2} for every $\D(\ov \J_X^d)$-admissible spanning subgraph $G\leq \Gamma_X$ we have that 
    $$
    |\P(\ov \J_X^d)(G)|\geq c_{\D(\ov \J_X^d)}(G).
    $$
    In particular, $\ov \J_X^d$ has at least $c_{\D(\ov \J_X^d)}(\Gamma_X)$ irreducible components.
    \end{enumerate}
\end{corollary}
This was proved for fine compactified Jacobians (i.e. if $\D(\ov \J_X^d)=\emptyset$) in \cite[Thm. 2.16]{viviani2023new}.
\begin{proof}
Part \eqref{C:bound-orb1} follows from  Proposition \ref{P:many-orbits}.
Part \eqref{C:bound-orb2} follows from Part~\eqref{C:bound-orb1} together with Lemma~\ref{L:BDrestr} and Proposition \ref{BDthm}\eqref{BDthm1}. The last assertion of Part~\eqref{C:bound-orb2} follows from Corollary \ref{C:sm-irr}\eqref{C:sm-irr2}.
\end{proof}

Observe that if a compactified Jacobian stack ${}^1\ov \J_X$ is contained in another one ${}^2\ov \J_X$, then we have that $\D({}^1\ov \J_X)\subseteq \D({}^2\ov \J_X)$. The next result says that a strict inclusion of connected compactified Jacobian stacks gives rise to a strict inclusions of their degeneracy subsets.

\begin{proposition}\label{P:incl-cJ}
Let ${}^1\ov \J_X$, ${}^2\ov \J_X$ be two connected  compactified Jacobian stacks of $X$ of degree $d$. Assume that:
\begin{enumerate}[(i)]
    \item ${}^1\ov \J_X\subseteq {}^2\ov \J_X$,
    \item $\D({}^1\ov \J_X)= \D({}^2\ov \J_X)$.
\end{enumerate}
Then we have that ${}^1\ov \J_X={}^2\ov \J_X$.
\end{proposition}
\begin{proof}
We are going to use, throughout the proof, that ${}^1\ov \J_X$ and ${}^2\ov \J_X$ are $\Theta$-complete, S-complete and universally closed (by Proposition \ref{P:gms} and \cite[Thm. 5.4]{alper-goodmodspaces}) and union of orbits (by Proposition~\ref{P:unionorb}).

 Since we have an open embedding  ${}^1\ov \J_X\subseteq {}^2\ov \J_X$ and ${}^2\ov \J_X$ is connected, it is enough to show that the embedding ${}^1\ov \J_X\subseteq {}^2\ov \J_X$ is also closed, or equivalently that ${}^1\ov \J_X$ is stable under specializations inside ${}^2\ov \J_X$. 

 With this aim, consider $\I^2\in {}^2\ov \J_X(\Spec R)$ over a DVR $R$ such that its generic fiber $\I^2_\eta$ belongs to ${}^1\ov \J_X(\eta)$. We want to show that also the special fiber $\I^2_o$ belongs to ${}^1\ov \J_X$. Since ${}^1\ov \J_X$ is universally closed, we can find, up to possibly passing to a finite extension of $R$, another sheaf $\I^1\in {}^1\ov \J_X(\Spec R)$ such that $\I^1_\eta=\I_\eta$. Since ${}^2\ov \J_X$ is S-complete,  we can find an ordered partition $Y_\bullet=X[W_\bullet]$ of $X$ such that 
 $$
 J:=\Gr_{Y_\bullet}(\I^1_o)=\Gr_{\ov{Y_\bullet}}(\I^2_o)\in {}^2\ov \J_X.
 $$
 If we set $(G^i,D^i):=(G(\I_o^i),D(\I_o^i))\in \P({}^i\ov \J_X)$ for $i=1,2$ and $(G,D):=(G(J),D(J))\in \P({}^2\ov \J_X)$, then  Corollary \ref{C:iso-spec} implies that
 \begin{equation}\label{E:I-magg}
    (G^1,D^1)\geq_{\O_{G^1}(W_\bullet)} (G,D)\leq_{\O_{G^2}(\ov W_\bullet)} (G^2,D^2).
 \end{equation}
 Choose now an element $(\wt G,\wt D)\in \P({}^1\ov \J_X)$ such that 
 \begin{equation}\label{E:II-magg}
     (G^1,D^1)\geq_{\wt \O} (\wt G,\wt D) \text{ and } \wt G\in m\S\S_{\D({}^1\ov \J_X)}.
 \end{equation}
Since ${}^2\ov \J_X$ is $\Theta$-complete, we can apply Proposition \ref{P:Theta-red} to 
$(G^1,D^1)\geq_{\O_{G^1}(W_\bullet)} (G,D)$ and $(G^1,D^1)\geq_{\wt \O} (\wt G,\wt D)$ in order to deduce that $\wt O$ and $\O_{G^1}(W_\bullet)$ are concordant and that we have 
\begin{equation}\label{E:III-magg}
(G^1,D^1)\geq_{\wt O\cup \O_{G^1}(W_\bullet)} (\wh G,\wh D) \text{ with } (\wh G,\wh D)\in \P({}^2\ov \J_X).
\end{equation}
By combining \eqref{E:I-magg}, \eqref{E:II-magg} and \eqref{E:III-magg}, we deduce that 
\begin{equation}\label{E:IV-magg}
    (\wt G,\wt D)\geq_{\O_{\wt G}(W_\bullet)} (\wh G,\wh D)\leq_{\wt O} (G,D),
 \end{equation}
where $\wh \O:=\wt O\setminus \O_{G^1}(W_\bullet)$. However, since $\wt G$ is a minimal element of $\S\S_{\D({}^1\ov \J_X)}$ and $\D({}^1\ov \J_X)=\D({}^2\ov \J_X)$ by assumption, we deduce that $\wt G$ is also minimal element of $\S\S_{\D({}^2\ov \J_X)}$. Therefore, we must have that 
\begin{equation}\label{E:V-magg}
(\wh G,\wh D)=(\wt G,\wt D)\in \P({}^1\ov \J_X).
\end{equation}
By combining \eqref{E:V-magg}, \eqref{E:IV-magg} and \eqref{E:I-magg}, and using that $\P({}^1\ov \J_X)$ is an upper subset, we get that $(G^2,D^2)\in \P({}^1\ov \J_X)$. This implies that $\I_o^2\in {}^1\ov \J_X$, and we are done. 
\end{proof}

We now state the main Theorem of this subsection that characterizes the V-compactified Jacobians as the ones having the minimal number of irreducible components with a given degeneracy subset.

\begin{theorem}\label{T:smallcJ}
    Let $\ov \J_X^d$ be a connected compactified Jacobian stack with degeneracy subset $\D:=\D(\ov \J_X^d)$. Then $\ov \J_X^d$ has the minimum number of irreducible components, namely $c_\D(\Gamma_X)$,  if and only if $\ov \J_X^d=\ov \J_X(\n)$ for some unique V-stability $\n$ of degree $d$ on $\Gamma_X$ with $\D(\n)=\D$. 
\end{theorem}

This was proved for fine compactified Jacobians (in which case $c_{\emptyset}(\Gamma_X)$ is the complexity of $\Gamma_X$) in \cite[Cor. 2.34]{viviani2023new}.

\begin{proof}
First of all, a V-compactified Jacobian $\ov \J_X(\n)$ has degeneracy subset equal to $\D(\n)$ by Remark~\ref{R:deg-VcJ} and it has $c_{\D(\n)}(\Gamma_X)$ irreducible components by Corollary \ref{C:sm-irr}\eqref{C:sm-irr2} and Theorem \ref{T:V=PT}.

Conversely, assume  that $\ov \J_X^d$ has $c_{\D}(\Gamma_X)$ irreducible components. Proposition~\ref{P:many-orbits} implies that $\P:=\P(\ov \J_X^d)$ contains a BD-set $\BD_I$ with degeneracy subset $\D=\D(\ov \J_X^d)$. Now Proposition \ref{BDthm}, together with our assumption on the number of irreducible components of $\ov \J_X^d$ and Corollary \ref{C:sm-irr}\eqref{C:sm-irr2}, implies that 
\begin{equation}\label{E:lowerPl}
c_{\D}(\Gamma_X)=|\P(\Gamma_X)|\geq |\BD_I(\Gamma_X)|\geq c_{\D}(\Gamma_X).
\end{equation}
We deduce that $\BD_I$ is a weak numerical PT-assignment. Therefore, we can apply Theorem~\ref{T:V=PT} to deduce that there exists a unique V-stability condition $\n$ of degree $d$ on $\Gamma_X$ such that $\BD_I=\P_\n$ and $\D=\D(\n)$. By taking the associated open subsets of $\TF_X^d$, we get the inclusion $\ov \J_X(\n)\subseteq \ov \J_X^d$ with the property that $\D(\ov \J_X(\n))=\D(\n)=\D=\D(\ov \J_X^d)$. We now conclude that $\ov \J_X(\n)=\ov \J_X^d$ by Proposition~\ref{P:incl-cJ}.
\end{proof}

\subsection{Smoothable compactified Jacobians}\label{SUb:smoothcJ}

The aim of this subsection is to classify the
smoothable compactified Jacobians. The main result is the following.

\begin{theorem}\label{T:cla-smoo}
Any
smoothable compactified Jacobians of $X$ is of the form $\ov \J_X(\n)$ for a unique V-stability $\n$ on $\Gamma_X$. 
\end{theorem}

We first prove that the collection of  orbits of  a smoothable compactified Jacobian gives rise to a PT-assignment on the dual graph of the curve.

    \begin{proposition} \label{P:PT-Jbar}
Let $\ov \J_X^d$ be a smoothable compactified Jacobian stack of $X$ of degree $d$.
Then $\P(\ov \J_X^d)$ is a PT-assignment of degree $d$ on $\Gamma_X$ with $\D(\P(\ov \J_X^d))=\D(\ov \J_X^d)$.
    \end{proposition}

\begin{proof}
Set $\P:=\P(\ov \J_X^d)$ and observe that $\ov \J_X^d=U(\P)$ by Proposition~\eqref{P:unionorb}. 
By combining Lemma~\ref{L:JbarVbullet} below with Proposition~\ref{P:cardPTstab}(\ref{P:cardPTstableq}), we deduce the inequality 
\begin{equation}\label{E:upperP}
|\P(G)|\leq c_{\D}(G),
\end{equation}
for every $G\in\S\S_{\D(\ov\J_X)}(\Gamma_X)$.

On the other hand, Proposition~\ref{P:many-orbits} implies that $\P$ contains some BD-set, say $\BD_I$, with degeneracy subset $\D(\ov \J_X^d)$. Therefore, Proposition~\ref{BDthm} gives the opposite inequality
\begin{equation}\label{E:lowerP}
|\P(G)|\geq |\BD_I(G)|\geq c_{\D}(G).
\end{equation}
By combining the inequalities in \eqref{E:upperP} and in \eqref{E:lowerP}, we deduce that $\BD_I=\P$ is a numerical PT assignment. By Theorem~\ref{T:V=PT}\eqref{T:V=PT2}, we conclude that $\P$ is a PT-assignment  with degeneracy subset $\D(\ov \J_X^d)$.
\end{proof}
The following was used in the above proof.
\begin{lemma} \label{L:JbarVbullet}Let $G\leq \Gamma_X$ be a $\D(\ov\J_X^d)$-admissible subgraph. Assume that $D \in \P(G)$ and $D'=D+\div(f) \in \P(G)$ for some $D,D' \in \Div(G)$ and some $f\in C^0(G,\ZZ)$. Then there exists an ordered partition $V_\bullet$ of $V(G)$ such that 
\[
D- D(\mathcal{O}(V_\bullet))=D+\div(f)- D(\mathcal{O}(\overline{V_\bullet})) \in \P(G \setminus E_{G}(V_{\bullet})). 
\]
\end{lemma}
\begin{proof} 
Fix an orientation of $E(\Gamma_X)\setminus E(G)$, such that $f(s(\ee))\geq f(t(\ee))$, and let 
\begin{align*}
    m: E(\Gamma_X)&\to \NN\\
    e&\mapsto \begin{cases}
        f(s(\ee))-f(t(\ee))+1, \textup{ if }e\notin E(G),\\
        0,\textup{ if }e\in E(G).
    \end{cases}
\end{align*}
Let $\X/\Delta$ be a one-parameter smoothing such that, for a fixed  uniformizer $\pi$ of $R$, for any edge $e\in E(\Gamma_X)$, a local equation of the surface $\X$ at the corresponding node $n_e$ of $X$ is of the form $xy=\pi^{m(e)+1}$, for some $m(e)\in\NN$. 
This defines a function
$$ m:E(\Gamma_X)\to \NN.$$
After performing $m(e)$ blow-ups of the surface $\X$ at each point $n_e$, we get a regular surface $\X^m$ endowed with a morphism 
\begin{equation*}
    \sigma:\X^m\to \X
\end{equation*}
that restricts to an isomorphism $\sigma^\eta : \X^m_\eta \to \X_\eta$ on the generic fiber, and to a contraction $\sigma^o : X^m := \mathcal{X}^m_o \to X=\mathcal{X}_o$ on the central fiber. The dual graph $\Gamma_{X^m}$ of $X^m$ is the graph $\Gamma_X^m$ obtained by subdividing each edge $e$ of $\Gamma_X$ into $m(e)+1$ edges and inserting $m(e)$ new edges. We denote the new (exceptional) vertices ordered from $s(\ee)$ to $t(\ee)$ by $v_1^\ee,\ldots, v_{m(e)}^\ee$ in such a way that $s(\ee^i)=v_{i-1}^{\ee}$ and $t(\ee^i)=v_i^{\ee}$.

We now lift the divisors $D,D'$ to divisors $\widetilde{D}, \widetilde{D}'$ on $\Gamma_{X^m}$, defined as follows:
\begin{itemize}
    \item $\widetilde{D}_v:=\begin{cases}
        \delta_{1,i}, \textup{ if }v=v^\ee_i,\\
        D_v, \textup{ if }v\in V(\Gamma_X).
    \end{cases}$
    \item $\widetilde{D}'_v:=\begin{cases}
        \delta_{m(e),i}, \textup{ if }v=v^\ee_i,\\
        (D+\div(f))_v, \textup{ if }v\in V(\Gamma_X).
    \end{cases}$
\end{itemize}
We claim that there exists a function $\widetilde{f} : V(\Gamma_{X^m}) \to \mathbb{Z}$ that gives  the equality
\begin{equation}\label{E:tildeDequiv}
    \widetilde{D}'=\widetilde{D}+\div(\widetilde{f}).
\end{equation}

 Assuming the claim, we let $[F] \in \ov \J_X^d$ be such that $\deg (F) =D \in \P(G)$. There exists a line bundle $L$ on $X^m$ such that $\sigma^0_*(L)=F$, and we can choose $L$ to have multidegree equal to $\widetilde{D}$ (this follows from \cite[Proposition~5.5]{estevespacini}: with the notation as in loc.cit., take $X^m \to Y$ to be the map that contracts all exceptional components of $\sigma^0 : X^m \to X$ where $\widetilde{D}$ has degree $0$, and set $L$ to be the pullback of $\mathcal{L}$.).  By Hensel's lemma, there exists a line bundle $\widetilde{\L}$ on $\X^m$ such that $\widetilde{\L}_o=L$.
We then let $\widetilde{L}'=\widetilde{\L} \otimes \mathcal{T}$, where we set 
\[
\mathcal{T}:=\mathcal{O}_{\mathcal{X}^m}\Bigg(\sum_{v \in V(\Gamma_{X^m})} \tilde{f}(v) X^m_v\Bigg),
\]
a line bundle that is trivial on the generic fiber $\X^m_\eta$, and whose multidegree equals $\div(\tilde{f})$ on the central fiber $\X^m$. 

Since  $\ov \J_X^d$ is smoothable, the stack $\ov\J_\X:=\ov\J_X^d\cup J^d_{\X_\eta}\subset \TF^d_{\X/\Delta}$ is $S$-complete over $\Delta$. Moreover, we have  that $(\sigma_*\widetilde{\L})_\eta=(\sigma_*\widetilde{\L}')_\eta$ and that $(\sigma_* \widetilde{\L})_0 \in \ov \J^d_X$ by construction  and $(\sigma_* \widetilde{\L}')_0 \in \ov \J^d_X$ because its multidegree equals $D'=D+ \div(f) \in \P(G)$, and as a consequence of Proposition~\ref{P:unionorb}.

Thus, by S-completeness, there exists an ordered partition $Y_\bullet=(Y_0,\ldots,Y_q)$ such that $$\Gr_{Y\bullet}((\sigma_*\widetilde{\L})_o)=\Gr_{\ov Y\bullet}((\sigma_*\widetilde{\L}')_o)\in\ov\J_X^d.$$ 
The statement then follows by setting $V_\bullet:=(V(\Gamma_{Y_0}),\ldots,V(\Gamma_{Y_q}))$.

We now prove the claim. Firstly, we define the function $\widetilde{f}$: 
\begin{align*}
    \wt f:V(\Gamma_{X^m})&\to \ZZ\\
    v&\mapsto \begin{cases}
        f(s(\ee))+1-i, \textup{ if }v=v^\ee_i,\\
        f(v),\textup{ if }v\in V(\Gamma_X).
    \end{cases}
\end{align*}

In order to prove Equation~\ref{E:tildeDequiv}, we distinguish two cases:
\begin{itemize}
    \item[1. $v\in V(\Gamma_X)$.] Since $\widetilde{D}_v=D_v$ and $\widetilde{D}'_v=(D+\div(f))_v$, we need to show $\div(\widetilde{f})_v=\div(f)_v$.
    \begin{align*}
         \div(\widetilde{f})_v&=-\widetilde{f}_v\val_{\Gamma_{X^m}}(v)+\sum_{v\neq w\in V(\Gamma_{X^m})}\tilde{f}(w)\val_{\Gamma_{X^m}}(v,w)=\\
         &=-{f}_v\val_{\Gamma}(v)+\sum_{v\neq w\in V(\Gamma_{X})}f(w)\val_{G}(v,w)+\sum_{e\in E(\Gamma)\setminus E(G)}(f(s(e))\val_{\Gamma_{X^m}}(v^\ee_1,v)+f(t(\ee))\val_{\Gamma_{X^m}}(v^\ee_{m(e)},v))=\\&=-{f}_v\val_{\Gamma}(v)+\sum_{v\neq w\in V(\Gamma_{X})}f(w)\val_{G}(v,w)+\sum_{e\in E(\Gamma)\setminus E(G)}f(v)=\\&=-{f}_v\val_{G}(v)+\sum_{v\neq w\in V(G)}f(w)\val_{G}(v,w)=\div(f)_v,
    \end{align*}
    where we have used the equalities $\val(v^\ee_1,v)=\delta_{v,s(\ee)}$, $\val(v^\ee_{m(e)},v)=\delta_{v,t(\ee)}$ and $m(e)=f(s(\ee))-f(t(\ee))+1$.
   
     \item[2. $v=v^\ee_i$ for] some $1\leq i\leq m(e)$. We observe that 
    \begin{equation*}
        \div(\widetilde{f})_v=-2\widetilde{f}(v_i^\ee)+\widetilde{f}(v^\ee_{i-1})+\widetilde{f}(v^\ee_{i+1})=\begin{cases}
            0,\textup{ if }i\neq1,m(e),\\
            -1, \textup{ if }i=1,\\
            1, \textup{ if }i=m(e),
        \end{cases}
    \end{equation*}
    which, together with
    $$
    \begin{cases}
        \widetilde{D}_v=\delta_{1,i},\\
        \widetilde{D}'_v=\delta_{m(e),i},
    \end{cases}
    $$
    shows the equality.
\end{itemize}
This concludes the proof of the Claim.
\end{proof}

Finally, we can prove the main result of this subsection.

\begin{proof}[Proof of Theorem \ref{T:cla-smoo}]
    Let $\ov \J_X$ be a
    smoothable compactified Jacobian of $X$. 
By Proposition \ref{P:PT-Jbar} we deduce that $\P(\ov \J_X)$ is a PT-assignment. Theorem~\ref{T:V=PT} implies that $\P=\P_\n$ for a unique V-stability $\n$ on $\Gamma_X$ and hence $\ov \J_X=U(\P(\ov \J_X^d))=U(\P_n)=\ov\J_X(\n)$ by Remark~\ref{R:Pn-cJ} and Equality~\eqref{E:U-cJ}.
\end{proof}

\section{Examples}\label{Sec:Examples}

In this section, we discuss examples of  compactified Jacobians or V-compactified Jacobians for some special classes of curves.

\subsection{Maximally degenerate compactified Jacobians}\label{Sub:maxdeg}

In this subsection, we classify all maximally degenerate connected compactified Jacobians, in the following sense.

\begin{definition}\label{D:maxdeg}
Let $X$ be a connected nodal curve. A connected compactified Jacobian stack $\ov \J_X^d$ of $X$ is said to be \emph{maximally degenerate} if $\D(\ov \J_X^d)=\BCon(X)$.    
\end{definition}

Let first introduce a canonical compactified Jacobian stack on $X$.

\begin{lemma}\label{L:candeg}
Let $X$ be a connected nodal curve. 
\begin{enumerate}
    \item \label{L:candeg1}
    The function 
$$ 
\begin{aligned}
 \psi^{\deg}(X):\left\{\text{Subcurves of } X\right\} & \rightarrow \RR\\
 Y & \mapsto \psi^{\deg}(X)_Y:=0,
\end{aligned}$$
is a numerical polarization on $X$ of characteristic $0$, called the \emph{canonical maximally degenerate} numerical polarization on $X$.
\item \label{L:candeg2} The associated V-stability condition $\s^{\deg}(X):=\s(\psi^{\deg}(X))$, called the \emph{canonical maximally degenerate} V-stability condition on $X$,  is given by 
$$
\s^{\deg}(X)_{Y}:=0 \text{ for any } Y\in \BCon(X),
$$
and its degeneracy subset is equal to  $\D(\n^{\deg}(X))=\BCon(X)$.
\item \label{L:candeg3}
Denote by $\n^{\deg}(X)$ the V-stability condition on $\Gamma_X$ of degree $g(X)-1$ corresponding to $\s^{\deg}(X)$ under the bijection of Proposition \ref{P:n-s}. Then the $\n^{\deg}(X)$-semistable set $\P_X^{\deg}:=\P_{\n^{\deg}(X)}$ is equal to 
$$\P_X^{\deg}=\BD_{I^{\deg}_X},$$
where $I^{\deg}_X$ is the $\BCon(\Gamma_X)$-forest function given by 
$$
\begin{aligned}
    I^{\deg}_X:m\S\S_{\BCon(\Gamma_X)}(\Gamma_X) & \longrightarrow \Div(\Gamma_X),\\
    \Gamma_X\setminus E(\Gamma_X) & \mapsto D_X^{\deg}:=\sum_{v\in V(\Gamma_X)} [g(X_v)-1]v.
    \end{aligned}
$$
\end{enumerate}
\end{lemma}
\begin{proof}
    Parts \eqref{L:candeg1} and \eqref{L:candeg2} are obvious.
    Part \eqref{L:candeg3} follows from the fact that $\Gamma_X\setminus E(\Gamma_X)$ is the unique minimal $\BCon(\Gamma_X)$-admissible spanning subgraph together with Theorem \ref{T:V=PT}\eqref{T:V=PT1} and the fact that $I_{\n^{\deg}(X)}=I^{\deg}_X$ as it follows easily from Lemma-Definition \ref{L:BDInwelldef}. 
\end{proof}

\begin{corollary}\label{C:posdeg}
  Let $X$ be a connected nodal curve. Then for every spanning subgraph $G\leq \Gamma_X$, we have that 
  $$
  \P_X^{\deg}(G)=D_X^{\deg}+\left\{D(\O): \: \O \text{ is an orientation of $G$}\right\}\subset \Div^{g(X)-1-|E(G)^\mathsf{c}|}(\Gamma_X).
  $$
 Moreover, we have that 
  $$
  |\P_X^{\deg}(G)|=c_{\BCon(\Gamma_X)}(G)=|\left\{\text{Spanning forests of } G\right\}|.
  $$
\end{corollary}
\begin{proof}
The description of $\P_X^{\deg}(G)$ follows from Lemma \ref{L:candeg}\eqref{L:candeg3} together with the Definition \ref{D:BD-set} of a BD-set. The formula for $|\P_X^{\deg}(G)|$ follows from the fact that  
$$\left\{D(\O): \: \O \text{ is an orientation of $G$}\right\}$$
is the set of outdegree sequences of $G$ whose cardinality is 
equal to the number of spanning forests of $G$ (see \cite[Prop. 40]{Ber} and the references therein)
\end{proof}

\begin{definition}\label{D:candeg}
    Let $X$ be a connected nodal curve. The compactified Jacobian stack 
    $$\ov \J_X^{\deg}:=\ov \J_X(\s^{\deg}(X))=\{I\in \TF_X: \chi(I)=0 \text{ and } \chi(I_Y)\geq 0 \text{ for any } Y\in \BCon(X)\}$$ 
    is maximally degenerate of degree $g(X)-1$, and it is called the  \emph{canonical maximally degenerate} compactified Jacobian stack of $X$.
\end{definition}

The canonical maximally degenerated compactified Jacobian space has been used by Alexeev in \cite{alexeev} to extend the Torelli morphism from the moduli space of stable curves to the moduli space of principally polarized stable abelic pairs, and it was then studied by Caporaso-Viviani in \cite{caporasoviviani} to prove a Torelli-type theorem for stable curves.

\begin{proposition}\label{P:maxdegcJ}
  Let $X$ be a connected nodal curve.  Then every maximally degenerate compactified Jacobian stack of $X$ is equivalent by translation to $\ov \J_X^{\deg}$. In particular, any maximally degenerate compactified Jacobian stack of $X$ is classical.
\end{proposition}
\begin{proof}
 Let $\ov \J_X$ be a maximally degenerate compactified Jacobian stack. By Corollary \ref{C:bound-orb}\eqref{C:bound-orb1}, $\P(\ov \J_X)$ contains $BD_I$ for some $I$ which is a $\BCon(\Gamma_X)$-forest function. Since $m\S\S_{\BCon(\Gamma_X)}(\Gamma_X)$ consists of only one element, namely $\Gamma_X\setminus E(\Gamma_X)$, we can assume, up to translation, that $I=I^{\deg}_X$. Then Lemma~\ref{L:candeg}\eqref{L:candeg3} implies that $\P_X^{\deg}\subseteq \P(\ov \J_X)$, which in turn implies that $\ov \J_X^{\deg}\subseteq \ov \J_X$. Since $\D(\ov \J_X^{\deg})=\D(\ov \J_X)=\BCon(X)$, Proposition \ref{P:incl-cJ} implies that $\ov \J_X^{\deg}= \ov \J_X$ and we are done.  
\end{proof}

\subsection{Irreducible curves}\label{Sub:irr}

For irreducible nodal curves, there is only one compactified Jacobian. 

\begin{proposition}\label{P:irr-cJ}
 If $X$ is an irreducible nodal curve, then the unique compactified Jacobian of degree $d$ is $\TF_X^d=\Simp_X^d$.   
\end{proposition}
\begin{proof}
    This follows from the fact that $\TF_X^d\fatslash \Gm$ is connected and proper. The latter follows from the fact that it is of finite type by Fact \ref{F:orbits}, it is separated by Proposition \ref{P:non-sepa} and it is universally closed (see \cite[Thm. 32]{esteves} and the references therein). 
\end{proof}

\subsection{Vine curves}\label{Sub:vine}

In this subsection, we classify the compactified Jacobians of a \emph{vine curve $X$}, i.e. $X$ is a nodal curve consisting of two smooth curves $C_1$ and $C_2$ meeting in some nodes. 

\begin{proposition}\label{P:vine-cJ}
Let $X=C_1\cup C_2$ be a vine curve. 
Any compactified Jacobian stack of $X$  is a V-compactified Jacobian stack and hence it is equal to 
$$
\ov \J_X(\s)=\{I\in \TF_X\: : \chi(I)=|\s| \text{ and } \chi(I_{C_i})\geq \s_{C_i}\}
$$
for some unique V-stability condition $\s$ on $X$. 
\end{proposition}

Note that a V-stability condition $\s$ on a vine curve $X=C_1\cup C_2$ of characteristic $\chi$ is given by a pair of integers $\s=(\s_{C_1},\s_{C_2})$ such that 
$$
\s_{C_1}+\s_{C_2}-\chi=
\begin{cases}
1 & \text{ if $\s$ is general,}\\
0 & \text{ if $\s$ is not general.}\\
\end{cases}
$$
Every such V-stability condition $\s$ is classical since $\s=\s(\psi)$ for any $\psi\in \Pol^{\chi}(X)$ such that
$$\psi=(\psi_{C_1},\psi_{C_2})=
\begin{cases}
(\s_{C_1}-\epsilon, \s_{C_2}-(1-\epsilon))  & \text{ if $\s$ is general,}\\ 
(\s_{C_1}, \s_{C_2}) & \text{ if $\s$ is non general,}     
\end{cases}$$
for any $0<\epsilon<1$. In particular, any compactified Jacobian stack on $X$ is classical. 
\begin{proof}
The case of fine compactified Jacobian  was proved in 
\cite[Ex. 7.5]{pagani2023stability}. 

On the other hand, all non-fine compactified Jacobians are maximally degenerate and hence the result follows from Proposition \ref{P:maxdegcJ}. 
\end{proof}

\subsection{Compact type curves}\label{Sub:compact}
The aim of this subsection is to describe all compactified Jacobians of a \emph{compact-type  curve}, i.e. a nodal curve $X$ whose dual graph $\Gamma_X$ is a tree $T$.

First of all, observe that the biconnected subsets of $T$ are described by the following bijection
\begin{equation}\label{E:BCon-tree}
\begin{aligned}
    E(T) \xrightarrow{\cong} & \left\{\text{Pairs of complementary biconnected subsets of $T$} \right\} \\
    e & \mapsto (W_e^1,W_e^2) \: \text{ where } \: T\setminus \{e\}=T[W_e^1]\coprod T[W_e^2],
\end{aligned}
\end{equation}
Moreover, since the second condition in the Definition \ref{D:Deg} of degeneracy subsets for $T$ is automatically satisfied (because there are no disjoint biconnected subsets of $V(T)$ whose union is still biconnected), we have the bijection 
\begin{equation}\label{E:Deg-tree}
  \begin{aligned}
   \Deg(T) & \xleftarrow[]{\cong} 2^{E(T)}\\
   \bigcup_{e\in S} \{W_e^1,W_e^2\} & \mapsfrom S.
  \end{aligned} 
\end{equation}

\begin{proposition}\label{P:comp-cJ}
Let $X$ be a compact-type curve with dual graph  $T$. 
\begin{enumerate}
    \item \label{P:comp-cJ1} For any $(G=T\setminus S,D)\in \OO^d(T)$, the function 
\begin{equation}\label{E:phiGD}
     \begin{aligned}
     \phi(G,D):\left\{\text{Subsets of } V(T) \right\}& \longrightarrow \RR,\\
     W & \mapsto D_W+e_S(W)+\frac{\val_S(W)}{2}.
 \end{aligned}
\end{equation}
is a numerical polarization of degree $d$ on $T$ whose degeneracy subset is equal to $\D(\phi(G,D))=E(G)^\mathsf{c}$ (under the identification \eqref{E:Deg-tree}).
   \item \label{P:comp-cJ2} The compactified Jacobian stack associated to $\n(G,D)$ is given by 
   \begin{equation}\label{E:JGD}
       \ov \J_X(G,D):=\ov \J_X(\phi(G,D))=\bigcup_{(\wt G, \wt D)\geq (G,D)}\TF_X(\wt G,\wt D)
       \end{equation}
   and it is such that 
   \begin{enumerate}
       \item \label{P:comp-cJ2a} The irreducible components of $\ov \J_X(G,D)$ are 
    $$
    \left\{\TF_X(T,D+D(\O)): \: \O \text{ is an orientation of the edges of $E(G)^\mathsf{c}$}\right\},
    $$
    and there are $2^{|E(G)^\mathsf{c}|}$ of them.
    \item \label{P:comp-cJ2b} The good moduli space of $\ov \J_X(G,D)$ is 
     \begin{equation}\label{E:gms-JGD}
    \TF_X(G,D)\xleftarrow[\cong]{(\nu_S)_*}\PIC_{X_S}^D\xleftarrow[\cong]{\nu_S^*}\PIC^D_X.
    \end{equation}
   \end{enumerate}
 \item  \label{P:comp-cJ3}  Any connected compactified Jacobian stack of degree $d$ on $X$ is equal to $\ov \J_X(G,D)$ for some uniquely determined $(G,D)\in \OO^d(T)$.
\end{enumerate}
\end{proposition}
\begin{proof}
Part \eqref{P:comp-cJ1}: the fact that $\phi(G,D)$ is a numerical polarization of degree $d$ on $T$ follows from formulas \eqref{add-val} together with $\deg(D)=d-|S|$. The fact that $\D(\phi(G,D))=S=E(G)^\mathsf{c}$ follows from the fact that (for $e\in E(T)$ and $i=1,2$) 
$$
\phi(G,D)_{W_e^i}-\frac{\val(W_e^1)}{2}=D_{W_e^i}+e_S(W_e^i)+\frac{\val_S(W_e^i)}{2}-\frac{1}{2}\in \ZZ \Leftrightarrow e\in S.$$

Part \eqref{P:comp-cJ2}: let us first prove \eqref{E:JGD}, which is equivalent to showing that the semistable locus for the  V-stability condition $\n(G,D):=\n(\phi(G,D))$ associated to $\phi(G=T\setminus S,D)$ is given by 
\begin{equation}\label{E:PnGD}
\P_{\n(G,D)}=\{(\wt G, \wt D)\in \OO^d(T)\: : (\wt G, \wt D)\geq (G,D)\}.
\end{equation}

Note that $\n(G,D)$ is equal to 
$$
\n(G,D)_{W_e^i}=D_{W_e^i}+e_S(W_e^i) \quad \text{ for any } e\in E(T) \text{ and any } i=1,2.
$$
It follows that $\P_{\n(G,D)}$ contains $(G,D)$ and hence,  being an upper subset, it contains all the elements $(\wt G, \wt D)\in \OO^d(T)$ that dominates $(G,D)$.  

Conversely, let $(\wt G=T\setminus \wt S, \wt D)\in \P_{\n(G,D)}$, which is equivalent to say that 
\begin{equation}\label{E:ineqGD}
\wt D_{W_e^i}+e_{\wt S}(W_e^i)\geq \n(G,D)_{W_e^i}=D_{W_e^i}+e_S(W_e^i) \quad \text{ for any } e\in E(T) \text{ and any } i=1,2.
\end{equation}
By summing \eqref{E:ineqGD} for $i=1,2$ and using that $\deg(D)=d-|S|$ and $\deg(\wt D)=d-|\wt S|$, we discover that:

$\bullet$ $\wt S\subseteq S$;

$\bullet$ for every $e\in \wt S$, we have that $\wt D_{W_e^i}+e_{\wt S}(W_e^i)=D_{W_e^i}+e_S(W_e^i)$ for $i=1,2$;

$\bullet$ for every $e\in S\setminus \wt S$, there exists a unique $i(e)\in \{1,2\}$ such that 
$$
\wt D_{W_e^{i}}+e_{\wt S}(W_e^{i})=
\begin{cases}
    D_{W_e^{i}}+e_S(W_e^{i})+1 & \text{ if } i=i(e),\\
    D_{W_e^{i}}+e_S(W_e^{i}) & \text{ if } i\neq i(e).
\end{cases}
$$
The above properties implies that 
$$(\wt G, \wt D)\geq_{\O} (G,D),$$
where $\O$ is the orientation of the edges of $S\setminus \wt S$  obtained by choosing, for any edge $e\in S\setminus \wt S$, the  orientation $\ee$ of $e$ such that $t(\ee)\in W_{e}^{i(e)}$. 
This concludes the proof of \eqref{E:PnGD}.

We now prove the remaining two assertions of Part \eqref{P:comp-cJ2}.

Part \eqref{P:comp-cJ2a}: the fact that any irreducible component of $\ov \J_X(G,D)$ is of the form $\TF_X(T,D+D(\O))$ for some orientation $\O$ of the edges of $E(G)^\mathsf{c}$ follows from \eqref{E:JGD}, and the fact that they are all distinct follows from the fact that $T$ is a tree.  

Part \eqref{P:comp-cJ2b}: first of all, $\TF_X(G,D)$ satisfies the first isomorphism of \eqref{E:gms-JGD} by Fact  \ref{F:orbits}\eqref{F:orbits3} and the second isomorphism by the fact that $X$ is of compact type. 

It remains to prove that $\TF_X(G,D)$ is the good moduli space of $\ov \J_X(G,D)$.
Since $\ov \J_X(G,D)$ is a compactified Jacobian stack, it has a proper good moduli space
$$
\mu: \ov \J_X(G,D)\to \ov J_X(G,D).
$$
We get therefore a morphism 
$$\phi:\TF_X(G,D)\subset \ov \J_X(G,D)\to \ov J_X(G,D).$$
On the other hand, any stratum $\TF_X(\wt G, \wt D)$ of $\ov \J_X(G,D)$ is such that $(\wt G, \wt D)\geq (G,D)$ by \eqref{E:JGD} and moreover, since $T$ is a tree, we can find a stratification $W_\bullet$ such that $(\wt G, \wt D)\geq_{\O_{\wt G}(W_\bullet)} (G,D)$ and such an orientation $\O_{\wt G}(W_\bullet)$ is unique. Therefore, by Proposition \ref{P:iso-spec}, any sheaf $I\in \ov \J_X(G,D)$ isotrivially degenerate to a canonical sheaf $\Gr(I)\in \TF_X(G,D)$. In other words, we have a morphism 
$$\begin{aligned}
    \Gr: \ov \J_X(G,D)& \longrightarrow \TF_X(G,D)\\
    I & \mapsto \Gr(I).
\end{aligned} $$
By the universal property of the good moduli space, we get a factorization
$$
\Gr: \ov \J_X(G,D)\xrightarrow{\mu} \ov J_X(G,D)\xrightarrow{\psi} \TF_X(G,D).
$$
It is now clear that $\phi$ and $\psi$ are inverses one of the other, and this concludes the proof of Part~\eqref{P:comp-cJ2b}.

 Part \eqref{P:comp-cJ3}: let $\ov \J_X$ be a compactified Jacobian stack of $X$ and let $\D(\ov \J_X)=S\subset E(T)$ under the identification \eqref{E:Deg-tree}. By Corollary \ref{C:bound-orb}\eqref{C:bound-orb1}, $\P(\ov \J_X)$ contains a BD-set $\BD_I$ with degeneracy subset $\D(\ov \J_X)=S$. Since $G:=T\setminus S$ is the unique minimal $\D(\ov \J_X)$-admissible spanning subgraph of $T$, the function $I$ is given by $I(G)=D\in \Div^{d-|S|}(T)$ and hence $\BD_I=\P(\ov \J_X(G,D))$. Therefore, we get the open embedding $\ov \J_X(G,D)\subseteq \ov \J_X$. Since $\ov \J_X$ is connected, it remains to show that the embedding $\ov \J_X(G,D)\subseteq \ov \J_X$ is also closed. Since $\ov \J_X(G,D)$ and $\ov \J_X$ are union of orbits, it is enough to show that 
 $$
 \begin{aligned}
 &\text{ If } (G_1,D_1)\geq_{\O} (G_2,D_2) \text{ with } (G_1,D_1)\in \P(\ov \J_X(G,D)) \text{ and } (G_2,D_2)\in \P(\ov \J_X)\\
 &\text{ then } (G_2,D_2)\in \P(\ov \J_X(G,D)).
 \end{aligned}
 $$
 Since $\P(\ov \J_X(G,D))=\{(\wt G, \wt D)\in \OO^d(T): (\wt G, \wt D)\geq (G,D)\}$ and $T$ is a tree, we can find an ordered partition $W_\bullet$ of $V(T)$ such that $(G_1,D_1)\geq_{\O_{G_1}(W_\bullet)} (G,D)$. 
We can now apply Proposition \ref{P:Theta-red} to the $\Theta$-complete stack $\ov \J_X$ and the three strata $(G_1,D_1)\geq_{\O_{G_1}(W_\bullet)} (G,D)$ and $(G_1,D_1)\geq_{\O} (G_2,D_2)$: we get that $O$ and $\O_{G_1}(W_\bullet)$ are concordant and if we set $(G_1,D_1)\geq_{\O\cup \O_{G_1}(W_\bullet)} (\wt G, \wt D)$ then $(\wt G, \wt D)\in \P(\ov \J_X)$. However, since $G$ is the minimal element of the poset of $D(\ov \J_X)$-admissible spanning subgraphs, we must have that $(\wt G, \wt D)=(G,D)$. Hence, we get that $(G_2,D_2)\geq_{\O_{G_2}(W_\bullet)} (\wt G, \wt D)=(G,D)$ which implies that $(G_2,D_2)\in \P(\ov \J_X(G,D))$ and we are done. 

Finally, notice that if $\ov \J_X(G_1,D_1)=\ov \J_X(G_2,D_2)$ then  \eqref{P:comp-cJ2} implies  that $(G_1,D_1)\geq (G_2,D_2)$ and $(G_2,D_2)\geq (G_1,D_1)$, which forces $(G_1,D_1)=(G_2,D_2)$.
\end{proof}

\begin{corollary}\label{C:comp-cJ}
    Let $X$ be a compact-type curve.
    \begin{enumerate}
        \item All compactified Jacobians of $X$ are classical.
        \item There is an anti-isomorphism of posets
        $$
        \begin{aligned}
         \OO^d(\Gamma_X) & \longrightarrow \{\text{Compactified Jacobian stacks of $X$ of degree $d$}\}\\
         (G,D) & \mapsto \ov \J_X(G,D),
        \end{aligned}
        $$
        where the order relation on compactified Jacobian stacks  is given by inclusion.
        \item There is a bijection 
        $$
        \begin{aligned}
          \frac{\left\{\text{Compactified Jacobians of $X$}\right\}}{\text{translation}} &\xrightarrow[\cong]{\D(-)}   \Deg(\Gamma_X) \xrightarrow[\cong]{} 2^{V(\Gamma_X)}\\
        \ov \J_X(G,D)& \mapsto \D(\ov \J_X(G,D))  \mapsto E(G)^\mathsf{c}. 
        \end{aligned}
        $$
        where the last bijection is induced by \eqref{E:BCon-tree}.
    \end{enumerate}
\end{corollary}

\subsection{Necklace (or cycle) curves}\label{Sub:cycle}
The aim of this subsection is to describe the V-stability conditions on a necklace (or cycle) curve with $n$ nodes, i.e. a nodal curve whose dual graph is the necklace (or cycle) graph $C_n$ of length $n$ which is the graph  whose vertices are $\{v_1,\ldots, v_n\}$ and whose edges are $\{e_1,\ldots, e_n\}$ in such a way that the edge $e_i$ joins $v_i$ and $v_{i+1}$, with the cyclic convention that $v_{n+1}:=v_1$ and $v_0:=v_n$ (see Figure\ref{F:Necklace}).

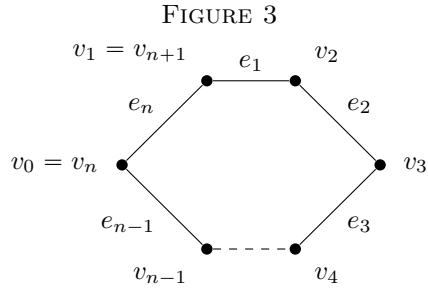
\begin{figure}[h!]
			\caption{}
			\label{F:Necklace}
			\centering
			\begin{tikzpicture}
				\node(0)[circle,fill,inner sep=1.5pt]{};
                \node(l0)[left =.1 of 0]{$v_0=v_n$};
				\node(1)[above right = of 0][circle,fill,inner sep=1.5pt]{};
                \node(l1)[above left =.1 of 1]{$v_1=v_{n+1}$};
				\node(2)[right = of  1][circle,fill,inner sep=1.5pt]{};
                \node(l2)[above right =.1 of 2]{$v_2$};
				\node(7)[below right = of 2][circle,fill,inner sep=1.5pt]{};
                \node(l7)[right =.1 of 7]{$v_3$};
				\node(5)[below right = of 0][circle,fill,inner sep=1.5pt]{};
                \node(l5)[below left =.1 of 5]{$v_{n-1}$};
				\node(6)[below left = of  7][circle,fill,inner sep=1.5pt]{};
                \node(l6)[below right =.1 of 6]{$v_4$};

				\draw (0) edge["$e_n$"] (1);

				\draw (5) to (6)[dashed];
				\draw (1) edge["$e_1$"] (2);
				\draw (7) edge["$e_2$",swap] (2);

				\draw (0) edge["$e_{n-1}$",swap] (5);
				\draw (7) edge["$e_3$"] (6);
				
			\end{tikzpicture}
		\end{figure}
Note that a more general classification of \emph{all} connected compactified Jacobians of necklace curves has recently been given in Santi's PhD thesis~\cite[Chapter~3]{santi}. As part of this classification, Santi also identifies in \cite[Chapter~4]{santi} the \emph{smoothable} fine compactified Jacobians, providing a combinatorial classification that is equivalent to the one given here.

The nontrivial biconnected subsets of $C_n$ are the subsets of $V(C_n)$ given by 
$$V_{ij}:=\begin{cases}
 \{v_{i+1},\ldots, v_j\}   & \text{ if } i<j,\\
  \{v_{i+1},\ldots,v_n,v_1,\ldots,  v_j\}   & \text{ if } j<i.\\
\end{cases}$$
for any pair $i\neq j\in [n]$. Note that $V_{ij}^\mathsf{c}=V_{ji}$. We also set $V_{ii}=\emptyset$ for any $i\in [n]$.

Let $I_{\bullet}=I_1\coprod \ldots \coprod I_r$ be an ordered partition of $[n]:=\{1,\ldots, n\}$ of some length $1\leq r:=l(I_\bullet)\leq n$. The position of an element $i\in [n]$ with respect to $I_\bullet$ is defined to be the integer $p_{I_{\bullet}}(i)=p(i)\in \{1,\ldots, r\}$ such that $i\in I_{p(i)}$.  We get an induced linear (or total) preorder $\leq_{I_{\bullet}}$ on $[n]$ by 
$$i\leq_{I_\bullet} j \stackrel{\textup{def}}{\Longleftrightarrow} p(i)\leq p(j).$$
Furthermore, any linear preorder on $[n]$ is of the form $\leq_{I_\bullet}$, for a unique ordered partition $I_\bullet$ on $[n]$.
We will also write that $i\sim_{I_\bullet} j$ if  $p(i)=p(j)$, or equivalently if and only if $i\leq_{I_\bullet} j$ and $j\leq_{I_\bullet} i$, and we write $i<_{I_\bullet} j$ if $p(i)<p(j)$, or equivalently if $i\leq_{I_\bullet} j$ and $i\not\sim_{I_\bullet} j$.
Note that $\sim_{I_\bullet}$ is an equivalence relation on $[n]$.

We now show how to construct V-stability conditions on $C_n$ starting from an ordered partition of $[n]$ and an integral function on $V(C_n)$.

\begin{definition}\label{D:nIf}
Let $I_{\bullet}=I_1\coprod \ldots \coprod I_r$ be an ordered partition of $[n]$ and let $f:V(C_n)\to \ZZ$ be any function and set $f(W):=\sum_{v\in W} f(v)$ for any subset $W\subseteq V(C_n)$. Consider the function 
$$
\begin{aligned}
\n(I_\bullet,f):\BCon(C_n) & \longrightarrow \ZZ,\\
V_{ij} & \mapsto \n(I_\bullet,f)_{V_{ij}}:=
\begin{cases}
-1+f(V_{ij})  & \text{ if } i\sim_{I_\bullet} j, \\
-1+f(V_{ij})  & \text{ if } i<_{I_\bullet} j, \\
f(V_{ij})  & \text{ if } i>_{I_\bullet} j. \\
\end{cases} 
\end{aligned}
$$
We will write  $\n(I_\bullet):=\n(I_\bullet, \un 0)$ where $\un 0$ is the identically zero function. 
\end{definition}

\begin{lemma}\label{L:nfI}
Notation as in Definition \ref{D:nIf}. 
\begin{enumerate}
    \item \label{L:nfI1} The function $\n(I_\bullet,f)$ is a V-stability on $C_n$ of degree $|f|:=f(V(C_n))$.
    \item \label{L:nfI2} The degeneracy subset of $\n(I_\bullet, f)$ is equal to 
\begin{equation}\label{E:DnIf}
\D(\n(I_\bullet, f))=\{V_{ij}  \: : i\sim_{I_\bullet} j, i\neq j\}.
\end{equation}
\end{enumerate}
\end{lemma}

\begin{proof}
Let us prove Part \eqref{L:nfI1} by checking the conditions of Definition \ref{D:Vstab}. 
For any $V_{ij}\in \BCon(C_n)$, we have that 
$$
\n(I_\bullet,f)_{V_{ij}}+\n(I_\bullet,f)_{V_{ji}}=
\begin{cases}
 |f|-2=|f|-\val(V_{ij}) & \text{ if } i\sim_{I_\bullet} j, \\
 |f|-1=|f|-\val(V_{ij})+1 & \text{ if } i\not \sim_{I_\bullet} j.
\end{cases}
$$
This shows that \eqref{E:sum-n} is satisfied with $d=|f|$ and that the nontrivial biconnected subsets of $C_n$ that are $\n(I_\bullet,f)$-degenerate are the subsets $V_{ij}$ such that $i\sim_{I_\bullet} j$. This also shows Part \eqref{L:nfI2}.

Consider now three nontrivial biconnected subsets $Z_1,Z_2,Z_3$ of $C_n$ such that 
$V(C_n)=Z_1\cup Z_2\cup Z_3$. Then there exist three pairwise distinct indices $ i,j,k\in [n]$, cyclically ordered with respect to the natural cyclic order on $[n]$, such that $Z_1=V_{ij}$, $Z_2=V_{jk}$ and $Z_3=V_{ki}$. If two of them are $\n(I_\bullet,f)$-degenerate, then also the third one is $\n(I_\bullet,f)$-degenerate since $\sim_{I_\bullet}$ is an equivalence relation. Otherwise we have that 
$$
\n(I_\bullet,f)_{V_{ij}}+\n(I_\bullet,f)_{V_{jk}}+\n(I_\bullet,f)_{V_{ki}}=
\begin{cases}
 |f|-3 & \text{ if } i\sim_{I_\bullet} j\sim_{I_\bullet} k, \\
 |f|-2 & \text{ if } i\sim_{I_\bullet} j\not \sim_{I_\bullet} k \text{ (up to cyclic permutations of the indices)}, \\
  |f|-2 & \text{ if } i<_{I_\bullet} j<_{I_\bullet} k \text{ (up to cyclic permutations of the indices)}, \\
   |f|-1 & \text{ if } i>_{I_\bullet} j>_{I_\bullet} k \text{ (up to cyclic permutations of the indices)}. \\
\end{cases}
$$
This shows that \eqref{E:tria-n} is satisfied with $d=|f|$ since $E(V_{ij}, V_{jk})+E(V_{jk},V_{ki})+E(V_{ki},V_{ij})=3$. 
\end{proof}

We now want to show that any V-stabiity condition on $C_n$ is given by the above construction. 

With this aim, we introduce some definitions. Two ordered partitions $I_\bullet$ and $I'_\bullet$ of $[n]$ are cyclic equivalent, and we will write $I_\bullet \sim_{cyc} I'_{\bullet}$, if and only if 
$$\begin{sis}
&I_\bullet=J_\bullet\coprod K_{\bullet},\\
& I'_\bullet=K_\bullet\coprod J_\bullet,\\
\end{sis}$$
for some ordered partitions $J_\bullet$ and $K_\bullet$ of two disjoint subsets of $[n]$. A cyclic equivalence class of ordered partitions on $[n]$ is called a cyclic partition of $[n]$ and it is denoted by $I_\bullet^{cyc}$. 
A cyclic partition $I_\bullet^{cyc}$ on $[n]$ determines a cyclic preorder on $[n]$, namely the equivalence relation $\sim_{I_\bullet^{cyc}}=\sim_{I_\bullet}$ and the ternary relation $R_{I_{\bullet}^{cyc}}$ given on equivalence classes by 
$$
R_{I_{\bullet}^{cyc}}(i,j,k)\Longleftrightarrow i<_{I_\bullet} j<_{I_\bullet} k \text{ or }
j<_{I_\bullet} k<_{I_\bullet} i \text{ or } k<_{I_\bullet} i<_{I_\bullet} j. 
$$
Furthermore, any cyclic preorder on $[n]$ is of the form $(\sim_{I_\bullet^{cyc}}, R_{I_{\bullet}^{cyc}})$ for a unique cyclic partition $I_\bullet^{cyc}$ of $[n]$.

\begin{proposition}\label{P:VCn}
\noindent
 \begin{enumerate}
 \item \label{P:VCn1} Any V-stability condition on $C_n$ is of the form $\n(I_\bullet,f)$, for some ordered partion $I_\bullet$ and some function $f$.
 \item \label{P:VCn2} Two V-stability conditions $\n(I_\bullet,f)$ and $\n(I'_\bullet,f')$ are equal if and only if 
 \[\begin{sis}
&I_\bullet=J_\bullet\coprod K_{\bullet},\\
& I'_\bullet=K_\bullet\coprod J_\bullet,\\
& f(v_i)=
\begin{cases}
f'(v_i)+1 & \text{ if } i-1\in J_\bullet, i\in K_\bullet, \\
f'(v_i)-1 & \text{ if } i-1\in K_\bullet, i\in J_\bullet, \\
f'(v_i) & \text{ otherwise. }
\end{cases}
\end{sis}\]
for some ordered partitions $J_\bullet$ and $K_\bullet$ of two disjoint subsets of $[n]$ (with the cyclic convention that $0=n\in [n]$). 
\item \label{P:VCn3} Any V-stability condition on the graph $C_n$ is classical, i.e. of the form $\n(\phi)$ for some numerical polarization $\phi$.
 \item \label{P:VCn4} Two V-stability conditions $\n(I_\bullet,f)$ and $\n(I'_\bullet,f')$ are equivalent by translation if and only if $I_\bullet$ and $I'_\bullet$ are cyclic equivalent. 
 \end{enumerate}
 \end{proposition}
\begin{proof}
Part \eqref{P:VCn1}: let $\n$ be a V-stability on $C_n$ of degree $d$. 
Define the binary relation $\sim$ on $[n]$ by 
$$
i\sim j \stackrel{\text{def}}{\Longleftrightarrow} i=j \text{ or } V_{ij} \text{ is $\n$-degenerate.} 
$$
From Definition \ref{D:Vstab}, it follows that $\sim$ is an equivalence relation. 

For any three indices $i,j,k\in [n]$, not pairwise equivalent under $\sim$, we define 
$$
R(i, j, k) \stackrel{\text{def}}{\Longleftrightarrow} 
\begin{cases} 
 \n_{V_{ij}}+\n_{V_{jk}}+\n_{V_{ki}}+3-d=1, & \text{ if } i,j,k \text{ are cyclically ordered,}\\
 \n_{V_{ik}}+\n_{V_{kj}}+\n_{V_{ji}}+3-d=2, & \text{ if } i,j,k \text{ are not cyclically ordered,}
 \end{cases}
$$
This ternary relation is a cyclic preorder because
\begin{itemize}
    \item it is cyclic since $\n_{V_{ij}}+\n_{V_{jk}}+\n_{V_{ki}}$ and $\n_{V_{ik}}+\n_{V_{kj}}+\n_{V_{ji}}$ are invariant under cyclic permutations;
    \item it is asymmetric, i.e. $R(i,j,k)\Rightarrow \not \!\!R(i,k,j)$,  by definition;
    \item it is transitive, i.e. $R(i,j,k)$ and $R(i,k,h)$ imply that $R(i,j,h)$ and $R(j,k,h)$. 
    
   Indeed, suppose for simplicity that $i,j,k,h$ are cyclically oriented (the other cases are left to the reader). Using 
   that $\n_{V_{ki}}+\n_{V_{ik}}-d+2=\n_{V_{jh}}+\n_{V_{hj}}-d+2=1$ since $V_{ki}$ and $V_{jh}$ are $\n$-nondegenerate, we have that 
\begin{equation}\label{E:ijkh}
\begin{sis}
    & \n_{V_{ij}}+\n_{V_{jk}}+\n_{V_{ki}}+\n_{V_{ik}}+\n_{V_{kh}}+\n_{V_{hi}}=
    \n_{V_{ij}}+\n_{V_{jk}}+\n_{V_{kh}}+\n_{V_{hi}}+d-1,\\
    & \n_{V_{ij}}+\n_{V_{jh}}+\n_{V_{hi}}+\n_{V_{jk}}+\n_{V_{kh}}+\n_{V_{hj}}=\n_{V_{ij}}+\n_{V_{jk}}+\n_{V_{kh}}+\n_{V_{hi}}+d-1\\
\end{sis}
\end{equation}
\end{itemize}
The first line in \eqref{E:ijkh} is equal to $2(d-2)+d-1$ because $R(i,j,k)$ and $R(i,k,h)$ hold true by hypothesis. Therefore, since $$\n_{V_{ij}}+\n_{V_{jh}}+\n_{V_{hi}}, \n_{V_{jk}}+\n_{V_{kh}}+\n_{V_{hj}}\in \{d-2,d-1\}$$
and the second line in \eqref{E:ijkh} is equal to the first line in \eqref{E:ijkh}, we must have that 
$$\n_{V_{ij}}+\n_{V_{jh}}+\n_{V_{hi}}= \n_{V_{jk}}+\n_{V_{kh}}+\n_{V_{hj}}=d-2,$$
which is equivalent to say that $R(i,j,h)$ and $R(j,k,h)$. 

Choose now any linear preorder $\leq$ inducing the cyclic preorder $(\sim, R)$ and let $I_{\bullet}$ be the  ordered partition of $[n]$ such that $\leq=\leq_{I_\bullet}$. Consider the  function
\begin{equation}\label{E:def-F}
\begin{aligned}
    F:\BCon(C_n) & \longrightarrow \ZZ, \\
    V_{ij} & \mapsto  F(V_{ij}):=
    \begin{cases}
    \n_{V_{ij}}+1 & \text{ if } i\sim j,\\
    \n_{V_{ij}}+1 & \text{ if } i< j,\\
     \n_{V_{ij}} & \text{ if } i> j.\\
     \end{cases}
\end{aligned}
\end{equation}
From the first condition of Definition \ref{D:Vstab}, we compute for any $i\neq j$:
\begin{equation}\label{E:F-cond1}
   F(V_{ij})+F(V_{ji})=
   \begin{cases}
     \n_{V_{ij}}+n_{V_{ji}}+2=d-2+2=d \text{ if } i\sim j, \\
      \n_{V_{ij}}+n_{V_{ji}}+1=d-1+1=d \text{ if } i\not\sim j.
   \end{cases}
\end{equation}
Furthermore, from the second condition of Definition \ref{D:Vstab}, we compute for any pairwise distinct and cyclically ordered $i,j,k$:
\begin{equation}\label{E:F-cond2}
F(V_{ij})+F(V_{jk})+F(V_{ki})=
\begin{cases}
    \n_{V_{ij}}+\n_{V_{jk}}+\n_{V_{ki}}+3=d-3+3=d & \text{ if } i\sim j\sim k,\\
      \n_{V_{ij}}+\n_{V_{jk}}+\n_{V_{ki}}+2=d-2+2=d & \text{ if } i\sim j\not \sim k,\\
      \n_{V_{ij}}+\n_{V_{jk}}+\n_{V_{ki}}+2=d-2+2=d & \text{ if } i<j< k \text{ (up to cyclic permut.)},\\
      \n_{V_{ij}}+\n_{V_{jk}}+\n_{V_{ki}}+1=d-1+1=d & \text{ if } i>j> k \text{ (up to cyclic permut.)}.\\
\end{cases}
\end{equation}
Conditions \eqref{E:F-cond1} and \eqref{E:F-cond2} imply that the function 
\begin{equation}\label{E:def-f}
\begin{aligned}
 f: V(C_n)& \longrightarrow \ZZ,\\
 v_i & \mapsto f(v_i):=F(V_{i-1,i}),
\end{aligned}    
\end{equation}
is such that $|f|=d$ and $F(V_{ij})=f(V_{ij})$.
It follows now from Definition \ref{D:nIf} that $\n=\n(I_\bullet, f)$.

Part \eqref{P:VCn2}: suppose that $\n(I_\bullet, f)=\n(I'_\bullet, f')$. Then by what proved in Part~\eqref{P:VCn1}, it follows that the cyclic preorders induced by $I_\bullet$ and $I_\bullet'$ coincide. This implies that we can write 
 $$\begin{sis}
&I_\bullet=J_\bullet\coprod K_{\bullet},\\
& I'_\bullet=K_\bullet\coprod J_\bullet,\\
\end{sis}$$
for some ordered partitions $J_\bullet$ and $K_\bullet$ of two disjoint subsets of $[n]$. The desired relation between $f$ and $f'$ follows by comparing the definitions of $\n(I_\bullet, f)$ and of $\n(I'_\bullet, f')$, see Definition \ref{D:nIf}. 

Part \eqref{P:VCn3}: consider a V-stability condition $\n(I_\bullet, f)$.
Denote by $p_{I_\bullet}: [n]\to \{1,\ldots, l(I_\bullet)\}$ the position function associated to $I_\bullet$.
Consider the function
$$
\begin{aligned}
   \phi(I_\bullet,f):\{\text{Subsets of }V(C_n)\} & \to \RR, \\
   Z &\mapsto  \phi(I_\bullet,f)_{Z}:=\sum_{v_i\in Z}\frac{p_{I_\bullet}(i-1)-p_{I_\bullet}(i)}{l(I_\bullet)}+f(Z).
\end{aligned}
$$
The function $\phi(I_\bullet, f)$ is  a numerical polarization on $C_n$ of degree equal to $|\phi(I_\bullet, l)|=|f|$. 
For any $i\neq j\in [n]$, we have that 
$$
\n(\phi(I_{\bullet}, f))_{V_{ij}}=
 \left\lceil \phi(I_{\bullet}, f)_{V_{ij}} -1 \right\rceil=
     \left\lceil \frac{p_{I_\bullet}(i)-p_{I_\bullet}(j)}{l(I_\bullet)}-1 \right\rceil+f(V_{ij})=
$$
$$
=\begin{cases}
-1+f(V_{ij}) & \text{ if } i\leq_{I_\bullet} j,\\
    f(V_{ij}) & \text{ if } i>_{I_\bullet} j.
\end{cases}$$
This shows that $\n(\phi(I_{\bullet}, f))=\n(I_\bullet, f)$, and we are done. 

Part \eqref{P:VCn4}: this follows from Part \eqref{P:VCn2} together with the observation that $\n(I_\bullet, f)+\tau=\n(I_\bullet, f+\tau)$ for any function $\tau:V(C_n)\to \ZZ$.
\end{proof}

 Finally we  describe the poset structure on V-stability conditions on $C_n$ of a given degree.
 
 \begin{proposition}
 Consider a V-stability condition $\n(I_\bullet,f)$ of degree $d$. For any other V-stability condition $\n$ of the same degree $d$, we have that 
     $$\n\geq \n(I_\bullet,f) \Longleftrightarrow \textup{ there exists a refinement  } J_\bullet  \text{ of } I_\bullet \textup{ such that } \n=\n(J_\bullet,f).$$ 
\end{proposition}
\begin{proof}
Up to translating with the function $-f$, we can assume that $f$ is the constant zero function.

First of all, we claim that if $J_\bullet$ is a refinement of $I_\bullet$ then  $\n(J_\bullet)\geq \n(I_\bullet)$.
Indeed, this follows from Definition \ref{D:nIf} using that 
$$
\begin{sis}
i\sim_{J_\bullet} j & \Rightarrow  i\sim_{I_\bullet} j, \\
i<_{J_\bullet} j & \Rightarrow  i\leq_{I_\bullet} j,\\
i>_{J_\bullet} j & \Rightarrow  i\geq_{I_\bullet} j,\\
\end{sis}
$$

Conversely, assume that $\n\geq \n(I_\bullet,f)$. From \eqref{E:n1>n2} it follows that for any $i\neq j\in [n]$ we must have
\begin{equation}\label{E:value-n}
\begin{sis}
 i<_{I_\bullet} j & \Longrightarrow \n_{V_{ij}}=-1,\\
  i>_{I_\bullet} j & \Longrightarrow \n_{V_{ij}}=0,\\
  i\sim_{I_\bullet} j & \Longrightarrow (\n_{V_{ij}},\n_{V_{ji}})=(-1,-1), (-1,0),(0,-1). 
\end{sis}
\end{equation}
Consider the binary relations
\begin{equation}\label{E:defleq}
\begin{sis}
& i\leq j\stackrel{\text{def}}{\Longleftrightarrow} \text{either } i=j \text{ or } \n_{V_{ij}}=-1,\\
& i\sim j\stackrel{\text{def}}{\Longleftrightarrow} i\leq j \text{ and } j\leq i \left(\stackrel{}{\Longleftrightarrow}\text{either } i=j \text{ or } \n_{V_{ij}}=\n_{V_{ji}}=-1\right),\\
&  i< j\stackrel{\text{def}}{\Longleftrightarrow} i\leq j \text{ and } i\not \sim j \left(\stackrel{}{\Longleftrightarrow} \n_{V_{ij}}=-1 \text{ and } \n_{V_{ji}}=0\right),\\
\end{sis}
\end{equation}
Property \eqref{E:value-n} implies that 
\begin{equation}\label{E:compa-<}
\begin{sis}
& i\sim j \Rightarrow i\sim_{I_\bullet} j,\\
& i<_{I_\bullet} j \Rightarrow i<j \Rightarrow i\leq j\Rightarrow i\leq_{I_\bullet} j. 
\end{sis}
\end{equation}
Since $\n$ is a V-stability condition of degree $0$ on $C_n$, it follows from \eqref{E:value-n} and \eqref{E:defleq} that 
$$
i\sim j \Longleftrightarrow i\leq j \text{ or } V_{ij} \text{ is $\n$-degenerate.}
$$
Therefore, from the properties of  V-stability conditions, it follows that $\sim$ is an equivalence relation. 

Let us now check that $\leq $ defines a linear preorder on $[n]$:
\begin{itemize}
    \item Reflexivity is obvious.
    \item Linearity, i.e. that either $i\leq j$ or $j\leq i$ for any $i\neq j\in [n]$, follows from the fact that either $\n_{V_{ij}}=-1$ or $\n_{V_{ji}}=-1$ by \eqref{E:value-n}. 
    \item Transitivity: if $i\leq j \leq k$ then $i\leq k$.

    Indeed, we can assume that $i\not\sim k$, for otherwise the result is obvious. We will assume that $\{i,j,k\}$ are cyclically ordered (the other case is similar and left to the reader). We will distinguish several cases: 
    \begin{itemize}
        \item If $i\sim j$ and $j\sim k$ then $i\sim k$ because $\sim$ is an equivalence relation, and we are done. 
         \item  If $i\sim j<k$ or $i<j\sim k$ then $i\not\sim k$ (because $\sim$ is an equivalence relation) and Definition~\ref{D:Vstab} implies that 
    $$
    1+\n_{V_{ki}}=\n_{V_{ij}}+\n_{V_{jk}}+\n_{V_{ki}}+3=1\Rightarrow
    \n_{V_{ki}}=0\Rightarrow \n_{V_{ik}}=-1\Rightarrow i< k.
    $$
     \item  If $i<j<k$ then either  $i\sim k$ and we are done, or $i\not \sim k$ in which case Definition \ref{D:Vstab} implies that
     $$
     \n_{V_{ki}}+1=\n_{V_{ij}}+\n_{V_{jk}}+\n_{V_{ki}}+3\in \{1,2\} \Rightarrow \n_{V_{ki}}=0
     \Rightarrow \n_{V_{ik}}=-1\Rightarrow i< k.
     $$
    \end{itemize}
\end{itemize}
The linear preorder $\leq$ correspond to a unique ordered partition $J_\bullet$ of $[n]$, i.e. $\leq=\leq_{J_\bullet}$. Property \eqref{E:compa-<} implies that 
$J_\bullet$ is a refinement of $I_\bullet$. 
Finally, comparing \eqref{E:value-n} and \eqref{E:defleq} with Definition \ref{D:nIf} for $J_\bullet$, we deduce that $\n=\n(J_\bullet)$, which concludes the proof. 
\end{proof}

    \bibliographystyle{alpha}	
    \bibliography{bibtex}
\end{document}